\documentclass{amsart}

\usepackage{lineno}
\usepackage{amssymb}
\usepackage{amsmath}
\usepackage{amsthm}
\usepackage{mathrsfs}
\usepackage{bbm}
\usepackage{cite}
\usepackage{color}
\usepackage{enumitem}
\usepackage{pgfplots}

\usepackage[colorlinks,hypertexnames=false]{hyperref}
\usepackage[msc-links]{amsrefs}

\usepackage{changes}

\usepackage{mathtools}
\usepackage{extarrows}
\usepackage{setspace}

\usepackage{graphicx}

\usepackage{imakeidx}
\makeindex

\hypersetup{
	colorlinks=true,
	linkcolor=black,
	filecolor=black,
	urlcolor=black,
	citecolor=black,
}

\newtheorem{thm}{Theorem}[section]

\newtheorem{defn}[thm]{Definition}
\newtheorem{lem}[thm]{Lemma}
\newtheorem{prop}[thm]{Proposition}

\newtheorem{rmk}[thm]{Remark}

\newtheorem{proposition}[thm]{Proposition}

\newtheorem{remark}[thm]{Remark}
\newtheorem{example}[thm]{Example}

\newcommand{\id}{\mathop{\mathrm{id}}}

\newcommand{\Conj}{\mathop{\mathrm{Conj}}}
\numberwithin{equation}{section}

\begin{document}

\title[Domains of existence of slice regular functions]{Domains of existence of slice regular functions in one quaternionic variable}
\author{Xinyuan Dou}
\email[Xinyuan Dou]{douxinyuan@ustc.edu.cn}
\address{Department of Mathematics, University of Science and Technology of China, Hefei 230026, China}
\address{Institute of Mathematics, AMSS, Chinese Academy of Sciences, Beijing 100190, China}
\author{Ming Jin}
\email[Ming Jin]{mjin@must.edu.mo}
\address{Faculty of Innovation Engineering, Macau University of Science and Technology, Macau, China}
\author{Guangbin Ren}
\email[Guangbin Ren]{rengb@ustc.edu.cn}
\address{Department of Mathematics, University of Science and Technology of China, Hefei 230026, China}
\author{Irene Sabadini}
\email[Irene Sabadini]{irene.sabadini@polimi.it}
\address{Dipartimento di Matematica, Politecnico di Milano, Via Bonardi, 9, 20133 Milano, Italy}

\thanks{This work was supported by the China Postdoctoral Science Foundation (BX20220327) and the NNSF of China (12171448). The fourth author, member of GNSAGA, INdAM, is partially supported by PRIN 2022 {\em Real and Complex Manifolds: Geometry and Holomorphic Dynamics}. The authors express their gratitude to Xiangyu Zhou, Xieping Wang, and Zhuo Liu for engaging in insightful discussions on the topic of several complex variables.  }


\begin{abstract}
	Recently, we introduced domains of slice regularity in the space  $\mathbb{H}$ of quaternions and also proved that domains of slice regularity satisfy a symmetry with respect to paths, called $2$-path-symmetry. In this paper, we give a full characterization by showing that all $2$-path-symmetric slice-open sets are domains of slice regularity. In fact, we will prove a counterpart of the Cartan-Thullen theorem for slice regular functions, namely that a slice-open set is a domain of existence for some slice regular function if and only if it is a domain of slice regularity, if and only if it is slice-regularly convex, if and only if it is $2$-path-symmetric. As a tool, we also prove an interpolation theorem of independent interest.
\end{abstract}

\maketitle

\section{Introduction}

The objective of this paper is to present a version of the Cartan-Thullen theorem adapted to the case of slice regular functions in one quaternionic variable. The Cartan-Thullen theorem \cite{Cartan1932001} is a fundamental result in the theory of several complex variables, providing a powerful characterization of domains of holomorphy in complex analysis. This theorem establishes an equivalence between holomorphically convex domains, domains of holomorphy and the domains of existence of holomorphic functions

In the context of one complex variable, all open sets are domains of holomorphy. However, in  several complex variables, Hartogs' phenomenon, see \cite{Hartogg1906001} or e.g. \cite{krantz}, demonstrates that not all open sets are necessarily domains of holomorphy and so the study of domains of holomorphy plays a crucial role in this context.

This paper aims to provide a topological equivalent characterization of the domains of existence of slice regular functions in one quaternionic variable on non-axially symmetric slice-open sets with respect to slice topology. In fact, we prove a Cartan-Thullen theorem and an interpolation theorem for such functions.
 The study is not a straightforward generalization of the case of one complex variable, although slice regular functions are one of the possible counterparts of holomorphic functions in the case of one quaternionic variable, and indeed we appeal more to techniques from the case of several complex variables.

Slice regular functions were introduced by Gentili and Struppa \cite{Gentili2007001} in case of one quaternionic variable, elaborating an idea of Cullen \cite{Cullen1965001} with the aim of defining a class of quaternionic "holomorphic" functions including polynomials and power series. Roughly speaking, slice regular functions are functions that are holomorphic with respect to imaginary unit $I$ when restricted to the complex plane $\mathbb{C}_I:=\mathbb{R}+\mathbb{R}I$. Slice regular functions were then extended to the case of functions defined on Clifford algebras \cite{Colombo2009002}, octonions \cite{Gentili2010001}, real alternative $*$-algebras \cite{Ghiloni2011001} and, more in general, to even-dimensional Euclidean spaces \cite{Dou2023002} and also to vector-valued functions, see \cite{Alpay2016001B}. This class of functions is particularly successful \cite{Alpay2013001,Colombo2010003,Stoppato2012001,Gentili2014001,Colombo2015001}, especially in view of the applications in operator theory \cite{Alpay2015001,Alpay2016001B, Gantner2020001,Colombo2011001,Colombo2011001B,Colombo2022001,MR3887616,MR3967697}.

For over a decade, the research on slice regular functions mainly focused on axially symmetric slice domains, where the representation formula shows that the values of a slice regular function are determined by the restriction to a complex plane (or restrictions to two different complex semi-planes).  Interested readers can refer to the books  \cite{Colombo2011001B,bookentire,Gentili2013001B,Alpay2016001B} and the references therein to learn more about slice regular functions.

Slice analysis on non-axially symmetric domains is systematically studied after the introduction of slice topology, see \cite{Dou2020001, Dou2021001} although it is also studied in \cite{GeSto2021, Gesto2021B}, but with a substantially different approach still relying on the Euclidean topology. In \cite{Dou2020001, Dou2021001} notions such as the path-representation formula and path-slice functions, are introduced to deal with non-axially symmetric domains.

To further develop the global theory of slice analysis, it is necessary to understand the properties of domains of slice regularity, an analogue of domains of holomorphy to this setting. Proposition 9.7 in \cite{Dou2020001} shows that all the domains of holomorphy satisfy a certain topology condition, called $2$-path-symmetry. In this paper we will prove the converse, i.e. all $2$-path-symmetric open sets in slice topology are domains of holomorphy. We will also  show that a open set in slice topology is a domain of holomorphy if and only if it is the domain of existence of some slice regular function, if and only if it satisfies certain convexity, called slice regular convexity. These equivalent conditions form the Cartan-Tullen theorem for slice regular functions.

As we already pointed out, the proof of Cartan-Thullen theorem for slice regular functions has some similarities to that one given in complex analysis in several variables, for example the use of an interpolation theorem, but also relevant differences and it can be summarized as follows. We firstly prove that for any $2$-path-symmetric slice-open set $\Omega$, there is $I\in\mathbb{S}$ such that the cardinality of $J\in\mathbb{S}$ such that $\Psi_J^I(\Omega_J)\neq\Omega_I$, where $\Psi_J^I:x+yJ\mapsto x+yI$, is countable, see Proposition \ref{pr-osmba}. This allows us to construct a Riemann domain $\breve{\Omega}$ over $\mathbb{C}$, see Theorem \ref{thm-Riemann doamin}. Then we use some holomorphic maps on certain subsets $\breve{\Omega}^{\scriptscriptstyle[I]}$, $I\in\mathbb{S}$ of the Riemann domain $\breve{\Omega}$, to construct some suitable slice regular functions, see Proposition \ref{prop-construct slice regular functions}. By the properties of these  holomorphic functions, we deduce that the corresponding slice regular functions have the good properties in Lemma \ref{pr-loitkso}. This result implies that there is a slice regular function such that its absolute value is smaller than $\epsilon>0$ on an invariant subset $K$ of $\Omega$ (an analogue of compact set $K$ with $K=\hat{K}_\Omega$ in complex analysis) and $f=1$ at a given point $q\in\Omega\backslash K$. Using these ``good'' slice regular functions, we prove an interpolation theorem, see Theorem \ref{th-it} useful in the proof part of the Cartan-Thullen theorem in this framework, see Section \ref{sec-Cartan-Thullen}.

It is important to point out that to prove the full statement of the Cartan-Thullen theorem, we need some notions that are specific to the quaternionic case, and deeply differ from those in the complex analysis case.

We provide here a brief account of some of these differences to show that the results in this paper are far from being trivial or easy. We consider a continuous function $r^+$ on $\breve{\Omega}$ (see Proposition \ref{pr-r+}) corresponding to the $\sigma$-distance in $\Omega$ (see \eqref{eq-cvsrc}). The function $r^+$ plays a similar role as the distance between the compact set and the boundary of the domain of holomorphy in complex analysis (see the proof of Proposition \ref{pr-increasing subsets} and Proposition \ref{eq-oits}). Since a Riemann domain $\breve{\Omega}$ is axially symmetric, we define conjugation on $\breve{\Omega}$ (see \eqref{eq-conjugation}); note that the values of $r^+$ on conjugated points coincide (see Proposition \ref{pr-rrcb}). We then introduce stem discrete sets (see Definition \ref{def-stem discrete}) which, in the proof of the interpolation theorem,  play a role similar to discrete points in complex analysis. Note that stem discrete points may not be discrete in $\Omega$, and their corresponding set may also not be discrete in $\breve{\Omega}$. They are discrete in the sense of $\sigma$-distance, that is, the corresponding set $\psi^\mathbb{S}_\Omega(A\cap\Omega^r)$ in $\breve{\Omega}$ of intersection of a stem discrete set $A$ and $\Omega^r=\{q\in\Omega:r\le \sigma(q,\mathbb{H}\backslash\Omega)\}$ is discrete for each $r>0$. An $\Omega$-invariant set (see Definition \ref{def-invariant}) plays a role similar to a compact set $K$ with $\widehat{K}_{\Omega}=K$. Moreover, a constant $C_x$ depending on the point $x\in\Omega$ appears in the definition of slice regular convex hull $\widehat{K}_{\Omega}$ of $K$ (see Definition \ref{def-slice regularly convex hull}), but this constant does not have a counterpart in complex analysis.

The structure of the paper is as follows.
In section \ref{sec-main results} we introduce the basic notions useful to state and prove the main results of this paper.
In section \ref{sec-2ps slice domains} we give some properties of 2-path-symmetric st-domains, among which the crucial property of countability of distinct slices.
In section \ref{sec-slice piecewise paths} we introduce slice-piecewise paths to describe path-connectedness of
st-domains.
In section \ref{sec-2ps Riemann domain} we introduce a notion of $2$-path-symmetric Riemann domain associated with a $2$-path-symmetric st-domain, via a suitably defined equivalence relation.
In Section \ref{sec-properties of 2ps Riemann domains} we study the correspondence between 2-path-symmetric st-domains of $ \mathbb H $ and the associated 2-path-symmetric Riemann domains over $ \mathbb C $ by proving a number of properties.
In Section \ref{Stem discrete sets} we introduce stem discrete sets which serve as prescribed points for interpolation theorem in the case of slice analysis.
In Section \ref{sec-interpolation}, we prove a couple of extension results of functions and an interpolation theorem.
Finally, in Section \ref{sec-Cartan-Thullen} we summarize what we prepared in the preceding sections and we  prove the Cartan-Thullen theorem in this framework.
At the end, we added a list of symbols an notions to help the reader to navigate the paper.

\section{Main results}\label{sec-main results}

In this section, we state our main results.  To this end, we recall some definitions and results  from \cite{Dou2020001}.  Denote by
\begin{equation*}
	\mathbb{H}:=\mathbb{R}+i\mathbb{R}+j\mathbb{R}+k\mathbb{R}
\end{equation*}
the quaternion number system where
\begin{equation*}
	i^2=j^2=k^2=ijk=-1.
\end{equation*}
The set of imaginary units of quaternions $\mathbb{H}$ is denoted by
\begin{equation*}\index{$\mathbb{S}$}
	\mathbb{S}:=\{I\in\mathbb{H}:I^2=-1\},
\end{equation*}
and let
\begin{equation*}\index{$\mathbb{S}^2_*$}
	\mathbb{S}^2_*:=\{(I,J)\in\mathbb{S}^2:I\neq J\}.
\end{equation*}

Below, we will explore the extension of slice regularity along paths, akin to what is done in the holomorphic setting. If the slice domain $\Omega$ in $\mathbb{H}$ does not intersect the real axis, then, by the definition of the slice topology, the theory mirrors that of the holomorphic setting. However, if $\Omega$ intersects the real axis, we encounter various challenges and, consequently, below we consider continuous paths starting from points lying on the real axis.

Let $U\subset\mathbb{C}$ and set
\begin{equation*}\index{$\mathscr{P}_0(U)$}
	\mathscr{P}_0(U):=\{\delta:[0,1]\rightarrow U:\delta\mbox{ is continuous and }\delta(0)\in\mathbb{R}\}.
\end{equation*}

{\begin{defn}\label{def-2ps}
	$\Omega\subset\mathbb{H}$ is called \textit{\textbf{$2$-path-symmetric}}\index{$2$-path-symmetric} if
	\begin{equation*}
		\gamma^K\subset\Omega,
	\end{equation*}
	for any $K\in\mathbb{S}$, $\gamma\in\mathscr{P}_0(\mathbb{C})$ and $(I,J)\in\mathbb{S}^2_*$ with $\gamma^I,\gamma^J\subset\Omega$, where $\gamma^I:=\Psi_i^I\circ\gamma$ \index{$\gamma^I$},
	\begin{equation*}\index{$\Psi_i^I$}
		\begin{split}
			\Psi_i^I:\quad\mathbb{C}\quad &\xlongrightarrow[\hskip1cm]{}\quad \mathbb{C}_I,
			\\ x+yi\ &\shortmid\!\xlongrightarrow[\hskip1cm]{}\ x+yI,
		\end{split}
	\end{equation*}
	and $\gamma^K\subset\Omega$ is short for $\gamma^K([0,1])\subseteq\Omega$.
\end{defn}}

The slice topology \index{Slice topology} on $\mathbb{H}$ is defined by
\begin{equation*}\index{$\tau_s$}
	\tau_s(\mathbb{H}):=\big\{U\subset\mathbb{H}:U_I\in\tau(\mathbb{C}_I), \forall\ I\in\mathbb{S}\big\},
\end{equation*}
where $\tau(\mathbb{C}_I)$ the Euclidean topology on the complex plane $\mathbb{C}_I:=\mathbb{R}+\mathbb{R}I$, and $U_I:= U\cap\mathbb{C}_I$.

{\bf Conventions}:\index{Slice-connected, slice-open, st-domain} We call $U\subset\mathbb{H}$ is slice-connected, if $U$ is connected in $\tau_s(\mathbb{H})$; we call $U\subset\mathbb{H}$ is slice-open, if $U$ is open in $\tau_s(\mathbb{H})$ and so on. A domain in $\tau_s(\mathbb{H})$ will be called st-domain since the terminology `slice domain' is used in the literature for a different concept. Note also that the symbol $\subset$ denotes the inclusion of sets, allowing equality.

Let $z\in\mathbb{C}$, $U\subset\mathbb{C}$, $I\in\mathbb{S}$ and $\mathbb{S}'\subset\mathbb{S}$. Denote
\begin{equation*}\index{$z^I$, $z^{\mathbb{S}'}$, $U^I$, $U^{\mathbb{S}'}$}
	z^I:=\Psi_i^I(z),\qquad \mbox{and}\qquad\qquad U^{\mathbb{S}'}:=\left\{w^J\in\mathbb{H}:w\in U,\ \mbox{and}\ J\in\mathbb{S}'\right\}.
\end{equation*}
Moreover, we write $z^{\mathbb{S}'}:=\{z\}^{\mathbb{S}'}$ and $U^I:=U^{\{I\}}$ for short.

\begin{example}\label{exa-e1}
	Let $I,J\in\mathbb{S}$ with $I\neq\pm J$.
Set
\begin{equation*}
		U_{a,b}:=\{\lambda e^{i\theta}+3i\in\mathbb{C}:\ 2<\lambda<4,\ a<\theta<b\}.
	\end{equation*}
 Then
	\begin{equation*}
		E_1:=\left(U_{-\frac{\pi}{2},\frac{3\pi}{4}}\right)^I
		\bigcup \left(U_{\frac{\pi}{4},\frac{3\pi}{2}}\right)^J
		\bigcup \left[U_{0,3\pi}\cap{\Conj}_{\mathbb{C}}\left(U_{0,3\pi}\right)\right]^\mathbb{S}
	\end{equation*}
	is a $2$-path-symmetric st-domain, where $\Conj_\mathbb{C}:\mathbb{C}\rightarrow\mathbb{C},\ z\mapsto\overline{z}$.
	Note that $U_{0,3\pi}\cap{\Conj}_{\mathbb{C}}\left(U_{0,3\pi}\right)$ is the connected component of $U_{-\frac{\pi}{2},\frac{3\pi}{4}}\cap U_{\frac{\pi}{4},\frac{3\pi}{2}}$ which intersects the real axis. The intersection of $E_1$ with the complex planes $\mathbb C_I$, $\mathbb C_J$, $\mathbb C_K$, $K\not=\pm I, \pm J$ is shown in Figure 1.

\begin{figure}[h]\includegraphics[width=12.5cm]{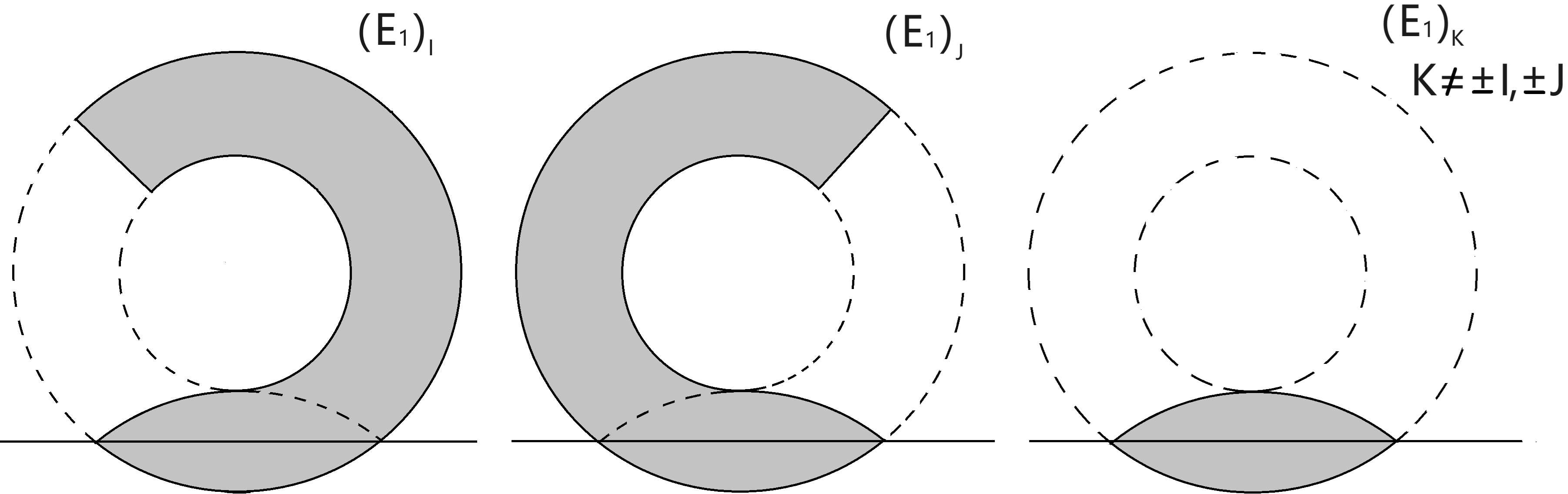}\caption{}\end{figure}
\end{example}

\begin{example}\label{exa-e2}
	Let $U$ be an open set in $\mathbb{R}$ with infinite connected components. Then there are $a_1,a_2,...,b_1,b_2,...\in\mathbb{R}$ such that
	\begin{equation*}
		U=\bigcup_{\ell=1}^{+\infty}(a_\ell,b_\ell)
		\qquad\mbox{with}\qquad (a_\imath,b_\imath)\cap(a_\jmath,b_\jmath)=\varnothing,\qquad\forall\ \imath\neq\jmath.
	\end{equation*}
	Let $I_1,I_2,...\in\mathbb{S}$ with $I_\imath\neq \pm I_\jmath$ for each $\imath\neq\jmath$. Then
	\begin{equation*}
		E_2:=\bigcup_{\ell=1}^{+\infty} \left[(V_\ell)^\mathbb{S}\cup (W_\ell)^{I_\ell}\right]
	\end{equation*}
	is a $2$-path-symmetric st-domain, where
	\begin{equation*}
		\begin{cases}
			V_\ell:=\left\{x+yi\in\mathbb{C}:\ a_\ell<x<b_\ell,\ |y|<1\right\},
			\\W_\ell:=W_\ell^1\cup W_\ell^2\cup W_\ell^3,
			\\W_\ell^1:=\left\{x+yi\in\mathbb{C}: \left(x-\frac{a_\ell+b_\ell}{2}\right)\big(x-\frac{a_{\ell+1}+b_{\ell+1}}{2}\big)<0,\ 2<y<3\right\},
			\\W_\ell^2:=\left\{x+yi\in\mathbb{C}:\frac{a_\ell+b_\ell}{2}<x<b_\ell,\ 0<y<3\right\},
			\\W_\ell^3:=\left\{x+yi\in\mathbb{C}:a_{\ell+1}<x<\frac{a_{\ell+1}+b_{\ell+1}}{2},\ 0<y<3\right\}.
		\end{cases}
	\end{equation*}
See Figure 2.
\begin{figure}[h]\includegraphics[width=12.5cm]{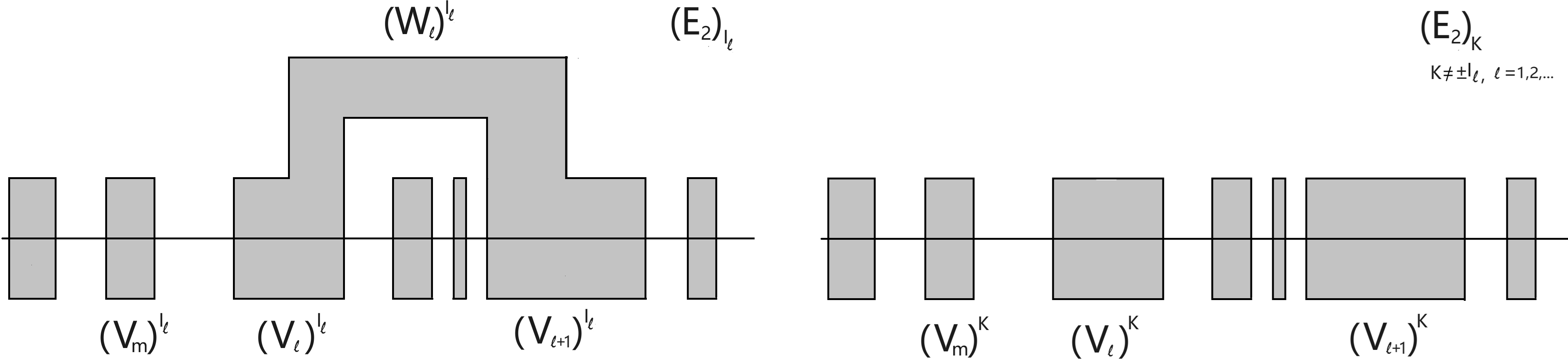}\caption{}\end{figure}
\end{example}

\begin{defn}
	Let $I\in\mathbb{S}$ and $U\in\tau(\mathbb{C}_I)$. A function $f:U\rightarrow\mathbb{H}$ is called \textit{\textbf{($I$-)holomorphic}}\index{($I$-)holomorphic}, if $f$ is real differentiable on $U\cap\mathbb{C}_I$ and
	\begin{equation*}
		\frac{1}{2}\left(\frac{\partial}{\partial x}+I\frac{\partial}{\partial y}\right)f(x+yI)=0,\qquad\forall\  x+yI\in U.
	\end{equation*}
	Denote by $\mathscr{O}_I(U)$\index{$\mathscr{O}_I(U)$} the class of $I$-holomorphic functions defined on $U$.
\end{defn}

\begin{defn}\label{df-sr}
	Let $\Omega\in\tau_s(\mathbb{H})$. A function $f:\Omega\rightarrow\mathbb{H}$ is called \textit{\textbf{slice regular}}\index{Slice regular}, if for each $I\in\mathbb{S}$, $f_I:=f|_{\Omega_I}$ is $I$-holomorphic.
	Denote by $\mathcal{SR}(\Omega)$\index{$\mathcal{SR}(\Omega)$} the class of slice regular functions defined on $\Omega$.
\end{defn}

\begin{defn}
	Let $\Omega\in\tau_s(\mathbb{H})$ and $\mathcal{F}\subset\mathcal{SR}(\Omega)$. $\Omega$ is called an \textit{\textbf{$\mathcal{F}$-domain of existence}}\index{$\mathcal{F}$-domain of slice existence} if there are no $\Omega_1,\Omega_2\in\tau_s(\mathbb{H})$ with the following properties:
	\begin{enumerate}[label=(\roman*)]
		\item $\varnothing\neq \Omega_1\subset\Omega_2\cap\Omega$.
		\item $\Omega_2$ is $\tau_s$-connected and $\Omega_2\nsubseteq \Omega$.
		\item For each $f\in\mathcal{F}$, there is $\widetilde{f}\in\mathcal{SR}(\Omega_2)$ such that $f=\widetilde{f}$ in $\Omega_1$.
	\end{enumerate}
	
	If there are slice-open sets $\Omega$, $\Omega_1$, $\Omega$ satisfying (i)-(iii), then we call $(\Omega,\Omega_1,\Omega_2)$ an \textit{\textbf{$\mathcal{F}$-triple}} \index{Slice-triple}. Moreover, we call $(\Omega,\Omega_1,\Omega_2)$ \textit{\textbf{special}}\index{Special slice-triple}, if $\Omega_1$ is a slice-connected component of $\Omega\cap\Omega_2$.
	
	$\Omega$ is called a \textit{\textbf{domain of slice regularity}}\index{Domain of slice regularity} if it is an $\mathcal{SR}(\Omega)$-domain of existence and $\Omega$ is called the \textit{\textbf{domain of existence}}\index{Domain of existence} of $f\in\mathcal{SR}(\Omega)$ if $\Omega$ is $\{f\}$-domain of existence.
\end{defn}

We now extend holomorphic convexity to the theory of slice regular functions:
\begin{defn}\label{def-slice regularly convex hull}
	Let $\Omega\in\tau_s(\mathbb{H})$ and $K\subset\Omega$. We call
	\begin{equation*}\index{$\widehat{K}_\Omega$}
		\widehat{K}_\Omega:=\left\{x\in\Omega:\exists\ C_x>0,\ s.t.\ |f(x)|\le C_x\sup_{y\in K} |f(y)|,\ \forall\ f\in\mathcal{SR}(\Omega)\right\},
	\end{equation*}
	the \textit{\textbf{slice regularly convex hull}}\index{Slice regularly convex hull} of $K$.
\end{defn}

For any $\Omega\subset\mathbb{H}$, denote
\begin{equation}\label{eq-o2ps1}\index{$\Omega_{2ps}$}
	\begin{split}
		\Omega_{2ps}:=\big\{&z\in\mathbb{C}:\exists\ \gamma\in\mathscr{P}_0(\mathbb{C})\mbox{ and }(I,J)\in\mathbb{S}^2_*,\ \\&s.t.\ \gamma^I,\gamma^J\subset\Omega \mbox{ and }\gamma(1)=z \big\},
	\end{split}
\end{equation}
and
\begin{equation}\label{eq-s omega 2ps}\index{$\mathbb{S}_\Omega^{2ps}$}
	\mathbb{S}_\Omega^{2ps}:=\{I\in\mathbb{S}:\Omega_I\neq (\Omega_{2ps})^I\}.
\end{equation}

\begin{defn} $K\subset\mathbb{H}$ is called \textit{\textbf{finite-non-symmetric compact}},\index{Finite-non-symmetric compact}  if $K_I$ is compact in $\tau(\mathbb{C}_I)$ and $|\mathbb{S}_{K}^{2ps}|<+\infty$.
\end{defn}

\begin{defn}
$\Omega\in\tau_s(\mathbb{H})$ is called \textit{\textbf{slice regularly convex}}\index{Slice regularly convex}, if $$(\widehat{K}_\Omega)_I\Subset\Omega_I,$$ for each finite-non-symmetric compact set $K\subset\Omega$ and $I\in\mathbb{S}$.
\end{defn}

The main result is as follows:

\begin{thm} (Cartan-Thullen)\label{thm-ct}
	Let $\Omega\in\tau_s(\mathbb{H})$. Then the following statements are equivalent:
	\begin{enumerate}[label=(\roman*)]
		\item $\Omega$ is the domain of existence of some slice regular function.
		\item $\Omega$ is a domain of slice regularity.
		\item $\Omega$ is slice regularly convex.
		\item $\Omega$ is $2$-path-symmetric.
	\end{enumerate}
\end{thm}

\section{$2$-path-symmetric st-domains}\label{sec-2ps slice domains}

In this section, we prove various properties of $2$-path-symmetric st-domains. In particular, next Proposition \ref{pr-osmba} shows that $2$-path-symmetric st-domains have at most countable distinct slices, where distinct is meant in the sense of some symmetry. To begin with, we define a binary relation which is an equivalence relation on $2$-path-symmetric open sets.

\begin{defn}
Define a binary relation $\approx_{\scriptscriptstyle \Omega}$ \index{$\approx_{\scriptscriptstyle \Omega}$} on $\Omega\subset\mathbb{H}$: Let $p,q\in\Omega$. Then $p\approx_{\scriptscriptstyle \Omega} q$ if and only if there is $\gamma\in\mathscr{P}_0(\mathbb{C})$ and $I,J\in\mathbb{S}$ such that
\begin{equation}\label{eq-gigj}
	\gamma^I,\gamma^J\subset\Omega,\qquad\gamma^I(1)=p\qquad\mbox{and}\qquad\gamma^J(1)=q.
\end{equation}
Let us set
\begin{equation*}
	[p]_{\approx_\Omega}:=\{w\in\Omega: w \approx_{\scriptscriptstyle \Omega} p\}.
\end{equation*}

\end{defn}

\begin{prop}\label{pr-losmt}
	Let $\Omega\subset\mathbb{H}$. Then the following statements holds:
	\begin{enumerate}[label=(\roman*)]
		\item\label{it-lp} Let $p,z^L\in\Omega$, with $z\in\mathbb{C}$, $L\in\mathbb S$ and $p\approx_{\scriptscriptstyle \Omega} z^L$. Then $p=z^K$ for some $K\in\mathbb{S}$.
		\item If $\Omega$ is $2$-path-symmetric, then $\approx_{\scriptscriptstyle \Omega}$ is an equivalence relation on $\Omega$.
		\item\label{lzz} Let $z^I,z^J\in\Omega$ with $z^I\approx_{\scriptscriptstyle \Omega} z^J$. Then there is $\gamma\in\mathscr{P}_0(\mathbb{C})$ such that
		\begin{equation*}
			\gamma(1)=z\qquad\mbox{and}\qquad \gamma^I,\gamma^J\subset\Omega.
		\end{equation*}
		\item\label{it-lzii} Let $z^{I_\imath}\in\Omega$ with $z^{I_\imath}\in [z^{I_1}]_{\approx_\Omega}$, $\imath=1,...,k$. If $\Omega$ is $2$-path-symmetric, then there is $\gamma\in\mathscr{P}_0(\mathbb{C})$ such that
		\begin{equation*}
			\gamma(1)=z,\qquad\mbox{and}\qquad \gamma^{I_\imath}\subset\Omega,\qquad \imath=1,...,k.
		\end{equation*}
		\item\label{it-lzio} Let $z^I\in\Omega$ with $|[z^I]_{\approx_\Omega}|>1$. If $\Omega$ is $2$-path-symmetric, then $[z^I]_{\approx_\Omega}=z^{\mathbb{S}}$.
	\end{enumerate}
\end{prop}

\begin{proof}
	(i) Since $p \approx_{\scriptscriptstyle \Omega} z^L$, there is $\gamma\in\mathscr{P}_0(\mathbb{C})$ and $I,J\in\mathbb{S}$ such that \eqref{eq-gigj} holds. Then $\gamma(1)=\Conj_{\mathbb{C}}^\lambda(z)$ for some $\lambda\in\{0,1\}$, where $\Conj_\mathbb{C}:\mathbb{C}\rightarrow\mathbb{C},\ z\mapsto\overline{z}$ and $\Conj_\mathbb{C}^\lambda:=(\Conj_\mathbb{C})^\lambda$ (i.e. we set $\Conj_\mathbb{C}^0$=Identity and $\Conj_\mathbb{C}^1=\Conj_\mathbb{C}$). By \eqref{eq-gigj},
	\begin{equation*}
		q=\gamma^J(1)=\left[{\Conj}_\mathbb{C}^\lambda(z)\right]^J=z^{(-1)^\lambda J}=z^K,
	\end{equation*}
	where $K:=(-1)^\lambda L\in\mathbb{S}$.
	
	(ii) The reflexivity and symmetry of $\approx_{\scriptscriptstyle \Omega}$ hold directly by definition. We only need to prove the transitivity. Let $u,v=x+yL,w\in\Omega$ with $u \approx_{\scriptscriptstyle \Omega} v$ and $v \approx_{\scriptscriptstyle \Omega} w$. By hypothesis, there is $\gamma\in\mathscr{P}_0(\mathbb{C})$ and $I,J\in\mathbb{S}$ such that
	\begin{equation*}
		\gamma^I,\gamma^J\subset\Omega,\qquad \gamma^I(1)=u,\qquad\mbox{and}\qquad \gamma^J=v.
	\end{equation*}
	According to (i) and since $v \approx_{\scriptscriptstyle \Omega} w$, we have that $w=x+y K=\gamma^K(1)$ for some $K\in\mathbb{S}$. It implies that $u\approx_{\scriptscriptstyle \Omega} w$.
	
	(iii) If $z\in\mathbb{R}$, then \ref{lzz} holds by taking $\gamma\equiv z$. Otherwise, $z\notin\mathbb{R}$. By definition, there is $\beta\in\mathscr{P}_0(\mathbb{C})$ and $I',J'\in\mathbb{S}$ such that $\beta^{I'},\beta^{J'}\subset\Omega$,
	\begin{equation*}
		\beta^{I'}(1)=z^I,
		\qquad\mbox{and}\qquad
		\beta^{J'}(1)=z^J.
	\end{equation*}
	It follows that $\beta(1)=\Conj_{\mathbb{C}}^\lambda(z)$ for some $\lambda\in\{0,1\}$. Then
	\begin{equation*}
		z^I=\beta^{I'}(1)=\left[{\Conj}_{\mathbb{C}}^\lambda(z)\right]^{I'}=z^{(-1)^\lambda I'}.
	\end{equation*}
	It implies that $I'=(-1)^\lambda I$ and $J'=(-1)^\lambda I$. Take $\gamma:=\Conj_{\mathbb{C}}^\lambda(\beta)$. Then
	\begin{equation*}
		\gamma(1)
		=\left[{\Conj}_{\mathbb{C}}^\lambda(\beta)\right](1)
		={\Conj}_{\mathbb{C}}^\lambda\left(\beta(1)\right)
		={\Conj}_{\mathbb{C}}^\lambda\left({\Conj}_{\mathbb{C}}^\lambda(z)\right)=z,
	\end{equation*}
	and
	\begin{equation*}
		\gamma^I=\left[{\Conj}_\mathbb{C}^\lambda(\beta)\right]^I=\beta^{(-1)^\lambda I}=\beta^{I'}\subset\Omega.
	\end{equation*}
	Similarly, $\gamma^J\subset\Omega$. It implies that \ref{lzz} holds.
	
	(iv) If $I_1=\cdots I_k$, then \ref{it-lzii} holds by taking $\gamma\equiv z$. Otherwise, without loss of generality, we suppose $I_1\neq I_2$. By \ref{lzz}, there is $\gamma\in\mathscr{P}_0(\mathbb{C})$ such that $\gamma(1)=z$ and $\gamma^{I_1},\gamma^{I_2}\subset\Omega$. Since $\Omega$ is $2$-path-symmetric, $\gamma^{J}\subset\Omega$ for each $J\in\mathbb{S}\supset\{I_3,...,I_k\}$. It implies that \ref{it-lzii} holds.
	
	(v) It follows from \ref{it-lp} that $[z^I]_{\approx_\Omega}\subset z^\mathbb{S}$. We only need to prove that
	\begin{equation}\label{eq-zms}
		z^\mathbb{S}\subset [z^I]_{\approx_\Omega}.
	\end{equation}Since $|[z^I]_{\approx_\Omega}|>1$, there is $w^L\approx_{\scriptscriptstyle \Omega} q$ for some $w\in\mathbb{C}$ and $L\in\mathbb{S}$ with $w^L\neq q$. By \ref{it-lp}, there is $J\in\mathbb{S}\backslash\{I\}$ such that $w^L=z^J$. According to \ref{it-lzii}, there is $\gamma\in\mathscr{P}_0(\mathbb{C})$ such that $\gamma(1)=z$ and $\gamma^I,\gamma^J\subset\Omega$. Since $\Omega$ is $2$-path-symmetric, we have $\gamma^K\subset\Omega$ for each $K\in\mathbb{S}$. It implies that $z^K=\gamma^K(1)\in[z^I]_{\approx_\Omega}$ for each $K\in\mathbb{S}$. Therefore \eqref{eq-zms} holds.
\end{proof}
\begin{remark}
Let $\Omega\subset\mathbb{H}$. It is easy to check by Proposition \ref{pr-losmt} \ref{lzz} that
\begin{equation}\label{eq-o2ps}
	\Omega_{2ps}=\left\{z\in\mathbb{C}:z^I,z^J\in\Omega\mbox{ with }z^I\approx_{\scriptscriptstyle \Omega} z^J,\mbox{ for some }(I,J)\in\mathbb{S}_*^2\right\}.
\end{equation}
Note also that, by definition,
\begin{equation}\label{eq-omr}\index{$\Omega_\mathbb{R}$}
\Omega_\mathbb{R}:=\Omega\cap\mathbb R\subset\Omega_{2ps}.
\end{equation}
Moreover, if $\Omega_\mathbb{R}=\varnothing$, then $\Omega_{2ps}=\varnothing$.
\end{remark}
Denote by
\begin{equation*}\index{$\mathcal{L}_z^w$}
	\begin{split}
		\mathcal{L}_z^w:\quad[0,1]\quad &\xlongrightarrow[\hskip1cm]{}\qquad \mathbb{C},
		\\ t\quad\ \ &\shortmid\!\xlongrightarrow[\hskip1cm]{}\ tw+(1-t)z.
	\end{split}
\end{equation*}
the segment from $z\in\mathbb{C}$ to $w\in\mathbb{C}$.

\begin{prop}\label{pr-2psopen}
	Let $\Omega\in\tau_s(\mathbb{H})$ be $2$-path-symmetric. Then $\Omega_{2ps}\in\tau(\mathbb{C})$.
\end{prop}

\begin{proof}
	Let $z\in\Omega_{2ps}$. By definition, there is a $\gamma\in\mathscr{P}_0(\mathbb{C})$ and $(I,J)\in\mathbb{S}_*^2$ such that $\gamma^I,\gamma^J\subset\Omega$ and $\gamma(1)=z$. It implies that $z^I\in\Omega_I$ and $z^J\in\Omega_J$. Therefore there is $r>0$ such that
	\begin{equation*}
		\left[B_\mathbb{C}(z,r)\right]^I,\left[B_\mathbb{C}(z,r)\right]^J\subset\Omega,
	\end{equation*}
	where $\left[B_\mathbb{C}(z,r)\right]^I$ is the disc in the plane $\mathbb C_I$ with center at $z^I$ and with radius $r>0$.
	For each $w\in B_{\mathbb{C}}(z,r)$, it is easy to check that
	\begin{equation*}
		\gamma\circ\mathcal{L}_z^w\in\mathscr{P}_0(\mathbb{C}),\quad (\gamma\circ\mathcal{L}_z^w)^I\subset\Omega_I,\quad
		(\gamma\circ\mathcal{L}_z^w)^J\subset\Omega_J
		\quad\mbox{and}\quad\gamma\circ\mathcal{L}_z^w(1)=w.
	\end{equation*}
	It implies by definition that $w\in\Omega_{2ps}$. Then $B_{\mathbb{C}}(z,r)\subset\Omega_{2ps}$, and $z$ is an interior point of $\Omega_{2ps}$. Therefore $\Omega_{2ps}$ is open.
\end{proof}

\begin{prop}
	Let $\Omega\in\tau_s(\mathbb{H})$ be $2$-path-symmetric. Then
	\begin{equation}\label{eq-2ps}
		\partial(\Omega_{2ps})\bigcap\Psi_I^i(\Omega_I)\bigcap\Psi_J^i(\Omega_J)=\varnothing, \qquad\forall\ (I,J)\in\mathbb{S}^2_*,
	\end{equation}
	where $\Psi_I^i:=\left(\Psi_i^I\right)^{-1}:\ \Omega_I\to\Omega$, $\Psi_I^i(x+Iy)=x+iy$.
\end{prop}

\begin{proof}
	If $\Omega_\mathbb{R}=\varnothing$, then $\Omega_{2ps}=\varnothing$, so is $\partial(\Omega_{2ps})$. It implies that \eqref{eq-2ps} holds.
	
	Otherwise $\Omega_\mathbb{R}\neq\varnothing$. Suppose that \eqref{eq-2ps} does not hold. Then there is a point
	\begin{equation*}
		z\in\partial(\Omega_{2ps})\bigcap\Psi_I^i(\Omega_I)\bigcap\Psi_J^i(\Omega_J).
	\end{equation*}
	Since $\Omega\in\tau_s(\mathbb{H})$, by definition $\Omega_I$ is open in $\mathbb{C}_I$ and then $\Psi_I^i(\Omega_I)$ is open in $\mathbb{C}$. Thence there is $r>0$ such that the open ball $\mathbb{B}:=B_{\mathbb{C}}(z,r)$ in $\mathbb{C}$ is contained in $\Psi_I^i(\Omega_I)\bigcap\Psi_J^i(\Omega_J)$.
	
	It follows from $z\in\mathbb{B}\cap\partial(\Omega_{2ps})$ that $\mathbb{B}\cap\Omega_{2ps}\neq\varnothing$. Let $u\in\mathbb{B}\cap\Omega_{2ps}$. By \eqref{eq-o2ps1}, there is $\alpha\in\mathscr{P}_0(\mathbb{C})$, such that $\alpha^I,\alpha^J\subset\Omega$ and $\alpha(1)=u$. Then by
	\begin{equation*}
			\left(\mathcal{L}_u^z\right)^K\subset\mathbb{B}^K\subset\left(\Psi_K^i(\Omega)\right)^K=\Omega_K,\qquad K=I,J,
	\end{equation*}
	we have
	\begin{equation*}
		\left(\alpha\circ\mathcal{L}_u^z\right)^K=\alpha^K\circ	\left(\mathcal{L}_u^z\right)^K\subset\Omega_K,\qquad K=I,J.
	\end{equation*}
	It implies that
	\begin{equation*}
		z=\left(\alpha\circ\mathcal{L}_u^z\right)(1)\in \Omega_{2ps}
	\end{equation*}
	which contradicts $z\in\partial(\Omega_{2ps})$.
\end{proof}

\begin{prop}\label{pr-los}
	Let $\Omega\subset\mathbb{H}$ be an st-domain with $\Omega_\mathbb{R}:=\Omega\cap\mathbb{R}\not=\varnothing$. Then each connected component of $\Omega_I$ in topological space $\mathbb{C}_I$, $I\in\mathbb{S}$, intersects $\mathbb{R}$.
\end{prop}

\begin{proof}
	Proof by contradiction. If $U$ is a connected component of $\Omega_I$ for some $I\in\mathbb{S}$ with $U_\mathbb{R}\neq\varnothing$, then it is easy to check that $U$ and $\Omega\backslash U$ are two disjoint nonempty slice-open sets. It implies that $\Omega$ is not slice-connected, a contradiction.
\end{proof}

\begin{rmk}\label{rm-2pssd}
	Let $\Omega\subset\mathbb{H}$ be an st-domain with $\Omega_\mathbb{R}=\varnothing$. According to \cite[Proposition 3.8]{Dou2020001}, $\Omega\subset\mathbb{C}_I$ for some $I\in\mathbb{S}$. It follows by connectedness that $\Omega$ is a domain in the topological space $\mathbb{C}_I$.
\end{rmk}

\begin{prop}\label{pr-osmba}
	Let $\Omega\subset \mathbb{H}$ be a non-empty $2$-path-symmetric st-domain. Then $\mathbb{S}_\Omega^{2ps}$ is countable.
\end{prop}

\begin{proof}
	(i) Suppose that $\Omega_\mathbb{R}=\varnothing$. By Remark \ref{rm-2pssd}, $\Omega$ is a domain $\mathbb{C}_I$ for some $I\in\mathbb{S}$. Therefore $\Omega_J=\varnothing$ for each $J\neq\pm I$. Recalling that, by \eqref{eq-s omega 2ps}, $\mathbb{S}_\Omega^{2ps}=\{I\in\mathbb{S}:\Omega_I\neq (\Omega_{2ps})^I\}$,
it follows from $\Omega_{2ps}=\varnothing$ that $\mathbb{S}_\Omega^{2ps}=\{\pm I\}$ is countable.
	
	(ii) Suppose that $\Omega_\mathbb{R}\neq\varnothing$. Let $\mathcal{B}$ be a countable topological bases of $\mathbb{C}$. Let $I\in \mathbb{S}_\Omega^{2ps}$. By definition $\Omega_I\neq (\Omega_{2ps})^I$. Let $U$ be a connected component of $\Omega_I$ in $\mathbb{C}_I$ with $U\nsubseteq (\Omega_{2ps})^I$. By Proposition \ref{pr-los}, $U_\mathbb{R}\neq\varnothing$. It follows by $U_\mathbb{R}\subset\Omega_\mathbb{R}\subset\Omega_{2ps}$ that $U\cap (\Omega_{2ps})^I\supset U_\mathbb{R}\neq\varnothing$.
	
	If $U\cap\partial_I\left[(\Omega_{2ps})^I\right]=\varnothing$, then $U$ is the union of two disjoint nonempty open set $U\backslash (\Omega_{2ps})^I$ and $U\cap (\Omega_{2ps})^I$, where $\partial_I\left[(\Omega_{2ps})^I\right]$ is the boundary of $(\Omega_{2ps})^I$ in $\mathbb{C}_I$. So $U$ is not connected, a contradiction.
	
	Otherwise $U\cap\partial_I\left[(\Omega_{2ps})^I\right]\neq\varnothing$. Let $z_I\in U\cap\partial_I\left[(\Omega_{2ps})^I\right]$ and $V_{(I)}\in\mathcal{B}$ with $z_I\in(V_{(I)})^I\subset\Omega_I$. Consider the map
	\begin{equation*}
		\begin{split}
			\varphi:\quad\mathbb{S}_\Omega^{2ps}\quad &\xlongrightarrow[\hskip1cm]{}\quad \mathcal{B},
			\\ J\qquad &\shortmid\!\xlongrightarrow[\hskip1cm]{}\quad V_{(J)}.
		\end{split}
	\end{equation*}
	We will prove that $\varphi$ is injective by contradiction: If $\varphi(K)=\varphi(L)$ for some $(K,L)\in\mathbb{S}_*^2$, then by $V_{(K)}=\varphi(K)=\varphi(L)=V_{(L)}$ we have \begin{equation*}
		\Psi_K^i(z_K)\in\partial(\Omega_{2ps})\cap V_{(K)}\cap V_{(L)}\subset\partial(\Omega_{2ps})\cap\Psi_K^i(\Omega_K)\cap\Psi_L^i(\Omega_L).
	\end{equation*}
	It implies that \eqref{eq-2ps} does not hold, a contradiction. Therefore $\varphi$ is injective and the cardinality $|\mathbb{S}_\Omega^{2ps}|<|\mathcal{B}|$. So $\mathbb{S}_\Omega^{2ps}$ is countable.
\end{proof}

\begin{prop}\label{pr-los2}
	Let $\Omega\subset\mathbb{H}$ be $2$-path-symmetric and $U\subset\mathbb{C}$ be path-connected with $U\cap\Omega_{2ps}\neq\varnothing$. If $U^I\subset\Omega_I$ and $U^J\subset\Omega_J$ for some $I\neq J$, then $U\subset\Omega_{2ps}$.
\end{prop}

\begin{proof}
	Let $z\in U\cap \Omega_{2ps}$. Then there are $\gamma\in\mathscr{P}_0(\mathbb{C})$ and $(I',J')\in\mathbb{S}^2_*$, s.t. $\gamma^{I'},\gamma^{J'}\subset\Omega$ and $\gamma(1)=z$. Since $\Omega$ is $2$-path-symmetric, $\gamma^K\subset\Omega$ for each $K\in\mathbb{S}$. Take $w\in U$, then there is a path $\beta$ in $U$ from $z$ to $w$. Thus
	\begin{equation*}
		\beta^L\subset U^L\subset\Omega_L,\qquad\forall\ L\in\{I,J\}.
	\end{equation*}
	Hence $(\gamma\beta)^{I},(\gamma\beta)^{J}\subset\Omega$ and $(\gamma\beta)(1)=w$. It implies $w\in\Omega_{2ps}$ and so $U\subset\Omega_{2ps}$.
\end{proof}

Let $\Omega\subset\mathbb{H}$ and set
\begin{equation*}\index{$\widetilde{\Omega}$}
	\widetilde{\Omega}:=\bigcup_{x+yI\in\Omega}x+y\mathbb{S}.
\end{equation*}
$\Omega$ is called \textit{\textbf{axially symmetric}}\index{Axially symmetric} if $\Omega=\widetilde{\Omega}$. We point out that
$\widetilde{\Omega}$ is said to be the \textit{\textbf{axially symmetric completion}}\index{Axially symmetric completion} of $\Omega$ according to the terminology in \cite{Colombo2009001}.

In the sequel, it is useful to introduce the notation
$$\Omega^{2ps}:=\{x+yI:x+yi\in\Omega_{2ps},\ I\in\mathbb{S}\}$$\index{$\Omega^{2ps}$}
to denote the axially symmetric completion of $\Omega_{2ps}$, i.e. $\Omega^{2ps}=\widetilde{\Omega^{2ps}}$.

\begin{rmk}\label{rmk-2ps}
	Let $\Omega\subset\mathbb{H}$ is $2$-path-symmetric. Then the following statements follow from the definitions:
	\begin{enumerate}[label=(\roman*)]
		\item\label{it-Omega 2ps} $\Omega^{2ps}$ is a subset of $\Omega$.
		\item\label{it-xyI} Let $I\in\mathbb{S}$. Then $x+yI\in\Omega^{2ps}$ if and only if $x+yJ\in\Omega^{2ps}$ for each $J\in\mathbb{S}$.
		\item If $\Omega\in\tau_s(\mathbb{H})$, then $\Omega^{2ps}\in\tau_s(\mathbb{H})$.
		\item\label{it-o2p} If $\Omega^{2ps}\in\tau_s(\mathbb{H})$, then for each $f\in\mathcal{SR}(\Omega^{2ps})$,
			\begin{equation*}
				f(x+yI)=(1,I)\begin{pmatrix}
					1&J\\1&K
				\end{pmatrix}\begin{pmatrix}
				f(x+yJ)\\f(x+yK)
			\end{pmatrix},\qquad\forall\ x+yI\in\Omega^{2ps}.
			\end{equation*}
		For all $J,K\in\mathbb{S}$ and $x+yI\in\Omega^{2ps}$.
		It implies that $f:(\Omega^{2ps},\tau)\rightarrow(\mathbb{H},\tau)$ is continuous.
		\item\label{it-o2c} $\Omega_{2ps}=\Conj_\mathbb{C}(\Omega_{2ps}):=\{z\in\mathbb{C}:\overline{z}\in\Omega_{2ps}\}$.
	\end{enumerate}
\end{rmk}

\section{Slice-piecewise path}\label{sec-slice piecewise paths}

In this brief section, we introduce the concept of slice-piecewise paths. These paths can describe the path-connectedness of st-domains in a finer way. In particular, we show that any two points in an st-domain can be connected by a slice-piecewise path.

\begin{defn}
	Let $n\in\mathbb{N}\backslash\{0\}$, $I_\imath\in\mathbb{S}$,
	\begin{equation*}
		I_\imath\in\mathbb{S},\qquad\gamma_\imath\in\mathscr{P}(\mathbb{C}):=\{\delta:[0,1]\rightarrow\mathbb{C}:\delta\mbox{ is continuous}\},\qquad \imath=1,...,n.
	\end{equation*}
	We call $\left((\gamma_1)^{I_1},...,(\gamma_n)^{I_n}\right)$, an ($n$-)slice-piecewise path from $\left[(\gamma_1)^{I_1}\right](0)$ to $\left[(\gamma_n)^{I_n}\right](1)$ if
	\begin{equation*}
		\gamma_\imath(1)=\gamma_{\imath+1}(0),\qquad\imath=1,...,n-1.
	\end{equation*}

	Moreover, we call an ($n$-)slice-piecewise path $\beta=(\beta_1,...,\beta_n)$ is in $\Omega\subset\mathbb{H}$, if $\beta_\imath\subset\Omega$, $\imath=1,...,n$.
\end{defn}

\begin{prop}\label{pr-fpp}
	Let $\Omega\subset\mathbb{H}$ be an st-domain, and $p,q\in\Omega$. Then there is a slice-piecewise path $\beta$ in $\Omega$ from $p$ to $q$.
\end{prop}

\begin{proof}
	According to \cite[Proposition 3.10]{Dou2020001}, $(\mathbb{H},\tau_s)$ is connected and locally path-connected, so is $(\Omega,\tau_s)$. Then there is a path $\alpha$ in $\Omega$ from $p$ to $q$. Let $t\in[0,1]$. As a result of \cite[Proposition 3.5]{Dou2020001}, there is an st-domain $U_t\subset \Omega$ containing $\alpha(t)$, such that $U_t$ is real-connected (that is $(U_t)_\mathbb{R}$ being connected in $\mathbb{R}$). Then the connected component $W_t$ of $\alpha^{-1}(U_t)$ containing $t$ is open in $[0,1]$. Since $[0,1]$ is compact, it follows that there is a minimal $m\in\mathbb{N}^+$, and $t_\imath\in[0,1]$, $\imath=1,2,...,m$, such that
	\begin{equation*}
		[0,1]\subset\bigcup_{\imath=1}^m W_{t_\imath}\qquad\mbox{and}\qquad 0<t_0<t_1<...<t_m<1.
	\end{equation*}
	In fact, $\cup_{\imath=1}^m W_{t_\imath}$ covers each point in $[0,1]$ at most twice. Otherwise, there is a point $q\in W_a\cap W_b\cap W_c$ in $[0,1]$, and $a,b,c\in\{t_1,t_2,...,t_m\}$ are different from each other. Then
	\begin{equation*}
		W:=W_a\cup W_b\cup W_c
	\end{equation*}
	is an open interval in $[0,1]$. Without loss of generality, we can suppose that
	\begin{equation*}
		\inf(W_a)=\inf(W)\qquad\mbox{and}\qquad\sup(W_b)=\sup(W).
	\end{equation*}
	Then $W=W_a\cup W_b$. So we can remove $W_c$, and $m$ is not minimal.
	
	Therefore there is a bijection
	\begin{equation*}
		L:\{1,2,...,m\}\rightarrow\{1,2,...,m\}
	\end{equation*}
	(reordering $\{W_{t_\imath}\}_{\imath=1}^m$ by the order of $\sup(W_{t_\imath})$), and
	\begin{equation*}
		s_\imath\in[0,1],\qquad\imath=1,2,...,m+1
	\end{equation*}
	with
	\begin{equation*}
		0=s_1<s_2<...<s_{m+1}=1\qquad\mbox{and}\qquad[s_\jmath,s_{\jmath+1}]\subset W_{t_{L(\jmath)}},\quad\jmath=1,2,...,m.
	\end{equation*}
	Then
	\begin{equation*}
		\alpha([s_\imath,s_{\imath+1}])\subset U_{t_{L(\imath)}},\qquad\imath=1,2,...,m.
	\end{equation*}
	Thanks to \cite[Corollary 3.9]{Dou2020001}, there are two paths $\beta_{2\imath-1}$ and $\beta_{2\imath}$ in $U$ such that $\beta_{2\imath-1}\beta_{2\imath}$ be a path from $\alpha(s_\imath)$ to $\alpha(s_{\imath+1})$ for each $\imath\in\{1,2,...,m\}$, and $\beta_\jmath$, $\jmath=1,...,2m$ is a path on a slice $\mathbb{C}_{I_\jmath}$ for some $I_\jmath\in\mathbb{S}$. Then $\Pi_{\imath=1}^{2m}\beta_\imath:=\beta_1\cdots\beta_{2m}$ is a path from $p$ to $q$, i.e. $(\beta_1,...,\beta_{2m})$ is a slice-piecewise path from $p$ to $q$.
\end{proof}

\section{$2$-path-symmetric Riemann domain}\label{sec-2ps Riemann domain}

In this section, we introduce a suitable notion of $2$-path-symmetric Riemann domain. To this end, we first define an equivalence relation $\sim_{\scriptscriptstyle\Omega}$ on $\bigsqcup_{I\in\mathbb{S}}\Omega_I\subset\mathbb{H}\times\mathbb{S}$. Then, for any 2-path-symmetric st-domain $\Omega\subset\mathbb H$ we define a new topological space $\breve{\Omega}$, see \eqref{eq-bob}, and a map $\pi_{\scriptscriptstyle\breve\Omega}$  which turns out to be a local homomorphism, see Proposition \ref{pr-usoi}. In Theorem  \ref{thm-Riemann doamin} we prove that $(\breve\Omega,\pi_{\scriptscriptstyle\breve\Omega})$ is a Riemann domain over $\mathbb C$ when $\Omega$ is a 2-path-symmetric domain and this is the reason to refer to $\breve \Omega$ as to the $2$-path-symmetric Riemann domain associated with $\Omega$.

Let $\Omega\subset\mathbb{H}$ and denote by
\begin{equation*}
	\bigsqcup_{I\in\mathbb{S}}\Omega_I
	:=\bigcup_{I\in\mathbb{S}}\big(\Omega_I\times\{I\}\big)
	=\big\{(u,I)\in\mathbb{H}\times\mathbb{S}:u\in\Omega_I\big\},
\end{equation*}
the disjoint union of $\{\Omega_I\}_{I\in\mathbb{S}}$ with the disjoint union topology
\begin{equation*}
	\tau\left(\bigsqcup_{I\in\mathbb{S}}\Omega_I\right):=\left\{U\subset\bigsqcup_{I\in\mathbb{S}}\Omega_I:U\cap\big(\Omega_I\times\{I\}\big)\in\tau\big(\Omega_I\times\{I\}\big),\ \forall\ I\in\mathbb{S}\right\}.
\end{equation*}
Let $u\in\bigsqcup_{I\in\mathbb{S}}\Omega_I$. Denote by $(z_u,I_u)$, the unique pair in $\mathbb{C}\times\mathbb{S}$ such that
\begin{equation}\label{eq-uxu}
	u=(z_u^{I_u},I_u),
\end{equation}
where $z_u^{I_u}:=(z_u)^{I_u}$.

Define a binary relation $\sim_{\scriptscriptstyle \Omega}$: Let $u,v\in\bigsqcup_{I\in\mathbb{S}}{\Omega_I}$. Then $u\sim_{\scriptscriptstyle \Omega} v$ if and only if the following statements hold:
\begin{equation}\label{eq-sim}
	\begin{split}
		&(i)\ \ z_u=z_v.\\
		&(ii)\ I_u=I_v\ \mbox{when}\ z_u\notin\Omega_{2ps}.
	\end{split}
\end{equation}
It is easy to check that if $u\sim_{\scriptscriptstyle \Omega} v$ and $u\neq v$, then
\begin{equation}\label{eq-jujv}
	I_u\neq I_v,\qquad\mbox{and}\qquad z_u=z_v\in\Omega_{2ps}.
\end{equation}
According to \eqref{eq-sim},
\begin{equation}\label{eq-zaia}
	z_a^{I_a}\approx_{\scriptscriptstyle \Omega} z_b^{I_b},\qquad\forall\ a,b\in\bigsqcup_{I\in\mathbb{S}}{\Omega_I}\mbox{ with } a\sim_{\scriptscriptstyle \Omega} b.
\end{equation}

\begin{proposition}
	Let $\Omega\subset\mathbb{H}$. Then the binary relation $\sim_{\scriptscriptstyle \Omega}$ is an equivalence relation.
\end{proposition}

\begin{proof}
	The reflexivity of $\sim_{\scriptscriptstyle \Omega}$ is directly by definition.
	
	Proof of transitivity: Let $u\sim_{\scriptscriptstyle \Omega} v$ and $v\sim_{\scriptscriptstyle \Omega} w$, we need to prove $u\sim_{\scriptscriptstyle \Omega} w$. If two of $u$, $v$ and $w$ coincide, then the transitivity holds. Otherwise, $u$, $v$ and $w$ are distinct to each other. $u\sim_{\scriptscriptstyle \Omega} v$ implies by \eqref{eq-jujv} that $z_u=z_v\in\Omega_{2ps}$. Similarly, $z_v=z_w\in\Omega_{2ps}$. In summary,
	\begin{equation*}
		z_u=z_v=z_w\in\Omega_{2ps}.
	\end{equation*}
	It is clear by \eqref{eq-sim} that $u\sim w$.
	
	Proof of symmetry: Let $u\sim_{\scriptscriptstyle \Omega} v$. If $u=v$, then $v\sim_{\scriptscriptstyle \Omega} u$ by reflexivity. Otherwise $u\neq v$. According to \eqref{eq-jujv},
	 $z_u=z=z_v\in\Omega_{2ps}$. Hence $v\sim_{\scriptscriptstyle \Omega} u$.
\end{proof}

Let $X$ be a set and let $P(X)$ be the power set of $X$, namely the collection of subsets of $X$:
\begin{equation*}\index{$P(X)$}
	P(X):=\{U\subset X\}.
\end{equation*}

Let $\Omega\in\tau_s(\mathbb{H})$ be a $2$-path-symmetric. We define the topological space
\begin{equation}\label{eq-bob}\index{$\breve{\Omega}$}
	\breve{\Omega}:=\bigsqcup_{I\in\mathbb{S}}{\Omega_I}\bigg/\sim_{\scriptscriptstyle \Omega}
\end{equation}
with the final topology
\begin{equation}\label{eq-two}\index{$\tau(\breve{\Omega})$}
	\tau(\breve{\Omega}):=\left\{U\subset\breve{\Omega}:\left( i_{\scriptscriptstyle \Omega}^I \right) ^{-1}(U)\in\tau\big(\Omega_I\times \{I\}\big),\ \forall\ I\in\mathbb{S}\right\}
\end{equation}
induced by inclusions $\left\{i_{\scriptscriptstyle \Omega}^I\right\}_{I\in\mathbb{S}}$ defined by
\begin{equation*}\index{$i_{\scriptscriptstyle \Omega}^I$}
	\begin{split}
		i_{\scriptscriptstyle \Omega}^I:\quad\Omega_I\times\{I\}\quad &\xlongrightarrow[\hskip1cm]{}\quad \breve{\Omega}\subset P\left(\bigsqcup_{I\in\mathbb{S}}\Omega_I\right),
			\\ u\qquad\ \ &\shortmid\!\xlongrightarrow[\hskip1cm]{}\quad
			\left[u\right]_{\sim_{\scriptscriptstyle \Omega}}
			:=\left\{v\in\bigsqcup_{I\in\mathbb{S}}\Omega_I:v\sim_{\scriptscriptstyle\Omega} u\right\}.
	\end{split}
\end{equation*}

It is easy to check that the definition implies that $i_{\scriptscriptstyle \Omega}^I$ is continuous and
\begin{equation}\label{eq-lir}
	\left( i_{\scriptscriptstyle \Omega} \right) ^{-1}\big([u]_{\sim_{\scriptscriptstyle \Omega}}\big)=\begin{cases}
		\left\{\left(z_u^J,J\right):J\in\mathbb{S}\right\},\qquad\ &z_u\in\Omega_{2ps},\\
		\left\{u=\left(z_u^{I_u},I_u\right)\right\},\qquad&\mbox{otherwise},
	\end{cases}
\end{equation}
where $i_{\scriptscriptstyle \Omega}: u\mapsto[u]_{\sim_{\scriptscriptstyle \Omega}}$ is the canonical projection, moreover
\begin{equation}\label{eq-wiiso}
	[(w^I,I)]_{\sim_\Omega}=[(w^J,J)]_{\sim_\Omega},\qquad\forall\ w\in\Omega_{2ps}\ \mbox{and}\ I,J\in\mathbb{S},
\end{equation}
Hence for each $U\subset\Omega_{2ps}$ and $I,J\in\mathbb{S}$,
\begin{equation}\label{eq-uii}
	[U^I\times \{I\}]_{\sim_\Omega}=\bigcup_{\xi\in U}[(\xi^I,I)]_{\sim_\Omega}=\bigcup_{\xi\in U}[(\xi^J,J)]_{\sim_\Omega}=[U^J\times \{J\}]_{\sim_\Omega},
\end{equation}
where
\begin{equation*}
	[U^I\times \{I\}]_{\sim_\Omega}:=\left\{[v]_{\sim_\Omega}:v\in U^I\times \{I\}\right\}.
\end{equation*}

\begin{prop}
	Let $\Omega\subset\mathbb{H}$ be a $2$-path-symmetric set, $I\in\mathbb{S}$ and $U\subset\Omega_I$. Then
	\begin{equation}\label{eq-lis}
		\left( i_{\scriptscriptstyle \Omega}^J \right) ^{-1}\bigg(\big[U\times\{I\}\big]_{\sim_{\scriptscriptstyle \Omega}}\bigg)=
		\begin{cases}
			\big(\Psi_I^J(U)\cap\Omega^{2ps}\big)\times\{J\},\qquad &J\neq I,
			\\U\times\{I\},\qquad &J=I,
		\end{cases}
	\end{equation}
	where $\Psi_I^J:=\Psi_i^J\circ\Psi_I^i$. \index{$\Psi_I^J$}
\end{prop}

\begin{proof}
``Proof of $\subset$'' Let
	\begin{equation*}
		u\in\left( i_{\scriptscriptstyle \Omega}^J \right) ^{-1}\big([U\times\{I\}]_{\sim_{\scriptscriptstyle \Omega}}\big)\subset\left( i_{\scriptscriptstyle \Omega}^J \right) ^{-1}(\breve{\Omega})=\Omega_J\times\{J\}.
	\end{equation*}
	Then there is $v\in U\times\{I\}$ such that $u\sim_{\scriptscriptstyle \Omega} v$. By \eqref{eq-sim}, $z_u=z_v$.
	
	If $I=J$, then
	\begin{equation*}
		u=\left(z_u^J,J\right)=\left(z_v^I,I\right)=v\in U\times\{I\}.
	\end{equation*}
	
	Otherwise, $I\neq J$. By \eqref{eq-sim}, $z_u^J\in\Omega^{2ps}$. It follows from $z_v^I\in U$ that $z_u^J=z_v^J\in \Psi_I^J(U)$. In summary,
	\begin{equation*}
		u=\left(z_u^J,J\right)\in\big(\Psi_I^J(U)\cap\Omega^{2ps}\big)\times\{J\}.
	\end{equation*}

	``Proof of $\supseteq$'' The case of $J=I$: Let $u\in U\times\{I\}$. Since $i_{\scriptscriptstyle \Omega}^I(u)=[u]_{\sim_{\scriptscriptstyle \Omega}}\in \big[U\times\{I\}\big]_{\sim_{\scriptscriptstyle \Omega}}$, we have $u\in\left( i_{\scriptscriptstyle \Omega}^I \right) ^{-1}\big(\big[U\times\{I\}\big]_{\sim_{\scriptscriptstyle \Omega}}\big)$.
	
	The case of $J\neq I$: Let $u\in\big(\Psi_I^J(U)\cap\Omega^{2ps}\big)\times\{J\}$. Then $z_u^J\in\Omega^{2ps}$. It implies that $u\sim_{\scriptscriptstyle \Omega} \left(z_u^I,I\right)=:v$ and $v\in U\times\{I\}$. Hence
	\begin{equation*}
		i_{\scriptscriptstyle \Omega}^J(u)=[u]_{\sim_{\scriptscriptstyle \Omega}}=[v]_{\sim_{\scriptscriptstyle \Omega}}\in\big[U\times\{I\}\big]_{\sim_{\scriptscriptstyle \Omega}},
	\end{equation*}
	i.e.
	$u\in\left( i_{\scriptscriptstyle \Omega}^J \right) ^{-1}\big(\big[U\times\{I\}\big]_{\sim_{\scriptscriptstyle \Omega}}\big)$.
\end{proof}

Let $\Omega\subset\mathbb{H}$ be a $2$-path-symmetric set, $\mu\in\breve{\Omega}$ and $u,v\in\left( i_{\scriptscriptstyle \Omega} \right) ^{-1}(\mu)$. It follows from $u\sim_{\scriptscriptstyle \Omega} v$ that $z_u=z_v$. So
\begin{equation}\label{eq-xx}
	z_\mu:=z_u
\end{equation}
do not depend on the choice of $u\in\left( i_{\scriptscriptstyle \Omega} \right) ^{-1}(\mu)$. Denote
\begin{equation}\label{eq-smli} \index{$\mathbb{S}_\mu$}
	\mathbb{S}_\mu:=\left\{I\in\mathbb{S}:z_\mu^I\in\Omega\mbox{ with }\left[(z_\mu^I, I)\right]_{\sim_{\scriptscriptstyle \Omega}}=\mu\right\}.
\end{equation}
It follows from \eqref{eq-lir} that
\begin{equation}\label{eq-msn}
	\mathbb{S}_\mu=\begin{cases}
		\mathbb{S},\qquad &z_\mu\in\Omega_{2ps},\\
		\{I_\mu\},\qquad &\mbox{otherwise},
	\end{cases},\qquad\qquad\forall\ \mu\in\breve{\Omega},
\end{equation}
where we denote by $I_\mu$ the unique element in $\mathbb{S}_\mu$.

\begin{prop}\label{pr-zinviant}
	Let $\Omega\in\tau_s(\mathbb{H})$ be $2$-path-symmetric, and $\mu\in\breve{\Omega}$. If $\mu=[(w^J,J)]_{\sim_\Omega}$ for some $w\in\mathbb{C}$ and $J\in\mathbb{S}$, then $w=z_\mu$.
\end{prop}

\begin{proof}
It follows directly by \eqref{eq-uxu} and \eqref{eq-xx}.
\end{proof}

\begin{prop}\label{pr-homeomorphism}
	Let $\Omega\in\tau_s(\mathbb{H})$ be $2$-path-symmetric, $I\in\mathbb{S}$ and $U\in\tau(\Omega_I)$. Then
	\begin{equation}\label{eq-lio}
		\left( i_{\scriptscriptstyle \Omega}^J \right) ^{-1}\bigg([U\times \{I\}]_{\sim_{\scriptscriptstyle \Omega}}\bigg)\in\tau\big(\Omega_J\times\{J\}\big),\qquad\forall\ J\in\mathbb{S},
	\end{equation}
	and
	\begin{equation}\label{eq-uib}
		\big[U\times\{I\}\big]_{\sim_{\scriptscriptstyle \Omega}}\in\tau(\breve{\Omega}).
	\end{equation}

	Moreover,
	\begin{equation*}
		\begin{split}
			\imath_{\scriptscriptstyle \Omega}^I:\Omega_I\times \{I\}&\xlongrightarrow[\hskip1cm]{}[\Omega_I\times \{I\}]_{\sim_{\scriptscriptstyle \Omega}},
			\\ u\quad\ \ &\shortmid\!\xlongrightarrow[\hskip1cm]{}\quad i_{\scriptscriptstyle \Omega}^I(u).
		\end{split}
	\end{equation*}
	is a homeomorphism.
\end{prop}

\begin{proof}
	According to Proposition \ref{pr-2psopen}, $\Omega_{2ps}\in\tau(\mathbb{C})$. It implies by definition that $\Omega^{2ps}\cap\mathbb{C}_J=(\Omega_{2ps})^J$ is open in $\mathbb{C}_J$. Then
	\begin{equation*}
		\Psi_I^J(U)\cap\Omega^{2ps}=\Psi_I^J(U)\cap\mathbb{C}_J\cap\Omega^{2ps}=\Psi_I^J(U)\cap(\mathbb{C}_J\cap\Omega^{2ps})\in\tau(\mathbb{C}_J).
	\end{equation*}
	It follows from \eqref{eq-lis} that \eqref{eq-lio} holds. And \eqref{eq-uib} holds by \eqref{eq-lio} and \eqref{eq-two}.
	
	Let $v,u\in\Omega_I\times \{I\}$ with $\imath_{\scriptscriptstyle \Omega}^I(u)=\imath_{\scriptscriptstyle \Omega}^I(v)=:\mu$. Then $u,v\in\left( i_{\scriptscriptstyle \Omega}^I \right) ^{-1}(\mu)$. By \eqref{eq-xx}, $z_u=z_\mu=z_v$. Therefore
	\begin{equation*}
		u=\left(z_u^I,I\right)=\left(z_v^I,I\right)=v.
	\end{equation*}
	It implies that $\imath_{\scriptscriptstyle \Omega}^I$ is injective. And by definition $\imath_{\scriptscriptstyle \Omega}^I$ is surjective.
	
	Let $V\in\tau(\Omega_I\times \{I\})$. Then $V=W\times\{I\}$ for some $W\in\tau(\Omega_I)$. By \eqref{eq-uib},
	\begin{equation*}
		\left[\left(\imath_{\scriptscriptstyle \Omega}^I\right)^{-1}\right]^{-1}(V)=\imath_{\scriptscriptstyle \Omega}^I(W\times\{I\})=i_{\scriptscriptstyle \Omega}^I(W\times\{I\})=[W\times\{I\}]_{\sim_{\scriptscriptstyle \Omega}}\in\tau(\breve{\Omega})
	\end{equation*}
	Hence $\left(\imath_{\scriptscriptstyle \Omega}^I\right)^{-1}$ is continuous. Since $i_{\scriptscriptstyle \Omega}^I$ is continuous, so is $\imath_{\scriptscriptstyle \Omega}^I$. In summary, $\imath_{\scriptscriptstyle \Omega}^I$ is a homeomorphism.
\end{proof}

Let $z\in\mathbb{C}$, $I\in\mathbb{S}$ and recall that $B_I(z^I,r):=[B_{\mathbb{C}}(z,r)]^I$ denotes the ball with center $z^I$ and radius  $r\ge 0$ in $\mathbb{C}_I$, where $B_{\mathbb{C}}(z,r)$ denotes the open ball with center $z$ and  radius $r$ in $\mathbb{C}$.
Note that $B_\mathbb{C}(z,0)=\varnothing$.

Let $\mu\in\breve{\Omega}$.  Denote
\begin{equation}\label{eq-rmi}
	r_\mu^I:=
	\begin{cases}
		\delta_I(z_\mu^I,\mathbb{C}_I\backslash\Omega_I),\qquad &I\in\mathbb{S}_\mu,\\
		0,\qquad&\mbox{otherwise},
	\end{cases}
\end{equation}
where $\delta_I$ is the distance function on $\mathbb{C}_I$. Moreover, we set
\begin{equation}\label{eq-rms}\index{$r_\mu^+$, $r_\mu^-$, $r_\mu^I$}
	r_\mu^+:=\sup_{I\in\mathbb{S}} r_\mu^I,\qquad r_\mu^-:=\inf_{I\in\mathbb{S}} r_\mu^I,
\end{equation}
\begin{equation}\label{eq-msm}\index{$\mathbb{S}_\mu^+$, $\mathbb{S}_\mu^-$}
	\mathbb{S}_\mu^+:=\{I\in\mathbb{S}:r_\mu^I=r_\mu^+\},\qquad\mbox{and}\qquad \mathbb{S}_\mu^-:=\{I\in\mathbb{S}:r_\mu^I=r_\mu^-\}.
\end{equation}

\begin{prop}\label{pr-losm}
	Let $\Omega\in\tau_s(\mathbb{H})$ be $2$-path-symmetric, and $\mu\in\breve{\Omega}$. Then the following statements holds:
	\begin{enumerate}[label=(\roman*)]
		\item \label{it-msb} $|\mathbb{S}\backslash\mathbb{S}_\mu^-|\le 1$.
		\item
		\begin{equation}\label{eq-msmb}
			\mathbb{S}_\mu^+=\begin{cases}
				\mathbb{S}\backslash\mathbb{S}_\mu^-,\qquad &|\mathbb{S}_\mu^+|=1,
				\\\mathbb{S}=\mathbb{S}_\mu=\mathbb{S}_\mu^-,\qquad &\mbox{otherwise}.
			\end{cases}
		\end{equation}
		Moreover
		\begin{equation}\label{eq-rml}
				r_\mu^\lambda=r_\mu^I,\qquad\forall\ I\in\mathbb{S}_\mu^\lambda,\qquad \lambda\in\{+,-\}.
		\end{equation}
		\item\label{it-lri} Let $r\in(0,r_\mu^+]$. Then
		\begin{equation}\label{eq-balldef}\index{$B_{\breve{\Omega}}(\mu,r)$}
			B_{\breve{\Omega}}(\mu,r):=\left[(B_\mathbb{C}\left(z_\mu,r\right))^I\times\{I\}\right]_{\sim_{\scriptscriptstyle \Omega}},
		\end{equation}
		does not depend on the choice of $I\in\mathbb{S}_\mu$ with $B_I\left(z_\mu^I,r\right)\subset\Omega$.
		\item $\beta_{\breve{\Omega}}:=\{B_{\breve{\Omega}}(\nu,r): \nu\in\breve{\Omega},\ 0<r\le r_\nu^+\}$ is a basis for the topology $\tau(\breve{\Omega})$.
	\end{enumerate}
\end{prop}

\begin{proof}
	(i) Proof by contradiction. If $|\mathbb{S}\backslash\mathbb{S}^-|\ge 2$, then there is $(I,J)\in\mathbb{S}^2_*$ such that
	\begin{equation*}
		r_\mu^I\ge r_\mu^J>r_\mu^-\ge 0.
	\end{equation*}
	By \eqref{eq-rmi}, $I,J\in\mathbb{S}_\mu$. It follows from \eqref{eq-msn} that $z_\mu\in\Omega_{2ps}$ and $\mathbb{S}_\mu=\mathbb{S}$. Let $U:=B_\mathbb{C}(z_\mu,r_\mu^J)$. It is easy to check by \eqref{eq-rmi} that $U^I\subset\Omega_I$ and $U^J\subset\Omega_J$. According to Proposition \ref{pr-los2}, $U\subset\Omega_{2ps}$. Hence
	\begin{equation*}
		U^K\subset\Omega_K,\qquad\forall\ K\in\mathbb{S}=\mathbb{S}_\mu.
	\end{equation*}
	Again by \eqref{eq-msn}, $r_\mu^K\ge r_\mu ^J$. Therefore $r_\mu^-=\inf_{L\in\mathbb{S}}r_{\mu}^L\ge r_\mu^J$, a contradiction.
	
	(ii) Suppose that $|\mathbb{S}\backslash\mathbb{S}_\mu^-|=0$. Then $\mathbb{S}=\mathbb{S}_\mu^-$. By \eqref{eq-msm},
	\begin{equation}\label{eq-rmir}
		r_\mu^I=r_\mu^-,\qquad\forall\ I\in\mathbb{S}.
	\end{equation}
	It implies from \eqref{eq-rms} that $r_\mu^+=r_\mu^-$ and $\mathbb{S}_\mu^+=\mathbb{S}_\mu^-$. According to \eqref{eq-rmi} and \eqref{eq-rmir}, $\mathbb{S}_\mu=\mathbb{S}$. In summary,
	\begin{equation*}
		\mathbb{S}_\mu^+=\mathbb{S}_\mu^-=\mathbb{S}_\mu=\mathbb{S}
	\end{equation*}
	and then \eqref{eq-msmb} holds.
	
	Suppose that $|\mathbb{S}\backslash\mathbb{S}_\mu^-|=1$. Then $\mathbb{S}\backslash\mathbb{S}_\mu^-=\{I\}$ for some $I\in\mathbb{S}$. By definition, $r_\mu^+\neq r_\mu^-$. It is easy to check that $\mathbb{S}_\mu^+=\{I\}=\mathbb{S}\backslash\mathbb{S}_\mu^-$. Therefore \eqref{eq-msmb} holds by \ref{it-msb}.
	
	(iii) We only need to prove (iii) holds when $|\mathbb{S}_\mu|>1$. Let $I,J\in\mathbb{S}_\mu$ with $I\neq J$, $B_I\left(z_\mu^J,r\right)\subset\Omega$ and $B_J\left(z_\mu^J,r\right)\subset\Omega$. By similar method in (i), $B_{\mathbb{C}}(z_\mu,r)\subset\Omega_{2ps}$. By \eqref{eq-uii},
	\begin{equation*}
			\left[(B_\mathbb{C}(z_\mu,r))^J\times\{J\}\right]_{\sim_{\scriptscriptstyle \Omega}}=\left[(B_\mathbb{C}(z_\mu,r))^I\times\{I\}\right]_{\sim_{\scriptscriptstyle \Omega}}.
	\end{equation*}
	It clear that (iii) holds.

	(iv) According to \eqref{eq-uib},
	\begin{equation}\label{eq-bbon}
		B_{\breve{\Omega}}(\nu,r)\in\tau(\breve{\Omega}),\qquad\forall\ \nu\in\breve{\Omega}\ \mbox{and}\ r\in(0,r_\nu^+].
	\end{equation}
	
	Let $U\in\tau(\breve{\Omega})$ and $\xi\in U$. We will find $V\in\beta_{\breve{\Omega}}$ such that $\xi\in V\subset U$: Fix $J\in\mathbb{S}_\xi$. Then by \eqref{eq-two},
	\begin{equation*}
		(z_\xi^J,J)\in\left(i_{\scriptscriptstyle \Omega}^J\right)^{-1}(U)\in\tau(\Omega_J\times\{J\}).
	\end{equation*}
	It implies that there is $r>0$ such that
	\begin{equation}\label{eq-bjzm}
		[B_\mathbb{C}(z_\xi,r)]^J\times\{J\}=B_J(z_\xi^J,r)\times\{J\}\subset\left(i_{\scriptscriptstyle \Omega}^J\right)^{-1}(U)\subset\Omega_J\times\{J\}.
	\end{equation}
	Hence
	\begin{equation}\label{eq-mibb}
		\xi\in\bigg[[B_\mathbb{C}(z_\xi,r)]^J\times\{J\}\bigg]_{\sim_{\scriptscriptstyle \Omega}}\subset U.
	\end{equation}
	By \eqref{eq-bjzm}, $[B_\mathbb{C}(z_\xi,r)]^J\subset\Omega$. It implies that $r\le r_\xi^J\le r_\xi^+$. Then
	\begin{equation*}
		\bigg[[B_\mathbb{C}(z_\xi,r)]^J\times\{J\}\bigg]_{\sim_{\scriptscriptstyle \Omega}}=B_{\breve{\scriptscriptstyle \Omega}}(\xi,r)\in\beta_{\breve{\Omega}}.
	\end{equation*}
	According \eqref{eq-bbon} and \eqref{eq-mibb}, $\beta_{\breve{\Omega}}$ is a basis for $\tau(\breve{\Omega})$.
\end{proof}

\begin{prop}\label{pr-ni}
	Let $\Omega\in\tau_s(\mathbb{H})$ be $2$-path-symmetric, $\mu,\nu\in\breve{\Omega}$ with $z_\mu=z_\nu\in\overline{\Omega_{2ps}}$. Then $\mu=\nu$.
\end{prop}

\begin{proof}
	If $z_\mu\in\Omega_{2ps}$, then choose $I\in\mathbb{S}_\mu$. By definition, $\mu=[(z_\mu^I,I)]_{\sim_{\scriptscriptstyle \Omega}}$. According to \eqref{eq-lir},
	\begin{equation*}
		\begin{split}
			\left( i_{\scriptscriptstyle \Omega} \right) ^{-1}(\mu)=&\left( i_{\scriptscriptstyle \Omega} \right) ^{-1}\bigg([(z_\mu^I,I)]_{\sim_{\scriptscriptstyle \Omega}}\bigg)=\bigg\{(z_\mu^J ,J):J\in\mathbb{S}\bigg\}
			\\=&\bigg\{(z_\nu^J,J):J\in\mathbb{S}\bigg\}=\left( i_{\scriptscriptstyle \Omega} \right) ^{-1}(\nu).
		\end{split}
	\end{equation*}
	It implies that $\mu=\nu$.
	
	Otherwise $z_\mu\in  \overline{\Omega_{2ps}}\backslash\Omega_{2ps}=\partial(\Omega_{2ps})$. Let $I\in\mathbb{S}_\mu$ and $J\in\mathbb{S}_\nu$. Then $z_\mu^I\in\Omega_I$, $z_\nu^J\in\Omega_J$ and
	\begin{equation*}
		z_\mu=z_\nu\in\partial(\Omega_{2ps})\bigcap\Psi_I^i(\Omega_I)\bigcap\Psi_J^i(\Omega_J).
	\end{equation*}
	It follows by \eqref{eq-2ps} that $I=J$. Then
	\begin{equation*}
		\mu=[(z_\mu^I,I)]_{\sim_{\scriptscriptstyle \Omega}}=[(z_\nu^I,I)]_{\sim_{\scriptscriptstyle \Omega}}=\nu.
	\end{equation*}
\end{proof}

Let $\Omega\subset\mathbb{H}$ be a $2$-path-symmetric st-domain. According to \eqref{eq-xx}, we define the map
\begin{equation*}
	\begin{split}
		\pi_{\breve{\scriptscriptstyle \Omega}}:\ \breve{\Omega}\ &\xlongrightarrow[\hskip1cm]{}\ \mathbb{C},
		\\ \mu\ &\shortmid\!\xlongrightarrow[\hskip1cm]{}\ z_\mu.
	\end{split}
\end{equation*}

\begin{prop}\label{pr-usoi}
	Let $\Omega\in\tau_s(\mathbb{H})$ be $2$-path-symmetric, $I\in\mathbb{S}$ and $U\in\tau\left(\Omega_I\right)$. Then
	\begin{equation}\label{eq-pso}
		\begin{split}
			\pi_{\breve{\scriptscriptstyle \Omega}}|_{[U\times\{I\}]_{\sim_{\scriptscriptstyle \Omega}}}:\ [U\times\{I\}]_{\sim_{\scriptscriptstyle \Omega}}\  &\xlongrightarrow[\hskip1cm]{}\ \Psi_I^i(U),
			\\ \big[(z^I,I)\big]_{\sim_{\scriptscriptstyle \Omega}}\ &\shortmid\!\xlongrightarrow[\hskip1cm]{}\quad z,\qquad\forall\ z^I\in U,
		\end{split}
	\end{equation}
	is a homeomorphism. Therefore, $\pi_{\breve{\scriptscriptstyle \Omega}}$ is a local homeomorphism.
\end{prop}

\begin{proof}
	(i) It is easy to check by \eqref{eq-lir} that the preimage
	\begin{equation}\label{eq-sco}
		\left(\pi_{\breve{\scriptscriptstyle \Omega}}|_{[U\times\{I\}]_{\sim_{\scriptscriptstyle \Omega}}}\right)^{-1}(z)=\bigg\{\big[(z^I,I)\big]_{\sim_{\scriptscriptstyle \Omega}}\bigg\},\qquad\forall\ z\in\Psi_I^i(U),
	\end{equation}
	is a singleton. It follows that $\pi_{\breve{\scriptscriptstyle \Omega}}|_{[U\times\{I\}]_{\sim_{\scriptscriptstyle \Omega}}}$ is bijective, moreover
	$\pi_{\scriptscriptstyle \Omega}|_{[U\times\{I\}]_{\sim_{\scriptscriptstyle \Omega}}}$ is continuous. In fact, let $V\in\tau\left(\Psi_I^i(U)\right)$. It implies that
	\begin{equation*}
		V^I\in\tau\left( U\right)\subset\tau\left(\Omega_I\right).
	\end{equation*}
	It follows from \eqref{eq-sco} and \eqref{eq-uib} that
	\begin{equation*}
		\left(\pi_{\breve{\scriptscriptstyle \Omega}}|_{[U\times\{I\}]_{\sim_{\scriptscriptstyle \Omega}}}\right)^{-1}(V)=\left[\Psi_i^I(V)\times\{I\}\right]_{\sim_{\scriptscriptstyle \Omega}}\in\tau(\breve{\Omega}).
	\end{equation*}
	Since $\left[\Psi_i^I(V)\times\{I\}\right]_{\sim_{\scriptscriptstyle \Omega}}\subset\left[U\times\{I\}\right]_{\sim_{\scriptscriptstyle \Omega}}$, we have
	\begin{equation*}
		\left(\pi_{\breve{\scriptscriptstyle \Omega}}|_{[U\times\{I\}]_{\sim_{\scriptscriptstyle \Omega}}}\right)^{-1}(V)\in\tau\big(\left[U\times\{I\}\right]_{\sim_{\scriptscriptstyle \Omega}}\big).
	\end{equation*}
	It implies that $\pi_{\breve{\scriptscriptstyle \Omega}}|_{[U\times\{I\}]_{\sim_{\scriptscriptstyle \Omega}}}$ is continuous.
	
Also	$\big(\pi_{\breve{\scriptscriptstyle \Omega}}|_{[U\times\{I\}]_{\sim_{\scriptscriptstyle \Omega}}}\big)^{-1}$ is continuous: let $W\in\tau\big(\left[U\times\{I\}\right]_{\sim_{\scriptscriptstyle \Omega}}\big)$. Then there is $D\subset U$ such that $W=[D\times\{I\}]_{\sim_{\scriptscriptstyle \Omega}}$. By \eqref{eq-lis} and \eqref{eq-two},
	\begin{equation*}
		D\times\{I\}=\left( i_{\scriptscriptstyle \Omega}^I \right) ^{-1}\bigg(\big[D\times\{I\}\big]_{\sim_{\scriptscriptstyle \Omega}}\bigg)\in\tau\big(\Omega_I\times\{I\}\big).
	\end{equation*}
	Since $D\subset U$, we have $D\in\tau(U)$ and $\Psi_I^i (D)\in\tau\big(\Psi_I^i(U)\big)$. Therefore by \eqref{eq-pso}
	\begin{equation*}
		\left(\big(\pi_{\breve{\scriptscriptstyle \Omega}}|_{[U\times\{I\}]_{\sim_{\scriptscriptstyle \Omega}}}\big)^{-1}\right)^{-1}(W)=\pi_{\breve{\scriptscriptstyle \Omega}}|_{[U\times\{I\}]_{\sim_{\scriptscriptstyle \Omega}}}\left([D\times\{I\}]_{\sim_\Omega}\right)=\Psi_I^i (D)\in\tau\big(\Psi_I^i(U)\big).
	\end{equation*}
	It implies that $\big(\pi_{\breve{\scriptscriptstyle \Omega}}|_{[U\times\{I\}]_{\sim_{\scriptscriptstyle \Omega}}}\big)^{-1}$ is continuous. In summary, $\pi_{\breve{\scriptscriptstyle \Omega}}|_{[U\times\{I\}]_{\sim_{\scriptscriptstyle \Omega}}}$ is a homeomorphism.
	
	(ii) According to \eqref{eq-uib}, for each $\mu\in\breve{\Omega}$, the set $[\Omega_{I_\mu}\times\{I_\mu\}]_{\sim_{\scriptscriptstyle \Omega}}$ is an open neighborhood of $\mu$. It follows from (i) that $\pi_{\scriptscriptstyle \Omega}$ is a local homeomorphism.
\end{proof}

\begin{prop}\label{pr-osmh}
	Let $\Omega\subset\mathbb{H}$ be a $2$-path-symmetric st-domain. Then $(\breve{\Omega},\tau)$ is Hausdorff.
\end{prop}

\begin{proof}
	Let $\mu,\nu\in\breve{\Omega}$ with $\mu\neq\nu$. According to Proposition \ref{pr-losm} (iii), there is $r'>0$ such that $\mu\in B_{\breve{\Omega}}(\mu,r')\in\tau(\breve{\Omega})$, and $\nu\in B_{\breve{\Omega}}(\nu,r')\in\tau(\breve{\Omega})$. If there is $r\in(0,r')$ such that
	\begin{equation}\label{eq-lioi}
		B_{\breve{\Omega}}(\mu,r)\cap B_{\breve{\Omega}}(\mu,r)=\varnothing,
	\end{equation}
	then $\tau(\breve{\Omega})$ is Hausdorff.
	
	(i) If $z_\mu\neq z_\nu$, then set $r=\min\left\{\frac{1}{2}|z_\mu-z_\nu|,r'\right\}$. By \eqref{eq-lis},
	\begin{equation*}
		\left( i_{\scriptscriptstyle \Omega}^I \right) ^{-1}\bigg(B_{\breve{\Omega}}(\xi,r)\bigg)\subset B_I(z_\xi^I,r)\times\{I\},\qquad\xi=\mu,\nu,\ \forall\ I\in\mathbb{S}.
	\end{equation*}
	Hence
	\begin{equation*}
		\begin{split}
			&\bigg(\left( i_{\scriptscriptstyle \Omega}^I \right) ^{-1}\big(B_{\breve{\Omega}}(\mu,r)\big)\bigg)\bigcap\bigg(\left( i_{\scriptscriptstyle \Omega}^I \right) ^{-1}\big(B_{\breve{\Omega}}(\nu,r)\big)\bigg)
			\\\subset& \big(B_I(z_\mu^I,r)\times\{I\}\big)\cap\big( B_I(z_\nu^I,r)\times\{I\}\big)=\varnothing,
		\end{split}\qquad\forall\ I\in\mathbb{S}.
	\end{equation*}
	It implies that
	\begin{equation}\label{eq-blir}
		\bigg(\left( i_{\scriptscriptstyle \Omega} \right) ^{-1}\big(B_{\breve{\Omega}}(\mu,r)\big)\bigg)\bigcap\bigg(\left( i_{\scriptscriptstyle \Omega} \right) ^{-1}\big(B_{\breve{\Omega}}(\nu,r)\big)\bigg)=\varnothing
	\end{equation}
	Then \eqref{eq-lioi} holds.
	
	(ii) Otherwise, $z_\mu=z_\nu$. By Proposition \ref{pr-ni}, $z_\mu\notin \overline{\Omega_{2ps}}$. Then there is $r\in(0,r')$ such that $B_\mathbb{C}(z_\mu,r)\cap\Omega_{2ps}=\varnothing$. It follows from \eqref{eq-lis} that
	\begin{equation}\label{eq-lior1}
		\left( i_{\scriptscriptstyle \Omega} \right) ^{-1}\bigg(B_{\breve{\Omega}}(\xi,r)\bigg)=\bigcup_{I\in\mathbb{S}_\xi}\bigg(B_I(z_\xi^I,r)\bigg),\qquad\xi=\mu,\nu.
	\end{equation}
	Since $\mu\neq\nu$, we have $\mathbb{S}_\mu\cap\mathbb{S}_\nu=\varnothing$. It follows from \eqref{eq-lior1} that \eqref{eq-blir} holds, and then \eqref{eq-lioi} holds.
\end{proof}

\begin{prop}\label{pr-fppo}
	Let  $\Omega\subset\mathbb{H}$ be a $2$-path-symmetric st-domain, and $\mu,\nu\in\breve{\Omega}$. Then there is a slice-piecewise path $\beta=(\beta_1,...,\beta_m)$ from $p\in\Omega$ to $q\in\Omega$, and $I_1,...,I_m\in\mathbb{S}$ such that
	\begin{equation*}
		\mu=[(p,I_1)]_{\sim_{\scriptscriptstyle \Omega}},\qquad\nu=[(q,I_m)]_{\sim_{\scriptscriptstyle \Omega}},\qquad\mbox{and}\qquad \beta_\imath\subset\Omega_{I_\imath}.
	\end{equation*}
\end{prop}

\begin{proof}
	By definition, there is $p,q\in\Omega$ and $J_1,J_m\in\mathbb{S}$ such that
	\begin{equation*}
		\mu=[(p,J_1)]_{\sim_{\scriptscriptstyle \Omega}}\qquad\mbox{and}\qquad
		\nu=[(q,J_m)]_{\sim_{\scriptscriptstyle \Omega}}.
	\end{equation*}
	According to Proposition \ref{pr-fpp}, there is a slice-piecewise path $\beta=(\beta_1,...,\beta_m)$ in $\Omega$ from $p$ to $q$. Then
	\begin{equation*}
		\beta_1\subset\mathbb{C}_{K_1},
		\qquad\beta_m\subset\mathbb{C}_{K_m}\qquad\mbox{and}\qquad
		\beta_\imath\subset\mathbb{C}_{I_\imath},\qquad
		\imath=2,...,m-1,
	\end{equation*}
	for some $K_1,K_m,I_\imath\in\mathbb{S}$.
	
	We will choose $I_1\in\{J_1,K_1\}$ such that $\mu=[(p,I_1)]_{\sim_{\scriptscriptstyle \Omega}}$ and $\beta_1\subset\Omega_{I_1}$:
	If $K_1=\pm J_1$, then let $I_1=J_1$. It implies that $\beta_1\subset\mathbb{C}_{K_1}=\mathbb{C}_{J_1}=\mathbb{C}_{I_1}$. Otherwise $K_1\neq \pm J_1$. Let $I_1=K_1$. It follows that $p\in\mathbb{C}_{J_1}\cap\mathbb{C}_{K_1}=\mathbb{R}$. By definition, $p\in\Omega^{2ps}$. Then
	\begin{equation*}
		\mu=[(p,J_1)]_{\sim_{\scriptscriptstyle \Omega}}=[(p,K_1)]_{\sim_{\scriptscriptstyle \Omega}}=[(p,I_1)]_{\sim_{\scriptscriptstyle \Omega}}.
	\end{equation*}
	
	Similarly, there is $I_m\in\{J_m,K_m\}$ such that $\nu=[(q,I_m)]_{\scriptscriptstyle \Omega}$ and $\beta_m\subset\Omega_{I_m}$.
\end{proof}

Recall that (see e.g. \cite[Page 88]{Fritzsche2002001B}) a Riemann domain over $\mathbb{C}$ is a pair $(G,\pi)$ with the following properties:
\begin{enumerate}[label=(\roman*)]
	\item $G$ is a connected Hausdorff space.
	\item $\pi:G\rightarrow\mathbb{C}$ is a local homeomorphism.
\end{enumerate}

\begin{thm}\label{thm-Riemann doamin}
	Let $\Omega\subset\mathbb{H}$ be a $2$-path-symmetric st-domain. Then $(\breve{\Omega},\pi_{\breve{\scriptscriptstyle \Omega}})$ is a Riemann domain over $\mathbb{C}$.
\end{thm}

\begin{proof}
	Let $\mu,\nu\in\breve{\Omega}$. By Proposition \ref{pr-fppo}, there is a slice-piecewise path $\beta=(\beta_1,...,\beta_m)$ from $p\in\Omega$ to $q\in\Omega$, and $I_1,...,I_m\in\mathbb{S}$ such that
	\begin{equation*}
		\mu=[(p,I_1)]_{\sim_{\scriptscriptstyle \Omega}},\qquad
		\nu=[(q,I_m)]_{\sim_{\scriptscriptstyle \Omega}} \qquad\mbox{and}\qquad
		\beta_\imath\subset\Omega_{I_\imath}.
	\end{equation*}
	Note that $i_{\scriptscriptstyle \Omega}^{I_\imath}\circ \beta'_\imath$ is a path in $\breve{\Omega}$, where
	\begin{equation*}
		\begin{split}
			\beta'_\imath:\ [0,1]\ &\xlongrightarrow[\hskip1cm]{}\ \Omega_{I_\imath}\times\{I_\imath\},
			\\ t\quad &\shortmid\!\xlongrightarrow[\hskip1cm]{}\ (\beta(t),I_\imath).
		\end{split}
	\end{equation*}
	It implies that
	\begin{equation*}
		\prod_{\imath=1}^m \left(i_{\scriptscriptstyle \Omega}^{I_\imath}\circ \beta'_\imath\right)=\left(i_{\scriptscriptstyle \Omega}^{I_1}\circ \beta_1'\right)\circ\cdots\circ \left(i_{\scriptscriptstyle \Omega}^{I_m}\circ \beta_m'\right)
	\end{equation*}
	is a path in $\breve{\Omega}$ from $\mu=[(p,I_1)]_{\sim_{\scriptscriptstyle \Omega}}$ to $\nu=[(q,I_m)]_{\sim_{\scriptscriptstyle \Omega}}$. Hence $\breve{\Omega}$ is path-connected and then connected. According to Propositions \ref{pr-usoi} and \ref{pr-osmh}, $\breve{\Omega}$ is Hausdorff and $\pi_{\breve{\scriptscriptstyle \Omega}}$ is a local homeomorphism. In summary, $(\breve{\Omega},\pi_{\breve{\scriptscriptstyle \Omega}})$ is a Riemann domain over $\mathbb{C}$.
\end{proof}

In the sequel we shall refer to $\breve\Omega$ as the 2-path-symmetric Riemann domain associated with $\Omega$.
\begin{example}
	Let $E_1, U_{a,b},I,J$ be defined as in Example \ref{exa-e1}. Then
	\begin{equation*}
		\breve{E_1}\cong \left.\left[\left(U'_{-\frac{\pi}{2},\frac{3\pi}{4}}\times\{1\}\right)\bigcup \left(U'_{\frac{\pi}{4},\frac{3\pi}{2}}\times\{2\}\right)\right]\right/\sim,
	\end{equation*}
	with the structure-preserving map
	\begin{equation*}
		\phi\left(\left[(z^K,K)\right]_{\sim_\Omega}\right)=\begin{cases}
			[(z,1)]_\sim,\qquad K=\pm I,
			\\ [(z,2)]_\sim,\qquad \mbox{otherwise},
		\end{cases}\qquad\forall\ z^K\in E_1,
	\end{equation*}
	where
	\begin{equation*}
		U'_{a,b}:=U_{a,b}\cap{\Conj}_{\mathbb{C}}(U_{a,b})
	\end{equation*}
	and $(z,n)\sim(w,m)$ with $(z,n)\neq(w,m)$ if and only if $n\neq m$ and $z=w\in U_{0,3\pi}\cap{\Conj}_{\mathbb{C}}\left(U_{0,3\pi}\right)$.
\end{example}


\begin{figure}\includegraphics[width=12.5cm]{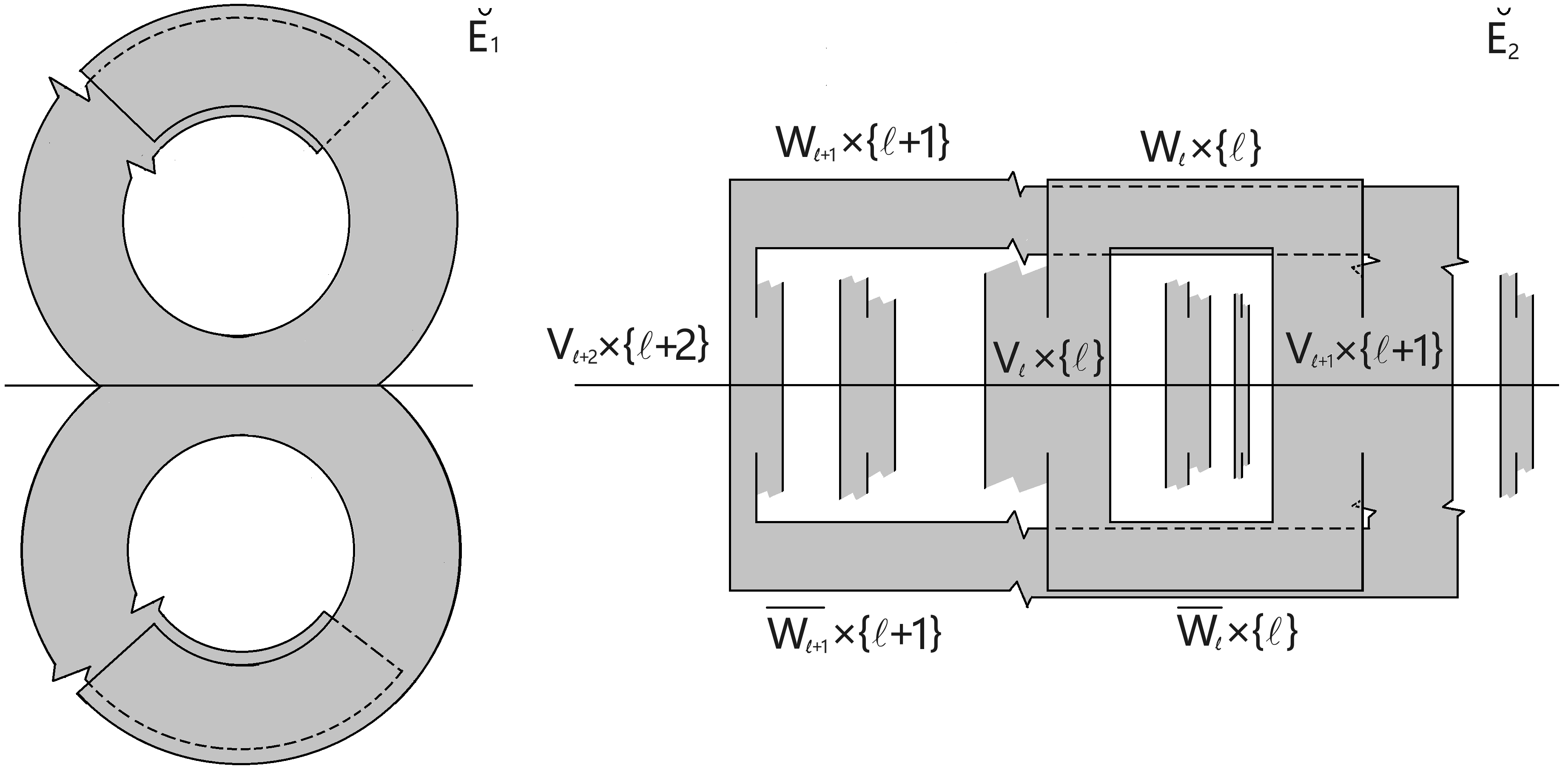}\caption{}\end{figure}

\begin{example}
	Let $E_2$, $V_\ell$, $W_\ell$, $I_\ell$ be defined as in Example \ref{exa-e2}. Then
	\begin{equation*}
		\breve{E_2}\cong\left.\bigcup_{\ell=1}^{+\infty}\left[V_\ell\cup W_\ell\cup{\Conj}_{\mathbb{C}}(W_\ell)\right]\times\{\ell\}\right/\sim
	\end{equation*}
	with the structure-preserving map
	\begin{equation*}
		\phi\left([(z^K,K)]_{\sim_\Omega}\right)=[(z,\ell)]_\sim,\qquad\forall\ z^K\in (V_\ell)^\mathbb{S}\cup (W_\ell)^{I_\ell}\cup (V_{\ell+1})^\mathbb{S},
	\end{equation*}
	where $(z,n)\sim(w,m)$ with $(z,n)\neq(w,m)$ if and only if $|n-m|=1$ and $z=w\in V_{\max(n,m)}$. See Figure 3 where the broken lines denote different branches.
\end{example}

\begin{remark}
	Let $\Omega\in\tau_s(\mathbb{H})$ be $2$-path-symmetric. Then $\breve{\Omega}$ is the disjoint union of some Riemann domains over $\mathbb{C}$, i.e.
	\begin{equation*}
		\breve{\Omega}=\bigcup_{U\in\mathcal{SCC}(\Omega)}\breve{U},
	\end{equation*}
	where $\mathcal{SCC}(\Omega)$ is the family of slice-connected components of $\Omega$.
\end{remark}

Let $\Omega\in\tau_s(\mathbb{H})$ be $2$-path-symmetric, $\mu\in\breve{\Omega}$ and $r\in(0,r^+_\mu]$. By Proposition \ref{pr-losm} \ref{it-lri},
\begin{equation*}
	\pi_{\breve{\scriptscriptstyle \Omega}} \left[B_{\breve{\Omega}}(\mu,r)\right]
	=\pi_{\breve{\scriptscriptstyle \Omega}}\left[\left[(B_\mathbb{C}\left(z_\mu,r\right))^I\times\{I\}\right]_{\sim_{\scriptscriptstyle \Omega}}\right]
	=B_\mathbb{C}\left(z_\mu,r\right).
\end{equation*}
Let $r_1,r_2\in(0,r)$ with $r_1+r_2\le r$, and let $\nu\in B_{\breve{\scriptscriptstyle \Omega}}(\mu,r_1)$. Then
\begin{equation*}
	z_\nu
	=\pi_{\breve{\scriptscriptstyle \Omega}}(\nu)
	\in \pi_{\breve{\scriptscriptstyle \Omega}} \left[B_{\breve{\Omega}}(\mu,r_1)\right]
	= B_{\mathbb{C}}\left(z_\mu,r_1\right),
\end{equation*}
and
\begin{equation}\label{eq-bbonr}
	\begin{split}
		B_{\breve{\Omega}}(\nu,r_2)
		&=\left[(B_\mathbb{C}\left(z_\nu,r_2\right))^I\times\{I\}\right]_{\sim_{\scriptscriptstyle \Omega}}
		\\&\subset\left[(B_\mathbb{C}\left(z_\mu,r\right))^I\times\{I\}\right]_{\sim_{\scriptscriptstyle \Omega}}
		=B_{\breve{\Omega}}(\mu,r).
	\end{split}
\end{equation}

\begin{prop}\label{pr-r+}
	Let $\Omega\in\tau_s(\mathbb{H})$ be $2$-path-symmetric. The map
	\begin{equation*}\index{$r^+$}
		\begin{split}
			r^+:\ \ \breve{\Omega}\ \ &\xlongrightarrow[\hskip1cm]{}(0,+\infty)=:\mathbb{R}^+,
			\\ \mu\ \ &\shortmid\!\xlongrightarrow[\hskip1cm]{}\quad r_\mu^+,
		\end{split}
	\end{equation*}
	is continuous.
\end{prop}

\begin{proof}
	Let $U\in\tau (\mathbb{R}^+)$ and $\mu\in(r^+)^{-1}(U)$. Then there is $t_1\in(0,r^+_\mu)$ such that $B_\mathbb{R}(r^+_\mu,t_1)\subset U$. Let $\nu\in B_{\breve{\Omega}}(\mu,t_1/2)$. By \eqref{eq-bbonr},
	\begin{equation*}
		B_{\breve{\Omega}}(\nu,r_{\mu}^+-t_1/2)\subset B_{\breve{\Omega}}(\mu,r_\mu^+),
	\end{equation*}
	and $\nu\in(r^+)^{-1}(U)$. Hence $\mu$ is a interior point in $(r^+)^{-1}(U)$, and $(r^+)^{-1}(U)$ is open. It is clear that $r^+$ is continuous.
\end{proof}

Let $\Omega\in\tau_s(\mathbb{H})$ be $2$-path-symmetric, and $U\subset\breve{\Omega}$. Denote
\begin{equation*}\index{$r^+(U),\ r^+_U$}
	r^+(U):=r^+_U:=\inf_{\mu\in U}r_\mu^+.
\end{equation*}
If $U$ is compact, then by \eqref{eq-rml},
\begin{equation*}
	r^+_U=r^+_\nu>0,\qquad\mbox{ for some }\qquad\nu\in U.
\end{equation*}

\section{Some properties of $2$-path-symmetric Riemann domains}\label{sec-properties of 2ps Riemann domains}

In this section, we carefully study a relation between the points in $\Omega$ and the points in the associated $2$-path-symmetric Riemann domain $\breve{\Omega}$. This relation is useful to construct  slice regular functions that  cannot be extended slice regularly to a larger domain, in the sense of Riemann domains. Moreover we prove a number of technical results which are crucial in the proof of the interpolation result in Section \ref{sec-interpolation} and in the proof of the Cartan-Thullen theorem.

Let $(\Omega,\pi)$ be a Riemann domain over $\mathbb{C}$ and $U\subset\Omega$. Denote by
\begin{equation*}\index{$\delta_{\scriptscriptstyle\Omega}$}
	\delta_{\scriptscriptstyle\Omega}(U,\partial \Omega):=\sup\left\{r\ge 0:\exists\ B'_\Omega (z,r),\ \forall\ z\in U\right\},
\end{equation*}
the boundary distance of $U$ on $\Omega$, where $B'_\Omega (z,r)$ is the unique subset (if it exists) of $\Omega$ such that $z\in B'_\Omega (z,r)$ and $\pi:B'_\Omega (z,r)\rightarrow B_\mathbb{C} (\pi(z),r)$ is a homeomorphism.
Let
\begin{equation*}\index{$U^{\land}_\Omega$, $\widehat{U}_\Omega$}
	U^{\land}_\Omega:=\widehat{U}_\Omega:=\left\{z\in\Omega:|f(z)|\le \sup_{w\in U}|f(w)|,\ f\in\mathscr{O}(\Omega)\right\}
\end{equation*}
be the holomorphically convex hull, where $\mathscr{O}(\Omega)$\index{$\mathscr{O}(\Omega)$} is the class of holomorphic functions on $\Omega$. Note that, by definition
\begin{equation}\label{eq-usw}
	U\subset \widehat{U}_\Omega.
\end{equation}
Recall each ($1$-dimensional) Riemann domain over $\mathbb{C}$ is a domain of holomorphy. By \cite[Lemma 7.5.5]{Noguchi2016001} we have
\begin{equation*}
	\delta_\Omega(U,\partial \Omega)=	\delta_\Omega(\widehat{U}_\Omega,\partial \Omega).
\end{equation*}

By Theorem \ref{thm-Riemann doamin} we know that if $\Omega\subset\mathbb{H}$ is a $2$-path-symmetric st-domain then $(\breve{\Omega},\pi_{\scriptscriptstyle{\breve{\Omega}}})$ is a Riemann domain over $\mathbb{C}$. It follows from
\begin{equation*}
	\pi_{\scriptscriptstyle{\breve{\Omega}}}\left[B_{\scriptscriptstyle{\breve{\Omega}}}(\mu,r)\right]=B_\mathbb{C}(z_\mu,r)=\pi_{\scriptscriptstyle{\breve{\Omega}}}\left[B'_{\scriptscriptstyle{\breve{\Omega}}}(\mu,r)\right],\qquad\forall\ \mu\in\breve{\Omega}\quad\mbox{and}\quad r\le r_\mu^+,
\end{equation*}
that
\begin{equation*}
	B_{\scriptscriptstyle{\breve{\Omega}}}(\mu,r)=B'_{\scriptscriptstyle{\breve{\Omega}}}(\mu,r),\qquad\forall\ \mu\in\breve{\Omega}\quad\mbox{and}\quad r\le r_\mu^+,
\end{equation*}
where $B_{\scriptscriptstyle{\breve{\Omega}}}(\mu,r)$ is defined by \eqref{eq-balldef}.

Note that sometimes, for graphical reasons, it is convenient to write $U^{\land}_\Omega$ instead of $\widehat{U}_\Omega$, like in the result below:
\begin{prop}\label{pr-K Ur holo convex eq K Ur}
	Let $\Omega$ be a Riemann domain over $\mathbb{C}$, $K\subset\Omega$ be compact with $\hat{K}_\Omega=K$, $U\in\tau(\Omega)$ and $r>0$. Then $K\cap U^r$ is compact in $U$ with
	\begin{equation}\label{eq-kcu}
		(K\cap U^r)^{\land}_U=K\cap U^r,
	\end{equation}
	where
	\begin{equation*}\index{$U^r$}
		U^r:=\{z\in U:\delta_U(z,\partial U)\ge r\}.
	\end{equation*}
\end{prop}

\begin{proof}
	Since
	\begin{equation*}
		\delta_U\left((K\cap U^r)^{\land}_U,\partial U\right)=\delta_U(K\cap U^r,\partial U)\ge \delta_U(U^r,\partial U)= r,
	\end{equation*}
	we have
	\begin{equation}\label{eq-kcur}
		(K\cap U^r)^{\land}_U\subset U^r.
	\end{equation}

	Let $z\in U\backslash K$. Then there is $f\in\mathscr{O}(\Omega)$ such that
	\begin{equation*}
		|f(z)|\ge \sup_{w\in K}|f(w)|\ge \sup_{w\in K\cap U^r}|f(w)|.
	\end{equation*}
	It implies that $z\notin (K\cap U^r)^{\land}_U$. Therefore $(K\cap U^r)^{\land}_U\subset K$, and then \eqref{eq-kcu} holds by \eqref{eq-kcur} and \eqref{eq-usw}.
\end{proof}

The $\sigma$-ball with center $q\in\mathbb{H}$ and radius $s>0$ is introduced in \cite{Gentili2012001} as follows:
\begin{equation*}
	\Sigma(q,s):=\{q\in\mathbb{H}:\sigma(p,q)<s\}
\end{equation*}
and is induced by the so-called $\sigma$-distance
\begin{equation*}
	\sigma(p,q):=\begin{cases}
		|q-p|,\qquad &\exists\ I\in\mathbb{S},\ s.t.\ p,q\in\mathbb{C}_I,\\
		\sqrt{(Re(q-p))^2+|Im(q)|^2+|Im(p)|^2},\qquad &\mbox{otherwise}.
	\end{cases}
\end{equation*}
Let $p\in\mathbb{H}$ and $U,V\subset\mathbb{H}$. Set
\begin{equation*}
	\sigma^{p}_{V}:=\sigma(p,V):=\inf_{q\in V}\left(\sigma(p,q)\right),
	\qquad\mbox{and}\qquad
	\sigma^{U}_{V}:=\sigma(U,V):=\inf_{p\in U, q\in V}\left(\sigma(p,q)\right).
\end{equation*}

\begin{prop}\label{pr-lbes}
	Let $\Omega\in\tau_s(\mathbb{H})$ be $2$-path-symmetric, $I\in\mathbb{S}$, $p\in\Omega_I$ and $U\subset\Omega_I$. Then
	\begin{equation}\label{eq-sqmbo}
		\sigma(p,\mathbb{H}\backslash \Omega)=\delta_{I}\left(p,\mathbb{C}_I\backslash\Omega_I\right),
	\end{equation}
	and
	\begin{equation}\label{eq-suh}
		\sigma(U,\mathbb{H}\backslash \Omega)=\delta_{I}(U,\mathbb{C}_I\backslash\Omega_I).
	\end{equation}
\end{prop}

\begin{proof}
	$\sigma^p_{\mathbb{H}\backslash \Omega}\le\delta_{I}\left(p,\mathbb{C}_I\backslash\Omega_I\right)$ holds directly by $B_I(p,r)\subset\Sigma(p,r)$ for each $r>0$.
	
	Conversely, suppose $B_I(p,r)\subset\Omega$ for some fixed $r>0$. If $B_I(p,r)\cap\mathbb{R}=\varnothing$, then
	\begin{equation*}
		\Sigma(p,r)=B_I(p,r)\subset\Omega.
	\end{equation*}
	Since the choice of $r>0$ with $B_I(p,r)\subset\Omega$ is arbitrary,
	\begin{equation}\label{eq-sbpr}
		\sigma(p,\mathbb{H}\backslash \Omega)\ge\delta_{I}\left(p,\mathbb{C}_I\backslash\Omega_I\right)
	\end{equation}
	and \eqref{eq-sqmbo} holds. Otherwise $B_I(p,r)\cap\mathbb{R}\neq\varnothing$. Then for each $x+yJ\in\Sigma(p,r)$, it is each to check that
	\begin{equation*}
		\mathcal{L}_x^{x+yi}\in\mathscr{P}_0(\mathbb{C})\qquad\mbox{and}\qquad (\mathcal{L}_x^{x+yi})^I,(\mathcal{L}_x^{x+yi})^{-I}\subset\Omega.
	\end{equation*}
	Since $\Omega$ is $2$-path-symmetric, $(\mathcal{L}_x^{x+yi})^J\subset \Omega$. Therefore
	\begin{equation*}
		x+yJ=(\mathcal{L}_x^{x+yi})^J(1)\in\Omega.
	\end{equation*}
	Hence $\Sigma(p,r)\subset\Omega$. Since the choice of $r>0$ with $B_I(p,r)\subset\Omega$ is arbitrary, \eqref{eq-sbpr} holds. Therefore \eqref{eq-sqmbo} holds. Then
	\begin{equation*}
		\sigma(U,\mathbb{H}\backslash \Omega)
		=\inf_{q\in U}\sigma(q,\mathbb{H}\backslash \Omega)
		=\inf_{q\in U}\delta_{I}(q,\mathbb{C}_I\backslash\Omega_I)
		=\delta_{I}(U,\mathbb{C}_I\backslash\Omega_I),
	\end{equation*}
	i.e. \eqref{eq-suh} holds.
\end{proof}

\begin{prop}\label{pr-lmi}
	Let $\Omega\in\tau_s(\mathbb{H})$ be $2$-path-symmetric, $\mu\in\breve{\Omega}$ and $I,J\in\mathbb{S}_\mu$. Then
	\begin{equation}\label{eq-lxy}
		\left[\left(\overline{z_\mu}^{-I},-I\right)\right]_{\sim_\Omega}=\left[\left(\overline{z_\mu}^{-J},-J\right)\right]_{\sim_\Omega}.
	\end{equation}
\end{prop}

\begin{proof}
	If $z_\mu\in\Omega_{2ps}$, then by Remark \ref{rmk-2ps} \ref{it-o2c} we have $\overline{z_\mu}\in\Omega_{2ps}$. It implies from \eqref{eq-lir} that $(\overline{z_\mu}^{-I},-I)\sim_{\scriptscriptstyle \Omega}(\overline{z_\mu}^{-J},-J)$. Hence \eqref{eq-lxy} holds. Otherwise, $z_\mu\notin\Omega_{2ps}$. By \eqref{eq-msn}, we have $\mathbb{S}_\mu=\{I_\mu\}$ and then $I=J=I_\mu$. It follows from $\overline{z_\mu}^{-I}=z_\mu^I\in\Omega$ that \eqref{eq-lxy} holds.
\end{proof}

For each $\mu\in\breve{\Omega}$, we choose a fixed $J_\mu\in\mathbb{S}_\mu$. Define
\begin{equation}\label{eq-conjugation}\index{${\Conj}_{\breve{\scriptscriptstyle\Omega}}$, $\overline{\mu}$}
	\begin{split}
		{\Conj}_{\breve{\scriptscriptstyle\Omega}}:\ \breve{\Omega}\ &\xlongrightarrow[\hskip1cm]{}\ \breve{\Omega},
		\\ \mu\ &\shortmid\!\xlongrightarrow[\hskip1cm]{}\ \overline{\mu}:=\left[\left(\overline{z_\mu}^{-J_\mu}, -J_\mu\right)\right]_{\sim_\Omega}=\left[\left(z_\mu^{J_\mu}, -J_\mu\right)\right]_{\sim_\Omega}.
	\end{split}
\end{equation}
According to Proposition \ref{pr-lmi}, ${\Conj}_{\breve{\scriptscriptstyle\Omega}}$ does not depend on the choice of $J_\mu\in\mathbb{S}_\mu$. It is easy to check that
\begin{equation}\label{eq-zomoz}
	z_{\overline{\mu}}=\overline{z_\mu},\qquad\mbox{and}\qquad \mathbb{S}_{\overline{\mu}}=-\mathbb{S}_{\mu}.
\end{equation}
If $z_\mu=\Omega_{2ps}$, then by \eqref{eq-sim},
\begin{equation*}
	\overline{\mu}=[(\overline{z_\mu}^I,I)]_{\sim_\Omega},\qquad\forall\ I\in\mathbb{S}_\mu=\mathbb{S}.
\end{equation*}

\begin{prop}\label{pr-rrcb}
	Let $\Omega\in\tau_s(\mathbb{H})$ be $2$-path-symmetric, $\mu\in\Omega$, and $U\subset\breve{\Omega}$. Then
	\begin{equation*}
		r_\mu^+=r_{\overline{\mu}}^+,
		\qquad\mbox{and}\qquad
		r_U^+=r_{\Conj_{\breve{\scriptscriptstyle\Omega}}(U)}^+.
	\end{equation*}
\end{prop}

\begin{proof}
	Let $I\in\mathbb{S}_\mu^+$. By \eqref{eq-rml}, $r_\mu^+=r_\mu^I>0$. Then $-I\in -\mathbb{S}_\mu=\mathbb{S}_{\overline{\mu}}$, and
	\begin{equation*}
		B_{-I}\left(z_{\overline{\mu}}^{-I},r_\mu^+\right)=B_{-I}\left(\overline{z_\mu}^{-I},r_\mu^+\right)=B_I(z_\mu^I,r_\mu^I)\subset\Omega.
	\end{equation*}
	It implies that
		$r_{\overline{\mu}}^+
		\ge r_{\overline{\mu}}^{-I}
		\ge r_\mu^+$.
	Similarly,
		$r_\mu^+
		=r_{\overline{\overline{\mu}}}^+
		\ge r_{\overline{\mu}}^+$.
	Therefore $r_\mu^+=r_{\overline{\mu}}^+$. Then
	\begin{equation*}
		r_U^+=\inf_{\nu\in U} r_\nu^+
		=\inf_{\nu\in U}r_{\overline{\nu}}^+
		=\inf_{\overline{\nu}\in \Conj_{\breve{\scriptscriptstyle\Omega}}(U)}r_{\overline{\nu}}^+=r_{\Conj_{\breve{\scriptscriptstyle\Omega}}(U)}^+.
	\end{equation*}
\end{proof}

\begin{prop}
	Let $\Omega\in\tau_s(\mathbb{H})$ be $2$-path-symmetric. Then $\Conj_{\breve{\scriptscriptstyle\Omega}}:\, \breve{\Omega}\to\breve{\Omega}$ is a homeomorphism.
\end{prop}

\begin{proof}
	$\Conj_{\breve{\scriptscriptstyle\Omega}}$ is bijective directly by $\Conj_{\breve{\scriptscriptstyle\Omega}}^2=\id_{\breve{\scriptscriptstyle\Omega}}$, where $\id_{\breve{\scriptscriptstyle\Omega}}$ is the identity map on $\breve\Omega$. Let $U\in\tau(\breve{\Omega})$ and $\mu\in \Conj_{\breve{\scriptscriptstyle\Omega}}^{-1}(U)$. It implies that
	\begin{equation*}
		\overline{\mu}={\Conj}_{\breve{\scriptscriptstyle\Omega}}(\mu)={\Conj}_{\breve{\scriptscriptstyle\Omega}}^{-1}(\mu)\in
		{\Conj}_{\breve{\scriptscriptstyle\Omega}}^{-1}\left[{\Conj}_{\breve{\scriptscriptstyle\Omega}}^{-1}(U)\right]= U,
	\end{equation*}
	and there is $r\in(0,r_\mu^+)$ such that $B_{\breve{\scriptscriptstyle \Omega}}(\overline{\mu},r)\subset U$. Let $J_\mu\in\mathbb{S}_\mu$. Then
	\begin{equation*}
		-J_\mu\in-\mathbb{S}_\mu=\mathbb{S}_{\overline{\mu}}
	\end{equation*}
	and
	\begin{equation*}
		\begin{split}
			B_{\breve{\scriptscriptstyle \Omega}}(\mu,r)
			&=\left[B_{J_\mu}\left(z_\mu^{J_\mu},r\right)\times \{J_\mu\}\right]_{\sim_\Omega}
			={\Conj}_{\breve{\scriptscriptstyle\Omega}}^2\left(\left[B_{J_\mu}\left(z_\mu^{J_\mu},r\right)\times\{J_\mu\}\right]_{\sim_\Omega}\right)
			\\&={\Conj}_{\breve{\scriptscriptstyle\Omega}}\left(\left[B_{J_\mu}\left(\overline{z_\mu}^{-J_\mu},r\right)\times \{-J_\mu\}\right]_{\sim_\Omega}\right)
			\\&={\Conj}_{\breve{\scriptscriptstyle\Omega}}\left(\left[B_{-J_\mu}\left(z_{\overline{\mu}}^{-J_\mu},r\right)\times \{-J_\mu\}\right]_{\sim_\Omega}\right)
			={\Conj}_{\breve{\scriptscriptstyle\Omega}}\left(B_{\breve{\scriptscriptstyle \Omega}}(\overline{\mu},r)\right)
			\\&\subset {\Conj}_{\breve{\scriptscriptstyle\Omega}}(U)
			={\Conj}_{\breve{\scriptscriptstyle\Omega}}^{-1}(U).
		\end{split}
	\end{equation*}
	Therefore, $\mu$ is an interior point in ${\Conj}_{\breve{\scriptscriptstyle\Omega}}^{-1}(U)$, and ${\Conj}_{\breve{\scriptscriptstyle\Omega}}^{-1}(U)$ is open. It is clear that ${\Conj}_{\breve{\scriptscriptstyle\Omega}}$ is continuous, so is ${\Conj}_{\breve{\scriptscriptstyle\Omega}}^{-1}={\Conj}_{\breve{\scriptscriptstyle\Omega}}$. In summary, ${\Conj}_{\breve{\scriptscriptstyle\Omega}}$ is a homeomorphism.
\end{proof}

Let $\Omega\in\tau_s(\mathbb{H})$ be $2$-path-symmetric and $t\ge0$ and set
\begin{equation}\label{eq-distance}\index{$\Omega^t$, $\Omega^{t^\circ}$, $\breve{\Omega}_{+}^{t}$, $\breve{\Omega}_{+}^{t^\circ}$}
	\Omega^t:=\{q\in\Omega:t\le \sigma(q,\mathbb{H}\backslash\Omega)\},
	\qquad
	\Omega^{t^\circ}:=\{q\in\Omega:t < \sigma(q,\mathbb{H}\backslash\Omega)\}.
\end{equation}
\begin{equation*}
	\breve{\Omega}_{+}^{t}:=\{\mu\in\breve{\Omega}:t\le r^+_\mu\},
	\qquad\mbox{and}\qquad
	\breve{\Omega}_{+}^{t^\circ}:=\{\mu\in\breve{\Omega}:t < r^+_\mu\}.
\end{equation*}

	\begin{prop}\label{pr-omegar}
		Let $\Omega\in\tau_s(\mathbb{H})$ be $2$-path-symmetric, $I\in\mathbb{}S$, $w\in\Omega_I$ with $\mathbb{S}_{[(w,I)]_{\sim_\Omega}}^+=\mathbb{S}$. Then
		\begin{equation}\label{eq-omegar}
			\Omega^r\cap[w]_{\approx_\Omega}=\varnothing,\qquad\forall\ r>\sigma(w,\mathbb{H}\backslash\Omega).
		\end{equation}
	\end{prop}
	
	\begin{proof}
		Let $u\in[w]_{\sim\Omega}$. Then $u\approx_{\scriptscriptstyle\Omega} w=z^I$, where $z:=\Psi_I^i(w)$. According to Proposition \ref{pr-losmt} \ref{it-lp}, $u=z^K$ for some $K\in\mathbb{S}$. By \eqref{eq-msmb}, $\mathbb{S}_\mu=\mathbb{S}_\mu^+=\mathbb{S}$, where $\mu:=[(w,I)]_{\sim_\Omega}=[(z^I,I)]_{\sim_\Omega}$. It follows from $z_\mu=z$, \eqref{eq-rmi}, \eqref{eq-rml} and \eqref{eq-sqmbo} that
		\begin{equation*}
			\begin{split}
				\sigma(u,\mathbb{H}\backslash\Omega)&=\delta_K\left(z^K,\mathbb{C}_K\backslash\Omega_K\right)=r_\mu^K
				\\&=r_\mu^+=r_\mu^I=\delta_I\left(z^I,\mathbb{C}_I\backslash\Omega_I\right)=\sigma(w,\mathbb{H}\backslash\Omega)<r.
			\end{split}
		\end{equation*}
		According to \eqref{eq-distance}, $u\notin\Omega^r$. Then \eqref{eq-omegar} holds.
	\end{proof}

Let $\Omega\in\tau_s(\mathbb{H})$ be $2$-path-symmetric, $\mathbb{S}'\subset\mathbb{S}$, $t\ge 0$ and $\rho\in\{t,t^\circ\}$. Denote
\begin{equation}\label{eq-definitions}\index{$\psi_{\scriptscriptstyle\Omega}^{\mathbb{S}'}$, $\varphi_{\scriptscriptstyle\breve{\Omega}}^\rho$, $\phi_{\scriptscriptstyle\Omega}^\rho$, $\phi_{\breve{\scriptscriptstyle\Omega}}^\rho$}
	\begin{split}
		\psi_{\scriptscriptstyle\Omega}^{\mathbb{S}'}:\ P(\Omega)\ &\xlongrightarrow[\hskip1cm]{}\ P(\breve{\Omega}),
		\\
		U\quad &\shortmid\!\xlongrightarrow[\hskip1cm]{}\
		\left[\bigcup_{I\in\mathbb{S}'} U_I\times\{I\}\right]_{\sim_\Omega},
		\\
		\varphi_{\scriptscriptstyle\breve{\Omega}}^\rho:\ P(\breve{\Omega})\ &\xlongrightarrow[\hskip1cm]{}\ P(\Omega),
		\\ V\quad &\shortmid\!\xlongrightarrow[\hskip1cm]{}\ \left\{q\in\Omega^\rho:\psi_{\scriptscriptstyle\Omega}^{\mathbb{S}}(\{q\})\cap V\neq\varnothing\right\},
		\\
		\phi_{\scriptscriptstyle\Omega}^\rho:\ P(\Omega)\ &\xlongrightarrow[\hskip1cm]{}\ P(\Omega),
		\\ U\quad &\shortmid\!\xlongrightarrow[\hskip1cm]{}\
			 \left[U\cup\left(U\cap\Omega^{2ps}\right)^\sim\right]\cap\Omega^\rho
			=[U]_{\approx_{\Omega}}^\cup\cap\Omega^\rho,
		\\
		\phi_{\breve{\scriptscriptstyle\Omega}}^\rho:\ P(\breve{\Omega})\ &\xlongrightarrow[\hskip1cm]{}\ P(\breve{\Omega}),
		\\ V\quad &\shortmid\!\xlongrightarrow[\hskip1cm]{}\ \left[V\cup{\Conj}_{\scriptscriptstyle\breve{\Omega}}(V)\right]\cap\breve{\Omega}^\rho_{+},
	\end{split}
\end{equation}
where
\begin{equation*}\index{$[U]_{\approx_\Omega}^\cup$}
	[U]_{\approx_\Omega}^\cup:=\bigcup_{V\in [U]_{\approx_\Omega}}V=\{q\in\Omega: q\in [p]_{\approx_\Omega},\mbox{ for some }p\in U\},
\end{equation*}
	and $\left(U\cap\Omega^{2ps}\right)^\sim$ \index{$U^\sim$} denotes the axially symmetric completion of $U\cap\Omega^{2ps}$.
By definition, $\phi_{\scriptscriptstyle\Omega}^\rho$ and $\phi_{\breve{\scriptscriptstyle\Omega}}^\rho$ are projection, i.e. $(\phi_{\scriptscriptstyle\Omega}^\rho)^2=\phi_{\scriptscriptstyle\Omega}^\rho$.

Let $x\in\Omega_{\mathbb{R}}$. By \eqref{eq-omr}, $x\in\Omega_\mathbb{R}\subset\Omega_{2ps}$. According to \eqref{eq-sim},
\begin{equation*}
	\left(x^I,I\right)\sim_{\Omega}\left(x^J,J\right),\qquad\forall\ I,J\in\mathbb{S}.
\end{equation*}
Thence for each $J\in\mathbb{S}$,
\begin{equation*}
	\begin{split}
		\psi_{\scriptscriptstyle \Omega}^\mathbb{S}(\{x\})
		=&\left[\bigcup_{I\in\mathbb{S}}\{x\}_I\times\{I\}\right]_{\sim_\Omega}
		=\left[\bigcup_{I\in\mathbb{S}}\{(x,I)\}\right]_{\sim_\Omega}
		\\=&\bigcup_{I\in\mathbb{S}}\left\{\left[(x^I,I)\right]_{\sim_\Omega}\right\}
		=\left[(x^J,J)\right]_{\sim_\Omega}.
	\end{split}
\end{equation*}
It implies that
\begin{equation}\label{eq-psoz}
	\psi_{\scriptscriptstyle \Omega}^\mathbb{S}(\{z^I\})=\left\{\left[(z^I,I)\right]_{\sim_\Omega}\right\}\bigcup\left\{\left[(\overline{z}^{-I},-I)\right]_{\sim_\Omega}\right\},\qquad\forall\ z^I\in\Omega.
\end{equation}
Let $z\in\Omega_{2ps}$ and $I,J\in\mathbb{S}$. Then it follows from
\begin{equation*}
	[(z^I,I)]_{\sim_\Omega}=[(z^J,J)]_{\sim_\Omega},\qquad\mbox{and}\qquad [(\overline{z}^{-I},-I)]_{\sim_\Omega}=[(z^{-J},-J)]_{\sim_\Omega}
\end{equation*}
that
\begin{equation}\label{eq-psoms}
	\psi_{\scriptscriptstyle \Omega}^\mathbb{S}(\{z^I\})=\psi_{\scriptscriptstyle \Omega}^\mathbb{S}(\{z^J\}).
\end{equation}

\begin{prop}\label{pr-distanse}
	Let $\Omega\in\tau_s(\mathbb{H})$ be $2$-path-symmetric, $t>0$, and $\rho\in\{t,t^\circ\}$. Then
	\begin{equation*}
		\mu:=\left[\left(w^I,I\right)\right]_{\sim_\Omega}\in\breve{\Omega}_+^\rho,\qquad\forall\ w^I\in\Omega^\rho.
	\end{equation*}
	Moreover,
	\begin{equation*}
		\psi_{\scriptscriptstyle \Omega}^\mathbb{S}(\Omega^\rho)\subset\breve{\Omega}_+^\rho.
	\end{equation*}
\end{prop}

\begin{proof}
	(i) Let $w^I\in\Omega^\rho$. By \eqref{eq-uxu} and \eqref{eq-xx}, $z_\mu=z_{(w^I,I)}=w$. There is
	\begin{equation}\label{eq-bibc}
		\varepsilon\in
		\begin{cases}
			\{0\},\qquad &\rho=t,
			\\(0,+\infty),\qquad &\rho=t^\circ,
		\end{cases}
	\end{equation}
	such that
	\begin{equation*}
		B_I\left(z_\mu^I,t+\varepsilon\right)
		=B_I\left(w^I,t+\varepsilon\right)
		\subset\Omega_I.
	\end{equation*}
	According to \eqref{eq-rms} and \eqref{eq-rml}, $r_\mu^+\ge r_\mu^I\ge t+\varepsilon$. It implies that $\mu\in\breve{\Omega}_+^\rho$.
	
	(ii) Let $\xi^J\in \Omega^\rho$. By \eqref{eq-psoz}
	\begin{equation*}
		\psi_{\scriptscriptstyle \Omega}^\mathbb{S}(\{\xi^J\})=\bigg\{\big[\big(\xi^J,J\big)\big]_{\sim_\Omega}\bigg\}\bigcup\bigg\{\big[\big(\overline{\xi}^{-J},-J\big)\big]_{\sim_\Omega}\bigg\},\qquad\forall\ z^J\in\Omega.
	\end{equation*}
	It implies by $\overline{\xi}^{-J}=\xi^J\in\Omega^\rho$ and (i) that
	\begin{equation*}
		\big[\big(\xi^J,J\big)\big]_{\sim_\Omega},\big[\big(\overline{\xi}^{-J},J\big)\big]_{\sim_\Omega}\in\breve{\Omega}^\rho_+.
	\end{equation*}
	Therefore
	\begin{equation*}
		\psi_{\scriptscriptstyle \Omega}^\mathbb{S}(\Omega^\rho)=\bigcup_{\xi^J\in \Omega^\rho}\psi_{\scriptscriptstyle \Omega}^\mathbb{S}\left(\{\xi^J\}\right)\subset\breve{\Omega}^\rho_+.
	\end{equation*}
\end{proof}
The technical result in the next lemma will be useful in the next section.
\begin{lem}
	Let $\Omega\subset\tau_s(\mathbb{H})$ be $2$-path-symmetric, $t>0$ and $\rho\in\{t,t^\circ\}$. Then
	\begin{equation}\label{eq-cvsrc}
		\begin{cases}
			\varphi_{\scriptscriptstyle\breve{\Omega}}^\rho\circ \psi_{\scriptscriptstyle\Omega}^{\mathbb{S}}=\phi_{\scriptscriptstyle\Omega}^\rho,\\
			\psi_{\scriptscriptstyle\Omega}^{\mathbb{S}}\circ\varphi_{\scriptscriptstyle\breve{\Omega}}^\rho=\phi_{\breve{\scriptscriptstyle\Omega}}^\rho.
		\end{cases}
	\end{equation}
\end{lem}

\begin{proof}
	Let $U\in P(\Omega)$ and $V\in P(\breve{\Omega})$.
	
	(i) Suppose $p\in\varphi_{\scriptscriptstyle\breve{\Omega}}^\rho\circ \psi_{\scriptscriptstyle\Omega}^{\mathbb{S}}(U)$. Then
	\begin{equation}\label{eq-vps}
		\begin{split}
			\varphi_{\scriptscriptstyle\breve{\Omega}}^\rho\circ \psi_{\scriptscriptstyle\Omega}^{\mathbb{S}}(U)
			=&\varphi_{\scriptscriptstyle\breve{\Omega}}^\rho\left(\left[\bigcup_{I\in\mathbb{S}} U_I\times\{I\}\right]_{\sim_\Omega}\right)
			\\=&\left\{q\in\Omega^\rho:\psi_{\scriptscriptstyle\Omega}^{\mathbb{S}}(\{q\})\cap\left[\bigcup_{I\in\mathbb{S}} U_I\times\{I\}\right]_{\sim_\Omega}\neq\varnothing\right\}.
		\end{split}
	\end{equation}
	implies that $p\in\Omega^\rho$ and
	\begin{equation*}
		\left[\bigcup_{I\in\mathbb{S}} \{p\}_I\times\{I\}\right]_{\sim_\Omega}\bigcap\left[\bigcup_{I\in\mathbb{S}} U_I\times\{I\}\right]_{\sim_\Omega}=\psi_{\scriptscriptstyle\Omega}^{\mathbb{S}}(\{p\})\cap\left[\bigcup_{I\in\mathbb{S}} U_I\times\{I\}\right]_{\sim_\Omega}\neq\varnothing,
	\end{equation*}
	i.e. there is $w\in U$ and $J,K\in\mathbb{S}_\mu$ such that
	\begin{equation*}
		\mu:=\left[(p,J)\right]_{\sim_\Omega}=\left[(w,K)\right]_{\sim_\Omega}\in \left[\bigcup_{I\in\mathbb{S}} {\{p\}}_I\times\{I\}\right]_{\sim_\Omega}\bigcap\left[\bigcup_{I\in\mathbb{S}} U_I\times\{I\}\right]_{\sim_\Omega}.
	\end{equation*}
	It follows from \eqref{eq-zaia} and $(p,J)\sim_{\scriptscriptstyle\Omega}(w,K)$ that
	\begin{equation*}
		p=z_{\scriptscriptstyle(p,J)}^{I_{(p,J)}}\approx_{\scriptscriptstyle\Omega} z_{\scriptscriptstyle(w,k)}^{I_{(w,k)}}=w.
	\end{equation*}
	It implies that $p\in[w]_{\approx_\Omega}\subset[U]_{\approx_\Omega}$. Hence $p\in [U]_{\approx_\Omega}\cap\Omega^\rho=\varphi_{\scriptscriptstyle\Omega}^\rho(U)$. Then
	\begin{equation}\label{eq-vbso}
		\varphi_{\breve{\scriptscriptstyle \Omega}}^\rho\circ\psi_{\scriptscriptstyle \Omega}^\mathbb{S}(U)\subset\varphi_{\scriptscriptstyle\Omega}^\rho(U).
	\end{equation}
	
	(ii) Suppose $p\in\phi_{\scriptscriptstyle\Omega}(U)=[U]_{\approx_\Omega}\cap\Omega^\rho$. Then there is $z^K\in U$ with $p\approx_{\scriptscriptstyle\Omega} z^K$. By Proposition \ref{pr-losmt} \ref{it-lp}, there is $J\in\mathbb{S}$ such that $p=z^J$. According to \eqref{eq-o2ps}, $J=K$ or $z\in\Omega_{2ps}$. Then by \eqref{eq-sim}, $(p,J)=(z^J,J) \sim_{\scriptscriptstyle\Omega} (z^K,K)$. Therefore
	\begin{equation*}
		\left[(p,J)\right]_{\sim_\Omega}=\left[(z^K,K)\right]_{\sim_\Omega}\in\left[U_K\times\{K\}\right]_{\sim_\Omega}\subset\left[\bigcup_{I\in\mathbb{S}} U_I\times\{I\}\right]_{\sim_\Omega}.
	\end{equation*}
	Hence
	\begin{equation*}
		\left[(p,J)\right]_{\sim_\Omega}\in \left[\bigcup_{I\in\mathbb{S}} {\{p\}}_I\times\{I\}\right]_{\sim_\Omega}\bigcap\left[\bigcup_{I\in\mathbb{S}} U_I\times\{I\}\right]_{\sim_\Omega}\neq\varnothing.
	\end{equation*}
	By \eqref{eq-vps}, $p\in\varphi_{\scriptscriptstyle\breve{\Omega}}^\rho\circ \psi_{\scriptscriptstyle\Omega}^{\mathbb{S}}(U)$. Then $\phi_{\scriptscriptstyle\Omega}(U)\subset\varphi_{\scriptscriptstyle\breve{\Omega}}^\rho\circ \psi_{\scriptscriptstyle\Omega}^{\mathbb{S}}(U)$. According to \eqref{eq-vbso},
	\begin{equation}\label{eq-vsb}
		\phi_{\scriptscriptstyle\Omega}=\varphi_{\scriptscriptstyle\breve{\Omega}}^\rho\circ \psi_{\scriptscriptstyle\Omega}^{\mathbb{S}}.
	\end{equation}
	
	(iii) Suppose that
	\begin{equation*}
		\mu\in\psi_{\scriptscriptstyle\Omega}^\mathbb{S}\circ\varphi_{\breve{\scriptscriptstyle\Omega}}^\rho(V)
		=\psi_{\scriptscriptstyle\Omega}^\mathbb{S}\left\{q\in\Omega^\rho:\psi_{\scriptscriptstyle\Omega}^\mathbb{S}(\{q\})\cap V\neq\varnothing\right\}.
	\end{equation*}
	By \eqref{eq-psoz} and Proposition \ref{pr-distanse}, there is $z^I\in\Omega^\rho$ with $\psi_{\scriptscriptstyle\Omega}^\mathbb{S}(\{z^I\})\cap V\neq\varnothing$ such that
	\begin{equation*}
		\begin{split}
			\mu
			\in\psi_{\scriptscriptstyle\Omega}^\mathbb{S}(\{z^I\})
			&=\left\{\left[(z^I,I)\right]_{\sim_\Omega}\right\}\bigcup\left\{\left[(\overline{z}^{-I},-I)\right]_{\sim_\Omega}\right\}
			\\&\subset(V\cup{\Conj}_{\breve{\scriptscriptstyle\Omega}}(V))\cap\breve{\Omega}_+^\rho
			=\phi_{\breve{\scriptscriptstyle\Omega}}^\rho(U).
		\end{split}
	\end{equation*}
	It implies that
	\begin{equation}\label{eq-psopr}
		\psi_{\scriptscriptstyle\Omega}^\mathbb{S}\circ\varphi_{\breve{\scriptscriptstyle\Omega}}^\rho(V)
		\subset \phi_{\breve{\scriptscriptstyle\Omega}}^\rho(U).
	\end{equation}
	
	(iv) Suppose that $\mu\in\varphi_{\scriptscriptstyle\Omega}^\rho(V)	=(V\cap{\Conj}_{\breve{\scriptscriptstyle\Omega}}(V))\cap\Omega^\rho$. Then
	\begin{equation*}
		\nu:={\Conj}_{\breve{\scriptscriptstyle\Omega}}^\lambda (\mu)\in V,\qquad\mbox{for some}\qquad\lambda\in\{0,1\}.
	\end{equation*}
	It follows from $\mu\in\breve{\Omega}^\rho_+$ that $r_\mu^+\ge t+\varepsilon$ for some $\varepsilon$ satisfying \eqref{eq-bibc}. Let $I\in\mathbb{S}_\mu^+$. Then $r_\mu^+=r_\mu^I$ and $B_I\left(z_\mu^I,r_\mu^I\right)\subset\Omega$. It implies that
	\begin{equation}\label{eq-zznm}
		z_\nu^J
		=z_{{\Conj}_{\breve{\scriptscriptstyle\Omega}}^\lambda (\mu)}^{J}
		={\Conj}_\mathbb{C}^\lambda(z_\mu)^{(-1)^\lambda I}
		=z_\mu^I
		\in\Omega^\rho,
		\qquad\mbox{where}\qquad J=(-1)^\lambda I.
	\end{equation}
	Hence
	\begin{equation*}
		\nu
		=\left[\left({\Conj}_{\breve{\Omega}}^\lambda(z_\mu)^{(-1)^\lambda I},(-1)^\lambda I\right)\right]_{\sim_\Omega}
		=\left[\left(z_\nu^J,J\right)\right]_{\sim_\Omega}
		\in \left[\bigcup_{I\in\mathbb{S}}\left\{z_\nu^J\right\}_I\times\{I\}\right]_{\sim_\Omega}.
	\end{equation*}
	Therefore
	\begin{equation*}
		\nu\in\left[\bigcup_{I\in\mathbb{S}}\left\{z_\nu^J\right\}\times\{I\}\right]_{\sim_\Omega}\cap V\neq\varnothing.
	\end{equation*}
	According to \eqref{eq-zznm} and Proposition \ref{pr-distanse},
	\begin{equation*}
		\mu=\left[\left(z_\mu^J,J\right)\right]_{\sim_\Omega}\in\breve{\Omega}_+^\rho.
	\end{equation*}
	It implies that $z_\nu^J\in\varphi_{\breve{\scriptscriptstyle \Omega}}^\rho(V)$. By \eqref{eq-zznm}, $z_\mu^I=z_\nu^J$. Then
	\begin{equation*}
		\begin{split}
			\mu
			&=\left[(z_\mu^I,I)\right]_{\sim_\Omega}
			=\left[\left(z_\nu^J,I\right)\right]_{\sim_\Omega}
			\in\bigcup_{L\in\mathbb{S},\ z_\mu^J\in\mathbb{C}_L}\left[\left(z_\nu^J,L\right)\right]_{\sim_\Omega}
			\\&=\left[\bigcup_{L\in\mathbb{S}}\left\{z_\nu^J\right\}_L\times\{L\}\right]_{\sim_\Omega}
			=\psi_{\scriptscriptstyle\Omega}^\mathbb{S}\left(\left\{z_\nu^J\right\}\right)
			\subset\psi_{\scriptscriptstyle\Omega}^\mathbb{S}\left(\varphi_{\scriptscriptstyle\breve{\Omega}}^\rho(V)\right).
		\end{split}
	\end{equation*}
	It implies that $\phi_{\breve{\scriptscriptstyle\Omega}}^\rho(U)
	\subset\psi_{\scriptscriptstyle\Omega}^\mathbb{S}\circ\varphi_{\breve{\scriptscriptstyle\Omega}}^\rho(V)$. By \eqref{eq-psopr}, $\psi_{\scriptscriptstyle\Omega}^\mathbb{S}\circ\varphi_{\breve{\scriptscriptstyle\Omega}}^\rho=\phi_{\breve{\scriptscriptstyle\Omega}}^\rho$. It follows from \eqref{eq-vsb} that \eqref{eq-cvsrc} holds.
\end{proof}

\section{Stem discrete sets}\label{Stem discrete sets}
In this section, we give a concept of stem discrete sets which applies to 2-path-symmetric open sets. The stem discrete sets play same roles as discrete sets in interpolation theorem in complex analysis.
When these sets have the property of invariance in Definition \ref{def-invariant}, we can deduce further consequences that we describe in various results in this section and that will be used in Section \ref{sec-interpolation}. In particular, in Proposition \ref{prop-construct slice regular functions}, we show how to construct a slice regular function defined on such a 2-path-symmetric open set using a formula that mimics a representation formula.
\begin{defn}\label{def-stem discrete}\index{($\breve{\Omega}$-)stem discrete, ($\breve{\Omega}$-)stem-finite}
	Let $\Omega\in\tau_s(\mathbb{H})$ be $2$-path-symmetric. $A\subset\Omega$ is called \textit{\textbf{($\breve{\Omega}$-)stem discrete}}, if for each $r>0$, $\psi^\mathbb{S}_\Omega(A\cap\Omega^r)$ has no accumulation point in $\breve{\Omega}$. Moreover, we call $A$ is \textit{\textbf{($\breve{\Omega}$-)stem-finite}} if $|\psi_{\scriptscriptstyle\Omega}^\mathbb{S}(A)|<+\infty$.
\end{defn}

\begin{defn}\label{def-invariant}\index{($\breve{\Omega}$-)invariant, ($\Omega$-)invariant}
	Let $\Omega\in\tau_s(\mathbb{H})$ be $2$-path-symmetric. $K\subset\breve\Omega$ is called \textit{\textbf{($\breve{\Omega}$-)invariant}}, if $K$ is compact and
	\begin{equation*}
		K=K_{\breve{\scriptscriptstyle\Omega}}^{\land}={\Conj}_{\breve{\scriptscriptstyle\Omega}}(K).
	\end{equation*}
	Moreover, $A\subset\Omega$ is called \textit{\textbf{($\Omega$-)invariant}}, if $\psi_{\scriptscriptstyle\Omega}^\mathbb{S}(A)$ is invariant and $\phi_{\scriptscriptstyle\Omega}^r(A)=A$ for some $r>0$.
\end{defn}

\begin{prop}\label{prop-invariant K}
	Let $\Omega\in\tau_s(\mathbb{H})$ be $2$-path-symmetric, and $K\subset\Omega$ be invariant. Then
	\begin{equation}\label{eq-psos}
		\phi_{\scriptscriptstyle\Omega}^{\sigma(K,\mathbb{H}\backslash\Omega)}(K)=K.
	\end{equation}
\end{prop}

\begin{proof}
	It follows from $K\subset\Omega^{\sigma(K,\mathbb{H}\backslash\Omega)}$ that
	\begin{equation}\label{eq-psosk}
		\phi_{\scriptscriptstyle\Omega}^{\sigma(K,\mathbb{H}\backslash\Omega)}(K)=[K]_{\approx_\Omega}^\cup\bigcap\Omega^{\sigma(K,\mathbb{H}\backslash\Omega)}\supset K.
	\end{equation}
	Since $K$ is invariant,
	\begin{equation*}
		K=\phi_{\scriptscriptstyle\Omega}^r(K)=[K]_{\approx_\Omega}^\cup\bigcap\Omega^r,\qquad\mbox{for some}\qquad r>0.
	\end{equation*}
	Hence
	\begin{equation}\label{eq-rge0}
		\sigma(K,\mathbb{H}\backslash\Omega)
		=\sigma\left([K]_{\approx_{\Omega}}^\cup\bigcap \Omega^r,\mathbb{H}\backslash\Omega\right)
		\ge \sigma\left(\Omega^r,\mathbb{H}\backslash\Omega\right)
		=r.
	\end{equation}
	Therefore, $\Omega^{\sigma(K,\mathbb{H}\backslash\Omega)}\subset\Omega^r$. It implies that
	\begin{equation}\label{eq-kksob}
		K=[K]_{\approx_\Omega}^\cup\bigcap \Omega^r\supset[K]_{\approx_\Omega}^\cup\bigcap \Omega^{\sigma(K,\mathbb{H}\backslash\Omega)}=\phi_{\scriptscriptstyle\Omega}^{\sigma(K,\mathbb{H}\backslash\Omega)}(K).
	\end{equation}
	Then \eqref{eq-psos} holds by \eqref{eq-psosk} and \eqref{eq-kksob}.
\end{proof}

\begin{prop}\label{pr-lotb2ps}
	Let $\Omega\in\tau_s(\mathbb{H})$ be $2$-path-symmetric, $K\subset\Omega$ be invariant, $p\in\Omega$ with $|[p]_{\approx_\Omega}\cap K|\ge 2$. Then $[p]_{\approx_\Omega}\subset K$.
\end{prop}

\begin{proof}
	Let $q\in [p]_{\approx_\Omega}$ and $z^I,\xi\in[p]_{\approx_\Omega}\cap K$ with $z^I\neq \xi$. By Proposition \ref{pr-losmt} \ref{it-lp},
	\begin{equation*}
		q=z^J,\qquad\mbox{and}\qquad \xi=z^L,\qquad\mbox{for some}\qquad J,L\in\mathbb{S}\mbox{ with }L\neq I.
	\end{equation*}
	According to \eqref{eq-o2ps}, $z\in\Omega_{2ps}$. Let $\mu:=[(z^I,I)]_{\sim_\Omega}$. It follows from \eqref{eq-msn} that $\mathbb{S}_\mu=\mathbb{S}$. By \eqref{eq-rmi} and \eqref{eq-sqmbo},
	\begin{equation*}
		\sigma(z^\mathcal{I},\mathbb{H}\backslash{\Omega})=\delta_\mathcal{I}(z^{\mathcal{I}},\mathbb{C}_\mathcal{I}\backslash\Omega_\mathcal{I})=r_\mu^\mathcal{I},\qquad\mathcal{I}=I,J,L.
	\end{equation*}
	It implies by $z^\mathcal{J}\in K$, $\mathcal{J}=I,L$, that
	\begin{equation*}
		r_\mu^\mathcal{J}=\sigma(z^\mathcal{J},\mathbb{H}\backslash{\Omega})\ge\sigma(K,\mathbb{H}\backslash\Omega).
	\end{equation*}
	By \eqref{eq-msmb},
	\begin{equation*}
		\sigma(q,\mathbb{H}\backslash\Omega)=\sigma(z^J,\mathbb{H}\backslash\Omega)=r_\mu^J\ge  r_\mu^-=\min\{r_\mu^I,r_\mu^L\}\ge\sigma(K,\mathbb{H}\backslash\Omega).
	\end{equation*}
	Then $q\in\Omega^{\sigma(K,\mathbb{H}\backslash\Omega)}$. By \eqref{eq-psos},
	\begin{equation*}
		q\in[p]_{\approx_\Omega}\cap\Omega^{\sigma(K,\mathbb{H}\backslash\Omega)}\subset[K]_{\approx_\Omega}^{\cup}\cap\Omega^{\sigma(K,\mathbb{H}\backslash\Omega)}=\phi_{\scriptscriptstyle\Omega}^{\sigma(K,\mathbb{H}\backslash\Omega)}(K).
	\end{equation*}
	It implies that $[p]_{\approx_\Omega}\subset K$.
\end{proof}

Let $\Omega\in\tau_s(\mathbb{H})$ be $2$-path-symmetric, $U\in\tau(\breve{\Omega})$, $g:U\rightarrow\mathbb{C}$ and $\mu\in\breve\Omega$. Define
\begin{equation*}\index{$B_\mathbb{C}^\mu$, $B_{\breve{\scriptscriptstyle\Omega}}^\mu$}
	B_\mathbb{C}^\mu:=B_{\mathbb{C}}(z_\mu,r^+_\mu),\qquad B_{\breve{\scriptscriptstyle\Omega}}^\mu:=B_{\breve{\scriptscriptstyle\Omega}}(\mu,r^+_\mu),
\end{equation*}
and
\begin{equation}\label{eq-gmbm}
	\begin{split}
		g_\mu:\ \pi_{\breve{\scriptscriptstyle\Omega}}(U\cap B_{\breve{\scriptscriptstyle\Omega}}^\mu)\ &\xlongrightarrow[\hskip1cm]{}\ \qquad \mathbb{C},
		\\ z\qquad\ &\shortmid\!\xlongrightarrow[\hskip1cm]{}\ g\circ\left(\pi_{\breve{\scriptscriptstyle\Omega}}|_{B_{\breve{\scriptscriptstyle\Omega}}^\mu}\right)^{-1}(z).
	\end{split}
\end{equation}
If $g$ is holomorphic, then by definition $g_\mu=g\circ\left(\pi_{\breve{\scriptscriptstyle\Omega}}|_{B_{\breve{\scriptscriptstyle\Omega}}^\mu}\right)^{-1}$ is also holomorphic. It follows from chain rule, $\frac{\partial}{\partial{\overline{z}}}g=0$ and $\frac{\partial}{\partial{z}}\Conj_\mathbb{C} =0$ that
\begin{equation}\label{eq-fppz}
	\begin{split}
		\frac{\partial}{\partial z} \left(g_{\mu}\circ{\Conj}_{\mathbb{C}}|_{B_\mathbb{C}^\mu}\right)
		&=\left[\left(\frac{\partial}{\partial z} g_\mu\right)\circ {\Conj}_\mathbb{C}|_{B_\mathbb{C}^\mu}\right]\cdot\left(\frac{\partial}{\partial z}{\Conj}_\mathbb{C}|_{B_\mathbb{C}^\mu}\right)
		\\&+\left[\left(\frac{\partial}{\partial \overline{z}} g_\mu\right)\circ {\Conj}_\mathbb{C}|_{B_\mathbb{C}^\mu}\right]\cdot\left(\frac{\partial}{\partial z}\overline{{\Conj}_\mathbb{C}|_{B_\mathbb{C}^\mu}}\right)=0.
	\end{split}
\end{equation}
Let $\nu\in B_{\breve{\scriptscriptstyle \Omega}}^\mu\cap U$. Then for each $\xi\in B_{\breve{\scriptscriptstyle \Omega}}^\mu\cap B_{\breve{\scriptscriptstyle \Omega}}^\nu\cap U$,
\begin{equation*}
	\left(\pi_{\breve{\scriptscriptstyle\Omega}}|_{B_{\breve\Omega}^\mu}\right)^{-1}(z_\xi)=\xi=\left(\pi_{\breve{\scriptscriptstyle\Omega}}|_{B_{\breve\Omega}^\nu}\right)^{-1} (z_\xi).
\end{equation*}
It follows that
\begin{equation*}
	g_\nu(z_\xi)=g\circ\left(\pi_{\breve{\scriptscriptstyle\Omega}}|_{B_{\breve{\scriptscriptstyle \Omega}}^\nu}\right)^{-1}(z_\xi)
	=g(\xi)
	=g\circ\left(\pi_{\breve{\scriptscriptstyle\Omega}}|_{B_{\breve{\scriptscriptstyle \Omega}}^\mu}\right)^{-1}(z_\xi)
	=g_\mu(z_\xi).
\end{equation*}
Hence
\begin{equation}\label{eq-gngm}
	g_\nu=g_\mu,\qquad\mbox{near}\qquad z_\nu.
\end{equation}
Let $I\in\mathbb{S}_\mu^+$. By Proposition \ref{pr-rrcb}, $r_\mu^+=r_{\overline{\mu}}^+$. Then
\begin{equation}\label{eq-schso}
	\begin{split}
		{\Conj}_{\breve{\scriptscriptstyle\Omega}} (B_{\breve{\scriptscriptstyle\Omega}}^\mu)
		=&{\Conj}_{\breve{\scriptscriptstyle\Omega}} \left([B_\mathbb{C}(z_\mu,r_\mu^+)^I\times\{I\}]_{\sim_\Omega}\right)
		\\=&\left[B_\mathbb{C}(\overline{z_\mu},r_{\overline{\mu}}^+)^{-I}\times\{-I\}\right]_{\sim_\Omega}=B_{\breve{\scriptscriptstyle\Omega}}^{\overline{\mu}}.
	\end{split}
\end{equation}
If $\overline{\mu}\in U$, then $\overline{\nu}\in{\Conj}_{\breve{\scriptscriptstyle\Omega}} (B_{\breve{\scriptscriptstyle\Omega}}^\mu)\cap U=B_{\breve{\scriptscriptstyle\Omega}}^{\overline{\mu}}\cap U$ and
\begin{equation}\label{eq-gngm1}
	g_{\overline{\nu}}=g_{{\overline\mu}},\qquad\mbox{near}\qquad z_{\overline{\nu}}=\overline{z_\nu}.
\end{equation}

\begin{rmk}
	Let $\Omega$ be $2$-path-symmetric, $I\in\mathbb{S}$, $z^J\in\Omega$ and $\mu:=[(z^J,J)]_{\sim_\Omega}$ with $J\neq\pm I$. Then
	\begin{equation*}
		\mu,\overline{\mu}\in[\Omega_J\times\{\pm J\}]_{\sim_\Omega}\subset\psi_{\scriptscriptstyle\Omega}^{\scriptscriptstyle\mathbb{S}\backslash\{-I\}}(\Omega).
	\end{equation*}
\end{rmk}

Let $I,J\in\mathbb{S}$. By calculation,
\begin{equation*}
	\begin{pmatrix}
		1 & I\\ 1& -I
	\end{pmatrix}\begin{pmatrix}
		1 & \\ & -1
	\end{pmatrix}=\begin{pmatrix}
		1 &  -I\\  1& I
	\end{pmatrix}=\begin{pmatrix}
		 &  1\\  1&
	\end{pmatrix}\begin{pmatrix}
		1 & I\\ 1& -I
	\end{pmatrix}.
\end{equation*}
Then check the matrices below
\begin{equation}\label{eq-bs1j}
	\begin{split}
			(1,-J)\begin{pmatrix}
			1 & I\\ 1& -I
		\end{pmatrix}
		=&(1,J)\begin{pmatrix}
			1 & \\ & -1
		\end{pmatrix}^{-1}\begin{pmatrix}
			1 & I\\ 1& -I
		\end{pmatrix}^{-1}
		\\=&(1,J)\begin{pmatrix}
			1 & I\\ 1& -I
		\end{pmatrix}^{-1}\begin{pmatrix}
			&  1\\  1&
		\end{pmatrix}^{-1}
		\\=&(1,J)\begin{pmatrix}
			1 & I\\ 1& -I
		\end{pmatrix}^{-1}\begin{pmatrix}
			&  1\\  1&
		\end{pmatrix}.
	\end{split}
\end{equation}

Similarly, it follows from
\begin{equation*}\index{$g_{\mu}$}
	\begin{pmatrix}
		1 & I\\ 1& -I
	\end{pmatrix}\begin{pmatrix}
		 & -1\\ 1&
	\end{pmatrix}=\begin{pmatrix}
		I & -1\\ -I& -1
	\end{pmatrix}=\begin{pmatrix}
		I & \\ & -I
	\end{pmatrix}\begin{pmatrix}
		1 & I\\ 1& -I
	\end{pmatrix}
\end{equation*}
that
\begin{equation*}
	\begin{split}
		J(1,J)\begin{pmatrix}
			1 & I\\ 1& -I
		\end{pmatrix}^{-1}
		=&(1,J)\begin{pmatrix}
			& -1\\ 1&
		\end{pmatrix}\begin{pmatrix}
			1 & I\\ 1& -I
		\end{pmatrix}^{-1}
		\\=&(1,J)\begin{pmatrix}
			1 & I\\ 1& -I
		\end{pmatrix}^{-1}\begin{pmatrix}
			I & \\ & -I
		\end{pmatrix}.
	\end{split}
\end{equation*}
It implies that
\begin{equation}\label{eq-bsop}
	\begin{split}
		\overline{\partial_J}(1,J)\begin{pmatrix}
			1 & I\\ 1& -I
		\end{pmatrix}^{-1}
		=&(1,J)\begin{pmatrix}
			1 & I\\ 1& -I
		\end{pmatrix}^{-1}\begin{pmatrix}
			\overline{\partial_I} & \\ & \overline{\partial_{-I}}
		\end{pmatrix},
	\end{split}
\end{equation}
where
\begin{equation*}
	\overline{\partial_L}:=\frac{\partial}{\partial \overline{z}^L}:=\frac{1}{2}\left(\frac{\partial}{\partial x}+L\frac{\partial}{\partial y}\right),\qquad L\in\mathbb{S}\cup\{i\}.
\end{equation*}
It is easy to check that
\begin{equation}\label{eq-fppo}
	\begin{cases*}
		\frac{\partial}{\partial \overline{z}^I} \Psi_i^I
		=\Psi_i^I \frac{\partial}{\partial \overline{z}^i}
		=\Psi_i^I \frac{\partial}{\partial \overline{z}},\\
		\frac{\partial}{\partial \overline{z}^{-I}} \Psi_i^I
		=\Psi_i^I \frac{\partial}{\partial \overline{z}^{-i}}
		=\Psi_i^I \frac{\partial}{\partial z}.
	\end{cases*}
\end{equation}

Let $\Omega\in\tau_s(\mathbb{H})$ be $2$-path-symmetric, $U\subset\Omega$ and $r\ge 0$. Denote
\begin{equation*}\index{$\breve{\Omega}^{\scriptscriptstyle[I]}$}\index{$U_{\breve{\scriptscriptstyle\Omega}}^{\scriptscriptstyle[I]}$, $U_{\breve{\scriptscriptstyle\Omega}}^{r,I}$}
	\breve{\Omega}^{\scriptscriptstyle[I]}:=\psi_{\scriptscriptstyle\Omega}^{\scriptscriptstyle\mathbb{S}\backslash\{-I\}}(\Omega),\qquad\qquad U_{\breve{\scriptscriptstyle\Omega}}^{\scriptscriptstyle[I]}:=\psi_{\scriptscriptstyle\Omega}^{\scriptscriptstyle\mathbb{S}\backslash\{-I\}}(U)=\left[\bigcup_{J\neq-I} U_J\times\{J\}\right]_{\sim_\Omega},
\end{equation*}
and
\begin{equation}\label{eq-ubscri}
	U_{\breve{\scriptscriptstyle\Omega}}^{r,I}:=\psi_{\scriptscriptstyle\Omega}^{\scriptscriptstyle\mathbb{S}}(U)\bigcap\left(\breve{\Omega}^{\scriptscriptstyle[I]}\right)^r.
\end{equation}

\begin{lem}\label{lem_Jexist}
	Let $\Omega\in\tau_s(\mathbb{H})$ be $2$-path-symmetric, $K\subset\Omega$ be invariant and $p\in\Omega$ with $[p]_{\approx_\Omega}\nsubseteq K$ and $|[p]_{\approx_\Omega}|>1$. Then
	\begin{equation}\label{eq-psokq}
		[p]_{\approx_\Omega}\cap K=\left\{w^{-J}\right\},\qquad\mbox{for some} \qquad w\in\mathbb{C}\quad\mbox{and}\quad J\in\mathbb{S}.
	\end{equation}
	and
	\begin{equation}\label{eq-wjjso}
		[(w^J,J)]_{\sim_\Omega}\notin K_{\breve{\scriptscriptstyle{\Omega}}}^{\scriptscriptstyle[J]}.
	\end{equation}
\end{lem}

\begin{proof}
	\eqref{eq-psokq} holds directly by Proposition \ref{pr-lotb2ps}. According to Proposition \ref{pr-losmt} \ref{it-lzio}, $[p]_{\approx_\Omega}=w^\mathbb{S}$. Suppose that \eqref{eq-wjjso} does not hold, i.e.
	\begin{equation*}
		\mu:=[(w^J,J)]_{\sim_\Omega}\in K_{\breve{\scriptscriptstyle{\Omega}}}^{\scriptscriptstyle[J]}=\left[\bigcup_{\hbar\neq-J}K_\hbar\times\{\hbar\}\right]_{\sim_\Omega}.
	\end{equation*}
	It implies that there is $v\in\mathbb{C}$ and $L\in\mathbb{S}\backslash\{-J\}$ such that $v^L\in K_L$ and $[(v^L,L)]_{\sim_\Omega}=\mu$. By Proposition \ref{pr-zinviant}, $v=z_\mu=w$. Hence
	\begin{equation*}
		w^L=v^L\in K_L\subset K,
	\end{equation*}
	a contradiction to \eqref{eq-psokq}. Therefore, \eqref{eq-wjjso} holds.
\end{proof}

\begin{prop}\label{pr-loitsh}
	Let $\Omega\in\tau_s(\mathbb{H})$ be $2$-path-symmetric, $K\subset\Omega$ be $\Omega$-invariant and $r\ge 0$. Then
	\begin{equation}\label{eq-kbso}
		K_{\breve{\scriptscriptstyle\Omega}}^{r,I}=\left(K_{\breve{\scriptscriptstyle\Omega}}^{r,I}\right)_{\breve{\Omega}^{\scriptscriptstyle[I]}}^\land,\qquad\forall\ I\in\mathbb{S}.
	\end{equation}
\end{prop}

\begin{proof}
	Since $K$ is $\Omega$-invariant, $\psi_{\scriptscriptstyle\Omega}^\mathbb{S}(K)=\left[\psi_{\scriptscriptstyle\Omega}^\mathbb{S}(K)\right]_{\breve{\Omega}}^\land$. It follows from $\breve{\Omega}^{\scriptscriptstyle[I]}\subset\breve{\Omega}$ and $K_{\breve{\scriptscriptstyle\Omega}}^{r,I}\subset\psi_{\scriptscriptstyle\Omega}^\mathbb{S}(K)$ that
	\begin{equation*}
		\left(K_{\breve{\scriptscriptstyle\Omega}}^{r,I}\right)_{\breve{\Omega}^{\scriptscriptstyle[I]}}^\land
		\subset\left(K_{\breve{\scriptscriptstyle\Omega}}^{r,I}\right)_{\breve{\Omega}}^\land\subset\left[\psi_{\scriptscriptstyle\Omega}^\mathbb{S}(K)\right]_{\breve{\Omega}}^\land=\psi_{\scriptscriptstyle\Omega}^\mathbb{S}(K)
	\end{equation*}
	According to (7.5.6) in \cite[Lemma 7.5.5]{Noguchi2016001},
	\begin{equation*}
		\begin{split}
			\delta_{\breve{\Omega}^{\scriptscriptstyle[I]}}
			\left(\left[K_{\breve{\scriptscriptstyle\Omega}}^{r,I}\right]_{\breve{\Omega}^{\scriptscriptstyle[I]}}^\land,\partial\breve{\Omega}^{\scriptscriptstyle[I]}\right)
			=\delta_{\breve{\Omega}^{\scriptscriptstyle[I]}}
			\left(K_{\breve{\scriptscriptstyle\Omega}}^{r,I},\partial\breve{\Omega}^{\scriptscriptstyle[I]}\right)
			\ge \delta_{\breve{\Omega}^{\scriptscriptstyle[I]}}\left(\left(\breve{\Omega}^{\scriptscriptstyle[I]}\right)^r,\partial\breve{\Omega}^{\scriptscriptstyle[I]}\right)= r,
		\end{split}
	\end{equation*}
	i.e. $\left[K_{\breve{\scriptscriptstyle\Omega}}^{r,I}\right]_{\breve{\Omega}^{\scriptscriptstyle[I]}}^\land\subset \left(\breve{\Omega}^{\scriptscriptstyle[I]}\right)^r$. It implies that
	\begin{equation*}
		\left(K_{\breve{\scriptscriptstyle\Omega}}^{r,I}\right)_{\breve{\Omega}^{\scriptscriptstyle[I]}}^\land\subset \psi_{\scriptscriptstyle\Omega}^\mathbb{S}(K)\bigcap\left(\breve{\Omega}^{\scriptscriptstyle[I]}\right)^r=K_{\breve{\scriptscriptstyle\Omega}}^{r,I}.
	\end{equation*}
	By definition, $K_{\breve{\scriptscriptstyle\Omega}}^{r,I}\subset \left(K_{\breve{\scriptscriptstyle\Omega}}^{r,I}\right)_{\breve{\Omega}^{\scriptscriptstyle[I]}}^\land$. Therefore \eqref{eq-kbso} holds.
\end{proof}

\begin{prop}\label{pr-U I in U sigma I}
	Let $\Omega\in\tau_s(\Omega)$ be $2$-path-symmetric, $I\in\mathbb{S}$ and $U\subset\Omega$. Then
	\begin{equation*}
		U^{\scriptscriptstyle[I]}_{\breve{\scriptscriptstyle\Omega}}\subset U_{\breve{\scriptscriptstyle\Omega}}^{\sigma_{\mathbb{H}\backslash\Omega}^U,I}.
	\end{equation*}
\end{prop}

\begin{proof}
	Let $\mu\in U_{\breve{\scriptscriptstyle\Omega}}^{\scriptscriptstyle[I]}=\bigcup_{L\in\mathbb{S}\backslash\{-I\}}[U_L\times\{L\}]_{\sim_\Omega}$. Then $\mu=[(z_{\mu}^J,J)]_{\sim_\Omega}$ for some $\mu\in\mathbb{S}_\mu\backslash\{-I\}$. It implies that
	\begin{equation*}
		B_J(z^J,\sigma_{\scriptscriptstyle\mathbb{H}\backslash\Omega}^{\scriptscriptstyle U})
		\subset\Sigma(z^J,\sigma_{\scriptscriptstyle\mathbb{H}\backslash\Omega}^{\scriptscriptstyle U})\subset\Omega.
	\end{equation*}
	Therefore
	\begin{equation*}
		B_{\breve{\scriptscriptstyle\Omega}}(\mu,\sigma_{\scriptscriptstyle\mathbb{H}\backslash\Omega}^{\scriptscriptstyle U})
		=[B_J(z^J,\sigma_{\scriptscriptstyle\mathbb{H}\backslash\Omega}^{\scriptscriptstyle U})\times\{J\}]_{\sim_\Omega}\subset\bigcup_{L\in\mathbb{S}\backslash\{-I\}}[U_L\times\{L\}]_{\sim_\Omega}=U_{\scriptscriptstyle\Omega}^{\scriptscriptstyle[I]}.
	\end{equation*}
	Then $\mu\in(U_{\breve{\scriptscriptstyle\Omega}}^{\scriptscriptstyle[I]})^{\sigma_{\scriptscriptstyle\mathbb{H}\backslash\Omega}^{\scriptscriptstyle U}}$ and $U_{\breve{\scriptscriptstyle\Omega}}^{\scriptscriptstyle[I]}
	\subset(U_{\breve{\scriptscriptstyle\Omega}}^{\scriptscriptstyle[I]})^{\sigma_{\scriptscriptstyle\mathbb{H}\backslash\Omega}^{\scriptscriptstyle U}}$. Therefore
	\begin{equation*}
		U_{\breve{\scriptscriptstyle\Omega}}^{\scriptscriptstyle[I]}
		\subset \psi_{\scriptscriptstyle\Omega}^{\scriptscriptstyle\mathbb{S}}(U)\bigcap\left(\breve{\Omega}^{\scriptscriptstyle[I]}\right)^{\sigma_{\scriptscriptstyle\mathbb{H}\backslash\Omega}^{\scriptscriptstyle U}}
		=U_{\breve{\scriptscriptstyle\Omega}}^{\sigma_{\mathbb{H}\backslash\Omega}^U,I}.
	\end{equation*}
\end{proof}

\begin{prop}\label{prop-construct slice regular functions}
	Let $\Omega\in\tau_s(\mathbb{H})$ be $2$-path-symmetric, $I\in\mathbb{S}$, $g:\breve{\Omega}^{\scriptscriptstyle[I]}\rightarrow\mathbb{C}$ be holomorphic. Define
	\begin{equation}\label{eq-gioxm}\index{$g_{\scriptscriptstyle \Omega}^I$}
		\begin{split}
			g_{\scriptscriptstyle \Omega}^I:\ \Omega\ &\xlongrightarrow[\hskip1cm]{}\ \mathbb{H},
			\\ z^J &\shortmid\!\xlongrightarrow[\hskip1cm]{}\
			\begin{cases}
				\Psi_i^I\circ g_{\mu}(z)=\Psi_i^I\circ g(\mu),\qquad &J=I,\\
				h(z,J),\qquad &J\neq\pm I,
			\end{cases}
		\end{split}
	\end{equation}
	for each $z^J\in\Omega$ with $z\in\mathbb{C}$ and $J\in\mathbb{S}\backslash\{-I\}$, where $\mu=[(z^J,J)]_{\sim_\Omega}$ and
	\begin{equation*}
		h(z,J)=(1,J)\begin{pmatrix}
			1 & I\\ 1&-I
		\end{pmatrix}^{-1}
		\begin{pmatrix}
			\Psi_i^I\circ g_\mu(z)\\\Psi_i^I\circ g_{\overline{\mu}}(\overline{z})
		\end{pmatrix}.
	\end{equation*}
	Then $g_{\scriptscriptstyle \Omega}^I$  is well-defined and slice regular. Moreover,
	\begin{equation}\label{eq-vblmsm}
		\varphi_{\breve{\scriptscriptstyle\Omega}}^0
		\left(\mathcal{Z}(g)\cap {\Conj}_{\breve{\scriptscriptstyle\Omega}}\circ\mathcal{Z}(g)\right)
		\subset \mathcal{Z}(g_{\scriptscriptstyle\Omega}^I).
	\end{equation}
\end{prop}

\begin{proof}
	(i) Let $z^J=w^L$ for some $z,w\in\mathbb{S}$ and $J,L\in\mathbb{S}\backslash\{-I\}$ with $(z,J)\neq(w,L)$. If $z\in\mathbb{R}$, then by \eqref{eq-omr}, $z=z^J\in\Omega_\mathbb{R}\subset\Omega_{2ps}$. It follows that $w^L=z^J\in\mathbb{R}$ that $w=w^L=z^J=z$. According to \eqref{eq-sim},
	\begin{equation*}
		\mu=[(z^J,J)]_{\sim_\Omega}=[(z^{-J},{-J})]_{\sim_\Omega}=[(z^J,{-J})]_{\sim_\Omega}=\overline{\mu}.
	\end{equation*}
	Similarly, $\nu=\overline{\nu}=\mu$, where $\nu=[(w^L,L)]_{\sim_\Omega}$. Hence if $L\neq I$, then
	\begin{equation*}
		\begin{split}
			h(w,L)=&(1,L)\begin{pmatrix}
				1 &I\\ 1&-I
			\end{pmatrix}^{-1}\begin{pmatrix}
				\Psi_i^I\circ g_\nu(w)\\\Psi_i^I\circ g_{\overline{\nu}}(\overline{w})
			\end{pmatrix}
			\\=&\frac{1}{2}(1,L)\begin{pmatrix}
				1 &1 \\ -I&I
			\end{pmatrix}\begin{pmatrix}
				\Psi_i^I\circ g_\mu(z)\\\Psi_i^I\circ g_{\mu}(z)
			\end{pmatrix}
			=\Psi_i^I\circ g_{\mu}(z).
		\end{split}
	\end{equation*}
	Similarly, if $J\neq I$, then $h(z,J)=\Psi_i^I\circ g_{\mu}(z)$. It is easy to check that $g_{\scriptscriptstyle \Omega}^I$ is well-defined in this case.
	
	Otherwise, $z\notin\mathbb{R}$ and it is immediate that $w=\overline{z}$ and $L=-J\neq\pm I$. Then
	\begin{equation*}
		\nu=[(w^L,L)]_{\sim_\Omega}=[(\overline{z}^{-J},-J)]_{\sim_\Omega}=\overline{\mu}.
	\end{equation*}
	By \eqref{eq-bs1j},
	\begin{equation*}
		\begin{split}
			h(w,L)=&(1,L)\begin{pmatrix}
				1 &I\\ 1&-I
			\end{pmatrix}^{-1}\begin{pmatrix}
				\Psi_i^I\circ g_{\nu}(w)\\\Psi_i^I\circ g_{\overline{\nu}}(\overline{w})
			\end{pmatrix}
			\\=&(1,-J)\begin{pmatrix}
				1 &I\\ 1&-I
			\end{pmatrix}^{-1}\begin{pmatrix}
				\Psi_i^I\circ g_{\overline{\mu}}(\overline{z})\\\Psi_i^I\circ g_{\mu}(z)
			\end{pmatrix}
			\\=&(1,J)\begin{pmatrix}
				1 &I\\ 1&-I
			\end{pmatrix}^{-1}\begin{pmatrix}
				 &1\\ 1&
			\end{pmatrix}\begin{pmatrix}
				\Psi_i^I\circ g_{\overline{\mu}}(\overline{z})\\\Psi_i^I\circ g_{\mu}(z)
			\end{pmatrix}
			\\=&(1,J)\begin{pmatrix}
				1 &I\\ 1&-I
			\end{pmatrix}^{-1}\begin{pmatrix}
				\Psi_i^I\circ g_{\mu}(z)\\\Psi_i^I\circ g_{\overline{\mu}}(\overline{z})
			\end{pmatrix}=h(z,J).
		\end{split}
	\end{equation*}
	In summary, $h(w,L)=h(z,J)$ for each $w^L=z^J$ with $L,J\in\mathbb{S}\backslash\{-I\}$. Hence $g_{\scriptscriptstyle \Omega}^I$ is well-defined.
	
	(ii) Let $w^J\in\Omega$, $\nu=[(w^J,J)]_{\sim_\Omega}$ and $z^J\in B_J\left(w^J,r_\nu^J\right)$. Then
	\begin{equation*}
		\mu:=[(z^J,J)]_{\sim_\Omega}\in \left[B_{\mathbb{C}}\left(w,r_\nu^J\right)^J\times\{J\}\right]_{\sim_\Omega}
		\subset B_{\breve{\scriptscriptstyle\Omega}}\left(\nu,r_\nu^J\right)
		=B_{\breve{\scriptscriptstyle\Omega}}\left(\nu,r_\nu^+\right)
		=B_{\breve{\scriptscriptstyle\Omega}}^\nu.
	\end{equation*}
	If $J=I$, then by \eqref{eq-gngm} and \eqref{eq-fppo} that
	\begin{equation*}
		\frac{\partial}{\partial \overline{z}^J} g_{\scriptscriptstyle\Omega}^I(z^J)=\frac{\partial}{\partial \overline{z}^I}\Psi_i^I\circ g_\mu(z)=\frac{\partial}{\partial \overline{z}}g_\nu(z)=0.
	\end{equation*}
	If $J=-I$, then
	\begin{equation*}
		[(\overline{z}^I,I)]_{\sim_\Omega}
		=[(\overline{z}^{-J},{-J})]_{\sim_\Omega}
		=[(z^{J},{-J})]_{\sim_\Omega}=\overline{\mu}.
	\end{equation*}
	It follows from \eqref{eq-gngm1}, \eqref{eq-fppo} and \eqref{eq-fppz} that
	\begin{equation*}
		\frac{\partial}{\partial \overline{z}^J} g_{\scriptscriptstyle\Omega}^I(z^J)
		=\frac{\partial}{\partial \overline{z}^{-I}} g_{\scriptscriptstyle\Omega}^I(\overline{z}^{I})
		=\frac{\partial}{\partial \overline{z}^{-I}} \Psi_i^I\circ g_{\overline{\mu}}(\overline{z})
		=\Psi_i^I \frac{\partial}{\partial z} g_{\overline{\nu}}\circ{\Conj}_{\mathbb{C}}(z)=0.
	\end{equation*}
	Otherwise $J\neq\pm I$. By the same method as above two cases and according to \eqref{eq-bsop},
	\begin{equation*}
		\begin{split}
			\frac{\partial}{\partial \overline{z}^J} g_{\scriptscriptstyle\Omega}^I(z^J)
			=&\frac{\partial}{\partial \overline{z}^J} (1,J)\begin{pmatrix}
				1 &I\\1& -I
			\end{pmatrix}^{-1}\begin{pmatrix}
				\Psi_i^I \circ g_\mu(z)\\ \Psi_i^I\circ g_{\overline{\mu}}(\overline{z})
			\end{pmatrix}
			\\=& (1,J)\begin{pmatrix}
				1 &I\\1& -I
			\end{pmatrix}^{-1}\begin{pmatrix}
				\frac{\partial}{\partial \overline{z}^I}\\ &\frac{\partial}{\partial \overline{z}^{-I}}
			\end{pmatrix}\begin{pmatrix}
				\Psi_i^I \circ g_\nu(z)\\ \Psi_i^I\circ g_{\overline{\nu}}(\overline{z})
			\end{pmatrix}
			\\=& (1,J)\begin{pmatrix}
				1 &I\\1& -I
			\end{pmatrix}^{-1}\begin{pmatrix}
				\Psi_i^I\\ &\Psi_i^I
			\end{pmatrix}\begin{pmatrix}
				\frac{\partial}{\partial \overline{z}} g_\nu(z)
				\\\frac{\partial}{\partial z} g_{\overline{\nu}}\circ{\Conj}_{\mathbb{C}}(z)
			\end{pmatrix}
			=0.
		\end{split}
	\end{equation*}
	Since the choice of $z^J\in B_J\left(w^J,r_\nu^J\right)$ is arbitrary, $g_{\scriptscriptstyle\Omega}^I|_{B_J\left(w^J,r_\nu^J\right)}$ is holomorphic. It follows that for each $J\in\mathbb{S}$, $g_{\scriptscriptstyle\Omega}^I|_{\Omega_J}$ is holomorphic at each point $w^J\in\Omega_J$, and then holomorphic on $\Omega_J$. Therefore, $g_{\scriptscriptstyle\Omega}^I$ is slice regular.
	
	(iii) Let $z^J\in\varphi_{\breve{\scriptscriptstyle\Omega}}^0
	\left(\mathcal{Z}(g)\cap {\Conj}_{\breve{\scriptscriptstyle\Omega}}\circ\mathcal{Z}(g)\right)$ with $J\neq -I$, and $\mu:=[(z^J,J)]_{\sim_\Omega}$. It follows from \eqref{eq-psoz} that there is $\lambda\in\{0, 1\}$ such that
	\begin{equation*}
		{\Conj}_{\breve{\scriptscriptstyle\Omega}}^\lambda(\mu)=[(z^J,(-1)^\lambda J)]_{\sim_\Omega}\in\psi_{\scriptscriptstyle \Omega}^\mathbb{S}(\{z^J\})\cap\mathcal{Z}(g)\cap {\Conj}_{\breve{\scriptscriptstyle\Omega}}\circ\mathcal{Z}(g)\neq\varnothing.
	\end{equation*}
	Then for each $\kappa\in\{0,1\}$
	\begin{equation*}
		\begin{split}
			{\Conj}_{\breve{\scriptscriptstyle\Omega}}^\kappa(\mu)=&{\Conj}_{\breve{\scriptscriptstyle\Omega}}^{\lambda+\kappa}\left[{\Conj}_{\breve{\scriptscriptstyle\Omega}}^\lambda(\mu)\right]
			\\\in&{\Conj}_{\breve{\scriptscriptstyle\Omega}}^{\lambda+\kappa}\left[\mathcal{Z}(g)\cap {\Conj}_{\breve{\scriptscriptstyle\Omega}}\circ\mathcal{Z}(g)\right]
			=\mathcal{Z}(g)\cap {\Conj}_{\breve{\scriptscriptstyle\Omega}}\circ\mathcal{Z}(g).
		\end{split}
	\end{equation*}
	Therefore $\mu,\overline{\mu}\in\mathcal{Z}(g)$. According to \eqref{eq-gmbm},
	\begin{equation*}
		g_\mu(z)=g\circ\left(\pi|_{B_{\breve{\scriptscriptstyle\Omega}}^\mu}\right)^{-1}(z)=g(\mu)=0.
	\end{equation*}
	Similarly $g_{\overline{\mu}}(\overline{z})=0$. It is easy to check by \eqref{eq-gioxm} that $g_{\scriptscriptstyle\Omega}^I(z^J)=0$. Hence $z^J\in\mathcal{Z}(g_{\scriptscriptstyle\Omega})$. It implies that \eqref{eq-vblmsm} holds.
\end{proof}

We recall some facts in slice analysis. Let $u,v\in\mathbb{H}$ and suppose that $u\overline{v}\in\mathbb{C}_I$ for some $I\in\mathbb{C}$. Then for any $J\in\mathbb{S}$,
\begin{equation*}
	\min_{L=\pm I} \left|u+Lv\right|\le \left|u+Jv\right|\le \max_{L=\pm I} \left|u+Lv\right|.
\end{equation*}
It implies that for each $w\in(\mathbb{C}_I)^{2\times 1}$ and $J\in\mathbb{S}$,
\begin{equation}\label{eq-l1jw}
	\left|(1,J)w\right|\le \max_{L=\pm I} \left|(1,L)w\right|.
\end{equation}

\begin{prop}\label{pr-loitps}
	Let $\Omega\in\tau_s(\mathbb{H})$ be $2$-path-symmetric, $U\subset\Omega$ with $r:=\sigma(U,\mathbb{H}\backslash\Omega)>0$, $I\in\mathbb{S}$ and $g:{\breve{\Omega}}^{\scriptscriptstyle [I]}\rightarrow\mathbb{C}$. Then
	\begin{equation}\label{eq-sulg}
		\sup_U \left|g_{\scriptscriptstyle{\Omega}}^{\scriptscriptstyle{I}}\right|\le\sup_{{U}_{\scriptscriptstyle{\breve{\Omega}}}^{\scriptscriptstyle [I]}}|g|.
	\end{equation}
\end{prop}

\begin{proof}
	Let $z^J\in U$ with $J\in\mathbb{S}\backslash\{-I\}$, and $\mu:=\{z^J\}$. If $J=I$ then
	\begin{equation*}
		|g_{\scriptscriptstyle{\Omega}}^{\scriptscriptstyle{I}}(z^I)|=|\Psi_i^I\circ g_\mu(z)|=|g_\mu(z)|=|g(\mu)|\le \sup_{{U}_{\scriptscriptstyle{\breve{\Omega}}}^{\scriptscriptstyle [I]}}|g|.
	\end{equation*}
	Otherwise, $J\in\mathbb{S}\backslash\{\pm I\}$. Then
	\begin{equation*}
		w:=\begin{pmatrix}
			1 & I\\ 1&-I
		\end{pmatrix}^{-1}
		\begin{pmatrix}
			\Psi_i^I\circ g_\mu(z)\\\Psi_i^I\circ g_{\overline{\mu}}(\overline{z})
		\end{pmatrix}\in(\mathbb{C}_I)^{2\times 1},
	\end{equation*}
	and
	\begin{equation*}
		\begin{pmatrix}
			(1,I)w\\(1,-I)w
		\end{pmatrix}=\begin{pmatrix}
		1 & I\\ 1&-I
	\end{pmatrix}w=
	\begin{pmatrix}
		\Psi_i^I\circ g_\mu(z)\\\Psi_i^I\circ g_{\overline{\mu}}(\overline{z})
	\end{pmatrix}
	=\begin{pmatrix}
		\Psi_i^I\circ g(\mu)\\\Psi_i^I\circ g(\overline\mu)
	\end{pmatrix}.
	\end{equation*}
	By $z^J\in U$, we have
	\begin{equation*}
		\mu=[(z^J,J)]_{\sim_\Omega}=[U_J\times\{J\}]_{\sim_\Omega}\subset \psi_{\scriptscriptstyle\Omega}^{\scriptscriptstyle\mathbb{S}\backslash\{I\}}(U)=U_{\breve{\scriptscriptstyle\Omega}}^{[I]},
	\end{equation*}
	and
	\begin{equation*}
		\overline{\mu}=[(z^J,-J)]_{\sim_\Omega}=[U_{-J}\times\{-J\}]_{\sim_\Omega}\subset \psi_{\scriptscriptstyle\Omega}^{\scriptscriptstyle\mathbb{S}\backslash\{I\}}(U)=U_{\breve{\scriptscriptstyle\Omega}}^{[I]}.
	\end{equation*}
	According to \eqref{eq-l1jw},
	\begin{equation*}
		\begin{split}
			|g_{\scriptscriptstyle{\Omega}}^{\scriptscriptstyle{I}}(z^I)|
			=&|(1,J)w|
			\le\max_{L=\pm I}|(1,L)w|
			=\max_{\nu=\mu,\overline{\mu}} |\Psi_i^I\circ g(\nu)|
			\\=&\max_{\nu=\mu,\overline{\mu}} |g(\nu)|\le \sup_{{U}_{\scriptscriptstyle{\breve{\Omega}}}^{\scriptscriptstyle [I]}}|g|.
		\end{split}
	\end{equation*}
	Therefore
	\begin{equation*}
		|g_{\scriptscriptstyle{\Omega}}^{\scriptscriptstyle{I}}(q)|
		\le\sup_{{U}_{\scriptscriptstyle{\breve{\Omega}}}^{\scriptscriptstyle [I]}}|g|,\qquad\forall\ q\in U.
	\end{equation*}
	and \eqref{eq-sulg} holds.
\end{proof}

\begin{prop}
	Let $\Omega\in\tau_s(\mathbb{H})$ be $2$-path-symmetric, $p\in\Omega$, and $U\subset\Omega$. Then
	\begin{equation}\label{eq-rlpso}
		r^+\left(\psi_{\scriptscriptstyle\Omega}^\mathbb{S}(\{p\})\right)\ge \sigma(p,\mathbb{H}\backslash\Omega),
	\end{equation}
	and
	\begin{equation*}
		r^+\left(\psi_{\scriptscriptstyle\Omega}^\mathbb{S}(U)\right)\ge\sigma(U,\mathbb{H}\backslash\Omega).
	\end{equation*}
\end{prop}

\begin{proof}
	Let $w^I=p$ and $\mu=\left[(w^I,I)\right]_{\sim_\Omega}$. By \eqref{eq-sqmbo},
	\begin{equation*}
		[B_{\mathbb{C}}(z_\mu,\sigma^p_{\scriptscriptstyle\mathbb{H}\backslash{\Omega}})]^I=B_I(p,\sigma^p_{\scriptscriptstyle\mathbb{H}\backslash{\Omega}})=B_I\big(p,\delta_I(p,\mathbb{C}_I\backslash\Omega_I)\big)\subset\Omega.
	\end{equation*}
	It follows from Proposition \ref{pr-rrcb} that $\sigma^p_{\scriptscriptstyle\mathbb{H}\backslash{\Omega}}\le r_\mu^I\le r_\mu^+=r_{\overline{\mu}}^+$. Then \eqref{eq-rlpso} holds by \eqref{eq-psoz}. It is clear that
	\begin{equation*}
		\begin{split}
			r^+\left(\psi_{\scriptscriptstyle\Omega}^\mathbb{S}(U)\right)
			=&r^+\left(\bigcup_{p\in U}\psi_{\scriptscriptstyle\Omega}^\mathbb{S}(\{p\})\right)
			=\inf_{p\in U}r^+\left(\psi_{\scriptscriptstyle\Omega}^\mathbb{S}(\{p\})\right)
			\\\ge &\inf_{p\in U}\sigma(p,\mathbb{H}\backslash\Omega)
			=\sigma(U,\mathbb{H}\backslash\Omega).
		\end{split}
	\end{equation*}
\end{proof}

\begin{prop}\label{pr-ltmnp}
	Let $\Omega\in\tau_s(\mathbb{H})$ be $2$-path-symmetric, and $\mu\in\breve{\Omega}$. Then $\mu\notin\psi_{\scriptscriptstyle\Omega}^{\mathbb{S}\backslash\mathbb{S}_{\mu}}(\Omega)$.
\end{prop}

\begin{proof}
	If $\mathbb{S}=\mathbb{S}_\mu$, then $\mu\notin\varnothing=\psi_{\scriptscriptstyle\Omega}^{\mathbb{S}\backslash\mathbb{S}_{\mu}}(\Omega)$ and Proposition \ref{pr-ltmnp} holds. Otherwise, $\mathbb{S}\neq\mathbb{S}_\mu$. Suppose that $\mu\in[\Omega_I\times\{I\}]_{\sim_\Omega}$ for some $I\in\mathbb{S}\backslash\mathbb{S}_\mu$. Let $J\in\mathbb{S}_\mu$. According to \eqref{eq-msn}, $z_\mu\notin\Omega_{2ps}$. Hence there is $w^I\in\Omega_I$ such that
	\begin{equation*}
		[(z_\mu^J,J)]_{\sim_\Omega}=\mu=[(w^I,I)]_{\sim_\Omega},\qquad\mbox{i.e.}\qquad(z_\mu^J,J)\sim_{\scriptscriptstyle\Omega} (w^I,I).
	\end{equation*}
	It follows from \eqref{eq-sim} that $J=I$, a contradiction. Thence $\mu\notin[\Omega_I\times\{I\}]_{\sim_\Omega}$ and
	\begin{equation*}
		\mu\notin\bigcup_{L\in\mathbb{S}\backslash\mathbb{S}_\mu}[\Omega_L\times\{L\}]_{\sim_\Omega}=\psi_{\scriptscriptstyle\Omega}^{\mathbb{S}\backslash\mathbb{S}_{\mu}}(\Omega).
	\end{equation*}
\end{proof}

\begin{prop}\label{pr-loitsm}
	Let $\Omega\in\tau_s(\mathbb{H})$ be $2$-path-symmetric, $K\subset\Omega$ be invariant, $z^I\in\Omega\backslash K$, and $\mu=[(z^I,I)]_{\sim_\Omega}$. Then
	\begin{equation*}
		\begin{cases}
			\overline{\mu}\in\breve{\Omega}^{\scriptscriptstyle[J]}\backslash \left(\breve{\Omega}^{\scriptscriptstyle[J]}\right)^{\sigma^K_{\mathbb{H}\backslash\Omega}},\quad&[z^I]_{\approx_\Omega}\cap K=\{z^J\}\mbox{ for some } J\in\mathbb{S},
			\\ \mu,\overline{\mu}\notin \psi^{\mathbb{S}}_{\scriptscriptstyle\Omega}(K),\ [z^I]_{\approx_\Omega}\cap K=\varnothing,\quad& \mbox{otherwise}.
		\end{cases}
	\end{equation*}
\end{prop}

\begin{proof}
	(i) Suppose $[z^I]_{\approx_\Omega}\cap K=\{z^J\}$ for some $J\in\mathbb{S}$. Then
	\begin{equation*}
		r_{\mu}^+\ge r_\mu^J=\sigma(z^J,\mathbb{H}\backslash\Omega)\ge \sigma(K,\mathbb{H}\backslash\Omega).
	\end{equation*}
	If $\sigma(z^I,\mathbb{H}\backslash\Omega)\ge \sigma(K,\mathbb{H}\backslash\Omega)$, then $z^I\in\Omega^{\sigma(z^I,\mathbb{H}\backslash\Omega)}\subset\Omega^{\sigma(K,\mathbb{H}\backslash\Omega)}$. By \eqref{eq-psoms},
	\begin{equation*}
		\psi_{\scriptscriptstyle\Omega}^{\mathbb{S}}(z^I)\cap\psi_{\scriptscriptstyle\Omega}^{\mathbb{S}}(K)=\psi_{\scriptscriptstyle\Omega}^{\mathbb{S}}(z^J)\cap\psi_{\scriptscriptstyle\Omega}^{\mathbb{S}}(K)=\psi_{\scriptscriptstyle\Omega}^{\mathbb{S}}(z^J)\neq\varnothing,
	\end{equation*}
	\eqref{eq-cvsrc} and \eqref{eq-psos} that
	\begin{equation*}
		z^I\in\varphi_{\breve{\scriptscriptstyle\Omega}}^{\sigma(K,\mathbb{H}\backslash\Omega)}\left(\psi_{\scriptscriptstyle\Omega}^{\mathbb{S}}(K)\right)=\phi_{\scriptscriptstyle\Omega}^{\sigma(K,\mathbb{H}\backslash\Omega)}(K)=K,
	\end{equation*}
	a contradiction. Therefore $\sigma(z^I,\mathbb{H}\backslash\Omega)< \sigma(K,\mathbb{H}\backslash\Omega)$. Let
	\begin{equation*}
		w^I\in \left[B_{\mathbb{C}}(z,\sigma^K_{\scriptscriptstyle\mathbb{H}\backslash\Omega})\right]^I \backslash\Omega.
	\end{equation*}
	Then
	\begin{equation*}
		\nu:=\left[(w^J,J)\right]_{\sim_\Omega}
		\in \left[\left[B_{\mathbb{C}}(z,\sigma^K_{\scriptscriptstyle\mathbb{H}\backslash\Omega})\right]^J\times\{J\}\right]_{\sim_\Omega}
		=B_{\breve{\scriptscriptstyle\Omega}}(\mu,\sigma^K_{\scriptscriptstyle\mathbb{H}\backslash\Omega}).
	\end{equation*}
	It implies by $I\notin\mathbb{S}_\nu$, \eqref{eq-msn} and \eqref{eq-zomoz} that $\mathbb{S}_{\overline{\nu}}=-\mathbb{S}_\nu=-\{J\}=\{-J\}$. According to Proposition \ref{pr-ltmnp}, $\overline{\nu}\notin\psi_{\scriptscriptstyle\Omega}^{\scriptscriptstyle\mathbb{S}\backslash\{-J\}}(\Omega)$. It follows from
	\begin{equation*}
		\overline{\nu}={\Conj}_{\breve{\scriptscriptstyle\Omega}}(\nu)
		\in{\Conj}_{\breve{\scriptscriptstyle\Omega}}\left[B_{\breve{\scriptscriptstyle\Omega}}(\mu,\sigma^K_{\scriptscriptstyle\mathbb{H}\backslash\Omega})\right]
		=B_{\breve{\scriptscriptstyle\Omega}}(\overline{\mu},\sigma^K_{\scriptscriptstyle\mathbb{H}\backslash\Omega})
	\end{equation*}
	that
	\begin{equation*}
		B_{\breve{\scriptscriptstyle\Omega}}(\overline{\mu},\sigma^K_{\scriptscriptstyle\mathbb{H}\backslash\Omega})\not\subset\psi_{\scriptscriptstyle\Omega}^{\scriptscriptstyle\mathbb{S}\backslash\{-J\}}(\Omega)=\breve{\Omega}^{[J]}.
	\end{equation*}
	Thence $\overline{\mu}\notin  (\breve{\Omega}^{[J]})^{\sigma^K_{\mathbb{H}\backslash\Omega}}$. According \eqref{eq-o2ps} and $z^I\approx_{\scriptscriptstyle\Omega} z^J$, we have $z\in\Omega_{2ps}$ and $\overline{z}^J=z^{-J}\in\Omega$. According to \eqref{eq-wiiso},
	\begin{equation*}
		\overline{\mu}=[(\overline{z}^{-J},-J)]_{\sim_\Omega}=[(\overline{z}^{J},J)]_{\sim_\Omega}
		\in[\Omega_J\times\{J\}]_{\sim_\Omega}
		\subset\psi_{\scriptscriptstyle\Omega}^{\mathbb{S}\backslash\{-J\}}(\Omega)=\breve{\Omega}^{[J]}.
	\end{equation*}
	In summary, $\overline{\mu}\in\breve{\Omega}^{[J]}\backslash (\breve{\Omega}^{[J]})^{\sigma^K_{\mathbb{H}\backslash\Omega}}$.
	
	(ii) Otherwise, $\left|[z^I]_{\approx_\Omega}\cap K\right|\neq 1$. If $\left|[z^I]_{\approx_\Omega}\cap K\right|\ge 2$, then by Proposition \ref{pr-lotb2ps},
	\begin{equation*}
		z^I\in[z^I]_{\approx_\Omega}\subset K,
	\end{equation*}
	a contradiction. Thence $[z^I]_{\approx_\Omega}\cap K=\varnothing$. If $\{\mu,\overline{\mu}\}\cap\psi_{\scriptscriptstyle\Omega}^\mathbb{S}(K)\neq\varnothing$. Then there is $w^J\in K$ and $\lambda,\kappa\in\{0,1\}$ such that ${\Conj}_{\breve{\scriptscriptstyle\Omega}}^\lambda(\mu)=[(w^J,(-1)^\kappa J)]_{\sim_\Omega}$. Then
	\begin{equation*}
		[(z^I,I)]_{\sim_\Omega}=\mu={\Conj}_{\breve{\scriptscriptstyle\Omega}}^\lambda[(w^J,(-1)^\kappa J)]_{\sim_\Omega}=[(w^J,(-1)^{\kappa+\lambda} J)]_{\sim_\Omega}=[({w'}^{J'},J')]_{\sim_\Omega},
	\end{equation*}
	where $w'={\Conj}_{\mathbb{C}}^{\kappa+\lambda}(w)$ and $J'=(-1)^{\kappa+\lambda} J$. By \eqref{eq-zaia}, $z^I\approx_{\scriptscriptstyle\Omega} {w'}^{J'}$. It implies that
	\begin{equation*}
		w^J={w'}^{J'}\in[z^I]_{\approx_\Omega} \cap K,
	\end{equation*}
	a contradiction. Thence $\{\mu,\overline{\mu}\}\cap\psi_{\scriptscriptstyle\Omega}^\mathbb{S}(K)=\varnothing$, i.e. $\mu,\overline{\mu}\notin \psi_{\scriptscriptstyle\Omega}^\mathbb{S}(K)$.
\end{proof}

\section{Interpolation Theorem}\label{sec-interpolation}

In this section, we will prove an interpolation theorem for slice regular functions, which is crucial for proving the Cartan-Thullen Theorem \ref{thm-ct} but also of independent interest. We begin by defining when an open set $\Omega$ is 2-path-symmetric and when a function is 2-path-slice, so that we can prove how to obtain a unique extension of 2-path-slice function to such a set $\Omega$, see Proposition \ref{uniqext}. Another extension result, this time to obtain a slice regular function, is proved in Proposition \ref{pr-l2psia}; this result will be needed to prove the Interpolation Theorem \ref{th-it} for slice regular function defined on $2$-path-symmetric slice-open sets.

\begin{defn}
	Let $\Omega\in\tau_s(\mathbb{H})$ be $2$-path-symmetric, $U\subset\Omega$ and let
	\begin{equation*}\index{$U_{\scriptscriptstyle \Omega,2}$, $U_{\scriptscriptstyle \Omega,\tilde{2}}$}
		U_{\scriptscriptstyle \Omega,2}:=\{p\in U: |[p]_{\approx_\Omega}\cap U|\ge 2\}.
	\end{equation*}
	We call
	\begin{equation}\label{eq-uso}
		U_{\scriptscriptstyle \Omega,\tilde{2}}:=\left(\bigcup_{p\in U_{\Omega,2}}[p]_{\approx_\Omega}\right)\bigcup U
	\end{equation}
	the \textit{\textbf{$\Omega$-$2$-path-symmetric completion}} of $U$. We say that $U$ is \textit{\textbf{$\Omega$-$2$-path-symmetric}}\index{$\Omega$-$2$-path-symmetric}, if
	$U=U_{\scriptscriptstyle\Omega,\widetilde{2}}$.
\end{defn}

\begin{defn}
	Let $U\subset\Omega\subset\mathbb{H}$. $f:U\rightarrow\mathbb{H}$ is called \textit{\textbf{$\Omega$-path-slice}}\index{$\Omega$-path-slice}, if for each $\gamma\in\mathscr{P}_0(\mathbb{C})$, $I,J,K\in\mathbb{S}$ with $\gamma^I,\gamma^J,\gamma^K\subset\Omega$, $J\neq K$ and $\gamma^I(1),\gamma^J(1),\gamma^K(1)\in U$, we have
	\begin{equation}\label{eq-fri}
		f\circ\gamma^I(1)=(1,I)\begin{pmatrix}
			1 & J\\1 & K
		\end{pmatrix}^{-1}\begin{pmatrix}
			f\circ\gamma^J(1)\\f\circ\gamma^K(1)
		\end{pmatrix}.
	\end{equation}
\end{defn}

\begin{prop}\label{uniqext}
	Let $\Omega\in\tau_s(\mathbb{H})$ be $2$-path-symmetric, $U\subset\Omega$, and $f:U\rightarrow\mathbb{H}$ be $\Omega$-path-slice. Then there is a unique $\Omega$-path-slice function $g:=f_{\scriptscriptstyle\Omega,\widetilde{2}}\, :\, U_{\scriptscriptstyle\Omega,\widetilde{2}}\rightarrow \mathbb{H}$ such that $g|_U=f$.
	\index{$f_{\scriptscriptstyle\Omega,\widetilde{2}}$}
\end{prop}

\begin{proof}
	(i) Denote
	\begin{equation*}
		\mathcal{U}:=\{[p]_{\approx_\Omega}\cap U\in P(\Omega):p\in U\}.
	\end{equation*}
	By Proposition \ref{pr-losmt} \ref{it-lp}, for each $\mu\in\mathcal{U}\subset P(\Omega)$, there is $z_{\mu}\in\mathbb{C}$ such that
	\begin{equation*}
		\mu= \left\{\left(z_{\mu}\right)^I\in U:I\in\mathbb{S}\right\}.
	\end{equation*}
	If $|\mu|\ge 2$, then $z_{\mu}\notin\mathbb{R}$ and we choose $I_{\mu},J_{\mu}\in\mathbb{S}$ with $z_{\mu}^{I_{\mu}},z_{\mu}^{J_{\mu}}\in \mu$ and $I_{\mu}\neq J_{\mu}$. For each $q\in[z_\mu^{I_\mu}]_{\approx_\Omega}$ with $|\mu|\ge 2$, denote
	\begin{equation*}
		z_q:=z_\mu,\qquad I_q:=I_{\mu},\qquad J_{q}:=J_{\mu}
	\end{equation*}
	and $L_q$ be the unique imaginary unit such that $q=z_q^{I_q}$. Define
	\begin{equation*}
		\begin{split}
			f_{\scriptscriptstyle\Omega,\widetilde{2}}:\ U_{\scriptscriptstyle\Omega,\widetilde{2}}\ &\xlongrightarrow[\hskip1cm]{}\ \mathbb{H},
			\\ q &\shortmid\!\xlongrightarrow[\hskip1cm]{}\
			\begin{cases}
				f(q),\qquad &q\in U,\\
				h(q),\qquad &\mbox{otherwise},
			\end{cases}
		\end{split}
	\end{equation*}
	where
	\begin{equation}\label{eq-hzl}
		h(q):=(1,L_q)\begin{pmatrix}
			1 & I_{q}\\1 & J_{q}
		\end{pmatrix}^{-1}\begin{pmatrix}
			f(z_{q}^{I_q}) \\ f(z_q^{J_q})
		\end{pmatrix}.
	\end{equation}
	
	Now we show that
	\begin{equation}\label{eq-f1iq}
		f_{\scriptscriptstyle\Omega,\widetilde{2}}(q)=(1,L_q)\begin{pmatrix}
			1 & I_{q}\\1 & J_{q}
		\end{pmatrix}^{-1}\begin{pmatrix}
			f(z_q^{I_{q}}) \\ f(z_q^{J_{q}})
		\end{pmatrix},\quad\forall\ q\in U_{\scriptscriptstyle\Omega,\widetilde{2}}\mbox{ with } |[q]_{\approx_\Omega}\cap U|\ge 2.
	\end{equation}
	Let $q\in U_{\scriptscriptstyle\Omega,\widetilde{2}}$ with $|[q]_{\approx_\Omega}\cap U|\ge 2$. If $q\notin U$, then \eqref{eq-f1iq} holds by \eqref{eq-hzl}. Otherwise, $q\in U$. It follows from Proposition \ref{pr-losmt} \ref{it-lzii} and $z_q^{I_q}, z_q^{J_q}\in[q]_{\approx_\Omega}$ that there is $\beta\in\mathscr{P}_0(\mathbb{C})$ such that
	\begin{equation}\label{eq-lijq}
		\beta^{L_q},\beta^{I_q},\beta^{J_q}\subset\Omega,\qquad \mbox{ and }\qquad \beta(1)=z_q.
	\end{equation}
	Since $f$ is $\Omega$-path-slice,
	\begin{equation*}
		\begin{split}
			f_{\scriptscriptstyle\Omega,\widetilde{2}}(q)=f(q)=f(\beta^{L_q}(1))
			=&(1,L_q)\begin{pmatrix}
				1 & I_{q}\\1 & J_{q}
			\end{pmatrix}^{-1}\begin{pmatrix}
				f(\beta^{I_q}(1)) \\ f(\beta^{J_q}(1))
			\end{pmatrix}
			\\=&(1,L_q)\begin{pmatrix}
				1 & I_{q}\\1 & J_{q}
			\end{pmatrix}^{-1}\begin{pmatrix}
				f(z_{q}^{I_q}) \\ f(z_q^{J_q})
			\end{pmatrix}.
		\end{split}
	\end{equation*}
	Thence \eqref{eq-f1iq} holds.
	
	Let $\gamma\in\mathscr{P}_0(\mathbb{C})$, $I,J,K\in\mathbb{S}$ with $\gamma^I,\gamma^J,\gamma^K\subset\Omega$, $J\neq K$ and $\gamma^I(1),\gamma^J(1),\gamma^K(1)\in U_{\scriptscriptstyle\Omega,\widetilde{2}}$. If $w:=\gamma(1)\notin\mathbb{R}$, then $|\mu|\ge|\{w^J,w^K\}|\ge 2$, where $\mu:=[w^I]_{\approx_\Omega}\cap U$. Then there is $\lambda\in\{0,1\}$ such that
	\begin{equation*}
		z_\mu={\Conj}_{\mathbb{C}}^\lambda(w),\qquad\mbox{and}\qquad z_\mu^{(-1)^\lambda \Gamma}=w^\Gamma,\qquad \Gamma=I,J,K.
	\end{equation*}
	It follows from \eqref{eq-f1iq} that
	\begin{equation}\label{eq-cont}
		\begin{split}
			&(1,I)\begin{pmatrix}
				1&J\\1&K
			\end{pmatrix}^{-1}\begin{pmatrix}
				f_{\scriptscriptstyle\Omega,\widetilde{2}}(\gamma^J(1))\\f_{\scriptscriptstyle\Omega,\widetilde{2}}(\gamma^K(1))
			\end{pmatrix}
			=(1,I)\begin{pmatrix}
				1&J\\1&K
			\end{pmatrix}^{-1}\begin{pmatrix}
				f_{\scriptscriptstyle\Omega,\widetilde{2}}(w^J)\\f_{\scriptscriptstyle\Omega,\widetilde{2}}(w^K)
			\end{pmatrix}
			\\=&(1,I)\begin{pmatrix}
				1&J\\1&K
			\end{pmatrix}^{-1}\begin{pmatrix}
				1&J\\1&K
			\end{pmatrix}\begin{pmatrix}
				1\\&(-1)^\lambda
			\end{pmatrix}
			\begin{pmatrix}
				1 & I_{\mu}\\1 & J_{\mu}
			\end{pmatrix}^{-1}\begin{pmatrix}
				f(z_{\mu}^{I_{\mu}}) \\ f(z_{\mu}^{J_{\mu}})
			\end{pmatrix}
			\\=&(1,I)\begin{pmatrix}
				1\\&(-1)^\lambda
			\end{pmatrix}\begin{pmatrix}
				1 & I_{\mu}\\1 & J_{\mu}
			\end{pmatrix}^{-1}\begin{pmatrix}
				f(z_\mu^{I_{\mu}}) \\ f(z_\mu^{J_{\mu}})
			\end{pmatrix}
			\\=&\left(1,(-1)^\lambda I\right)\begin{pmatrix}
				1 & I_{\mu}\\1 & J_{\mu}
			\end{pmatrix}^{-1}\begin{pmatrix}
				f(z_\mu^{I_{\mu}}) \\ f(z_\mu^{J_{\mu}})
			\end{pmatrix}
			=f_{\scriptscriptstyle\Omega,\widetilde{2}}(w^I)
			=f_{\scriptscriptstyle\Omega,\widetilde{2}}(\gamma^I(1)).
		\end{split}
	\end{equation}
	Otherwise, $w\in\mathbb{R}$. Then $\gamma^\Gamma(1)=w$, $\Gamma=I,J,K$. It implies that
	\begin{equation}\label{eq-cont1}
		\begin{split}
			(1,I)\begin{pmatrix}
				1&J\\1&K
			\end{pmatrix}^{-1}\begin{pmatrix}
				f_{\scriptscriptstyle\Omega,\widetilde{2}}(\gamma^J(1))\\f_{\scriptscriptstyle\Omega,\widetilde{2}}(\gamma^K(1))
			\end{pmatrix}
			=&(1,I)\begin{pmatrix}
				1&J\\1&K
			\end{pmatrix}^{-1}\begin{pmatrix}
				1&J\\1&K
			\end{pmatrix}\begin{pmatrix}
				f_{\scriptscriptstyle\Omega,\widetilde{2}}(w)\\0
			\end{pmatrix}
			\\=&(1,I)\begin{pmatrix}
				f_{\scriptscriptstyle\Omega,\widetilde{2}}(w)\\0
			\end{pmatrix}
			=f_{\scriptscriptstyle\Omega,\widetilde{2}}(\gamma^I(1)).
		\end{split}
	\end{equation}
	According \eqref{eq-cont} and \eqref{eq-cont1}, $f_{\scriptscriptstyle\Omega,\widetilde{2}}$ is $\Omega$-path-slice.
	
	(ii) Suppose that $h:U_{\scriptscriptstyle\Omega,\widetilde{2}}\rightarrow\mathbb{H}$ is $\Omega$-path-slice with $h|_U=f$. Let $q\in U_{\scriptscriptstyle\Omega,\widetilde{2}}\backslash U$. It is easy to check that $|[q]_{\approx_\Omega}\cap U|\ge 2$. According to Proposition \ref{pr-losmt} \ref{it-lzii}, there is $\beta\in\mathscr{P}_0(\mathbb{C})$ such that \eqref{eq-lijq} holds. By \eqref{eq-f1iq},
	\begin{equation*}
		\begin{split}
			h(q)=&h(\beta^{L_q}(1))
			=(1,L_q)\begin{pmatrix}
				1 & I_q\\ 1& J_q
			\end{pmatrix}^{-1}
			\begin{pmatrix}
				h(\beta^{I_q}(1))\\h(\beta^{J_q}(1))
			\end{pmatrix}
			\\=&(1,L_q)\begin{pmatrix}
				1 & I_q\\ 1& J_q
			\end{pmatrix}^{-1}
			\begin{pmatrix}
				f(z_q^{I_q})\\f(z_q^{J_q})
			\end{pmatrix}
			=f_{\scriptscriptstyle\Omega,\widetilde{2}}(q).
		\end{split}
	\end{equation*}
	It implies from $h|_U=f=f_{\scriptscriptstyle\Omega,\widetilde{2}}|_U$ that $h=f_{\scriptscriptstyle\Omega,\widetilde{2}}$.
\end{proof}

Recalling the notion of analytic set, see e.g. \cite[Definition 2.3.1]{Noguchi2016001}, we prove the next result:

\begin{lem}\label{pr-lopbar}
	Let $(\Omega,\pi)$ be a Riemann domain over $\mathbb{C}$, $K\subset\Omega$ be compact with $K_\Omega^\land= K$, $A$ be an analytic set with $A\neq\Omega$, $z\in A\backslash K$ and $\varepsilon>0$. Then there is $f\in\mathscr{O}(\Omega)$ such that
	\begin{equation}\label{eq-f1fab}
		f(z)=1,\qquad f|_{A\backslash\{z\}}\equiv 0,\qquad\mbox{and}\qquad\sup_{K}|f|<\varepsilon.
	\end{equation}
\end{lem}

\begin{proof}
	Since $A$ is an analytic set with $A\neq\Omega$, $z$ is an isolated point of $A$. Therefore the function
	\begin{equation*}
		\begin{split}
			g:\ A \ &\xlongrightarrow[\hskip1cm]{}\ \mathbb{C},
			\\ \xi\ &\shortmid\!\xlongrightarrow[\hskip1cm]{}\
			\begin{cases}
				1,\qquad &\xi=z,\\
				0,\qquad &\mbox{otherwise},
			\end{cases}
		\end{split}
	\end{equation*}
	is holomorphic. Since every non-compact Riemann surface is a Stein manifold, the Riemann domain $\Omega$ over $\mathbb{C}$ is a Stein manifold. By interpolation theorem (e.g. \cite[Theorem 4.5.11]{Noguchi2016001}), there is $\widetilde{g}\in\mathscr{O}(\Omega)$ such that $\widetilde{g}|_A=g$. Since $z\notin K=K_\Omega^\land$, there is $h\in\mathscr{O}(\Omega)$ such that $|h(z)|>\sup_{K}|h|$. Then there is $\lambda\in\mathbb{N}$ such that
	\begin{equation*}
		\sup_{K}\left|\left(\frac{h}{h(z)}\right)^\lambda\right|<\frac{\varepsilon}{\max_K|\widetilde{g}|+1}.
	\end{equation*}
	It is easy to check that $f:=\widetilde{g}\cdot\left(\frac{h}{h(z)}\right)^\lambda\in\mathscr{O}(\Omega)$ satisfies \eqref{eq-f1fab}.
\end{proof}

\begin{lem}\label{pr-loitkso}
	Let $\Omega\in\tau_s(\mathbb{H})$ be $2$-path-symmetric, $K\subset\Omega$ be invariant, $\varepsilon>0$, $A\subset\Omega$ be stem-finite, and $g:A\rightarrow\mathbb{H}$ be an $\Omega$-path-slice function with $A\backslash[p]_{\approx_\Omega}, A_{\scriptscriptstyle\Omega,\widetilde{2}}\cap K\subset\mathcal{Z}(g_{\scriptscriptstyle\Omega,\widetilde{2}})$ for some $p\in A$. Then there is $f\in\mathcal{SR}(\Omega)$ such that
	\begin{equation}\label{eq-faqm}
		f|_A=g,\qquad\mbox{and}\qquad \sup_K|f|<\varepsilon.
	\end{equation}
\end{lem}

\begin{proof}
	If $A\cap[p]_{\approx_\Omega}\subset K$, then $g\equiv 0$ and \eqref{eq-faqm} holds for $f\equiv 0$. Otherwise, let
	\begin{equation}\label{eq-ziap}
		z^I\in A\cap[p]_{\approx_\Omega}\backslash K,\qquad\mbox{for some}\quad z\in\mathbb{C}\quad\mbox{and}\quad I\in\mathbb{S}.
	\end{equation}
	Denote $\mu:=[(z^I,I)]_{\sim_\Omega}$. By \eqref{eq-rge0}, $r:=\sigma(K,\mathbb{H}\backslash\Omega)>0$.
	
	(i) If $|[p]_{\approx_\Omega}|=1$, then $[p]_{\approx_\Omega}=\{z^I\}$. Since $A$ is stem-finite,
	\begin{equation*}
		|A_{\scriptscriptstyle\breve{\Omega}}^{\scriptscriptstyle[I]}|=|\psi_{\scriptscriptstyle\Omega}^{\scriptscriptstyle\mathbb{S}\backslash\{-I\}} (A)|\le|\psi_{\scriptscriptstyle\Omega}^{\scriptscriptstyle\mathbb{S}} (A)|<\infty.
	\end{equation*}
	Thence $A_{\scriptscriptstyle\breve{\Omega}}^{\scriptscriptstyle[I]}$ is an analytic set in the Riemann domain $\breve{\Omega}^{\scriptscriptstyle[I]}$. Proposition
	\ref{pr-loitsm} implies that $\mu\notin K_{\breve{\scriptscriptstyle{\Omega}}}^{r,I}\subset\psi_{\scriptscriptstyle \Omega}^\mathbb{S}(K)$. According to Proposition \ref{pr-loitsh}, $K_{\breve{\scriptscriptstyle{\Omega}}}^{r,I}=\left( K_{\breve{\scriptscriptstyle{\Omega}}}^{r,I}\right)_{\breve{\Omega}^{\scriptscriptstyle[I]}}^\land$. According to Lemma \ref{pr-lopbar}, there is $h_1\in\mathscr{O}(\breve{\Omega}^{\scriptscriptstyle[I]})$ such that
	\begin{equation}\label{eq-hhh}
		h_1(\mu)=1,\qquad h_1|_{A_{\scriptscriptstyle\Omega}^{\scriptscriptstyle[I]}\backslash\{\mu\}}\equiv 0,\qquad\mbox{and}\qquad\sup_{K_{\breve{\scriptscriptstyle{\Omega}}}^{r,I}} |h_1|<\frac{\varepsilon}{|g(z^I)|+1}.
	\end{equation}
	Let
	\begin{equation}\label{eq-fh1}
		f:=(h_1)_{\scriptscriptstyle\Omega}^I\cdot g(z^I).
	\end{equation}
	Then
	\begin{equation*}
		\begin{split}
			f(z^I)=&(h_1)_{\scriptscriptstyle\Omega}^I(z^I)\cdot g(z^I)
			=\Psi_i^I((h_1)_\mu(z))\cdot g(z^I)
			\\=&\Psi_i^I(h_1(\mu))\cdot g(z^I)
			=\Psi_i^I(1)\cdot g(z^I)
			=g(z^I).
		\end{split}
	\end{equation*}
	If $A=\{z^I\}$, then $f|_A=g$. Otherwise, there is $w^L\in A\backslash [p]_{\approx_\Omega}=A\backslash\{z^I\}\subset \mathcal{Z}(h_1)$ with $L\in\mathbb{S}\backslash\{-I\}$, and $\nu:=[(w^L,L)]_{\sim_\Omega}$. According to \eqref{eq-sim}, $\nu\notin\{\mu,\overline{\mu}\}$. It follows from
	\begin{equation*}
		\nu=[(w^L,L)]_{\sim_\Omega}
		\in [A\times\{L\}]_{\sim_\Omega}
		\subset\psi_{\scriptscriptstyle\Omega}^{\scriptscriptstyle \mathbb{S}\backslash\{-I\}}(A)=A_{\breve{\scriptscriptstyle\Omega}}^{\scriptscriptstyle [I]}
	\end{equation*}
	that $\nu\in A_{\breve{\scriptscriptstyle\Omega}}^{\scriptscriptstyle[I]}\backslash\{\mu\}$. If $J=I$, then
	\begin{equation*}
		\begin{split}
			f(w^J)=&(h_1)_{\scriptscriptstyle\Omega}^I(w^J)\cdot g(z^I)
			=\Psi_i^I((h_1)_\nu(w))\cdot g(z^I)
			\\=&\Psi_i^I(h_1(\nu))\cdot g(z^I)
			=\Psi_i^I(0)\cdot g(z^I)
			=0=g_{\scriptscriptstyle\Omega,\widetilde{2}}(w^J)=g(w^J).
		\end{split}
	\end{equation*}
	Otherwise, $J\neq I$. Then
	\begin{equation*}
		\overline\nu=[(w^L,-L)]_{\sim_\Omega}\subset[U_{-L}\times\{-L\}]_{\sim_\Omega}\subset\psi_{\scriptscriptstyle\Omega}^{\scriptscriptstyle \mathbb{S}\backslash\{-I\}}(A)=A_{\breve{\scriptscriptstyle\Omega}}^{\scriptscriptstyle [I]},
	\end{equation*}
	implies that $\overline{\nu}\in A_{\breve{\scriptscriptstyle\Omega}}^{\scriptscriptstyle [I]}\backslash\{\mu\}$ and
	\begin{equation*}
		 \nu\in{\Conj}_{\breve{\scriptscriptstyle\Omega}}( A_{\breve{\scriptscriptstyle\Omega}}^{\scriptscriptstyle [I]}\backslash\{\mu\})\subset{\Conj}_{\breve{\scriptscriptstyle\Omega}} (\mathcal{Z}(h_1)).
	\end{equation*}
	By \eqref{eq-vblmsm},
	\begin{equation*}
		w^J\in \varphi_{\breve{\scriptscriptstyle\Omega}}^0
		\left(\{\nu\}\right)
		\subset \varphi_{\breve{\scriptscriptstyle\Omega}}^0
		\left(\mathcal{Z}(h_1)\cap {\Conj}_{\breve{\scriptscriptstyle\Omega}}\circ\mathcal{Z}(h_1)\right)
		\subset \mathcal{Z}((h_1)_{\scriptscriptstyle\Omega}^I).
	\end{equation*}
	It is clear that
	\begin{equation*}
		f(w^J)=(h_1)_{\scriptscriptstyle\Omega}^I(w^J)\cdot g(z^I)=0\cdot g(z^I)=0=g(w^J).
	\end{equation*}
	In summary,
	\begin{equation}\label{eq-faga}
		f|_{A\backslash[p]_{\approx_\Omega}}\equiv 0 \equiv g|_{A\backslash[p]_{\approx_\Omega}}\qquad\mbox{and}\qquad f|_A=g.
	\end{equation}
	
	By Proposition \ref{pr-U I in U sigma I}, $K_{\breve{\scriptscriptstyle\Omega}}^{[I]}\subset K_{\breve{\scriptscriptstyle\Omega}}^{r,I}$. It follows from Proposition \ref{pr-loitps} and \eqref{eq-hhh} that
	\begin{equation*}
		\sup_K \left|(h_1)_{\scriptscriptstyle{\Omega}}^{\scriptscriptstyle{I}}\right|
		\le\sup_{K_{\scriptscriptstyle{\breve{\Omega}}}^{\scriptscriptstyle [I]}}|h_1|
		\le\sup_{K_{\scriptscriptstyle{\breve{\Omega}}}^{r,I}}|h_1|
		<\frac{\varepsilon}{|g(z^I)|+1}.
	\end{equation*}
	According to \eqref{eq-fh1}, $\sup_K \left|f\right|< \varepsilon$. Hence \eqref{eq-faqm} holds.
	
	(ii) If $|[p]_{\approx_\Omega}|>1$ and $[p]_{\approx_\Omega}\cap K\neq\varnothing$, then by Lemma \ref{lem_Jexist},
	\begin{equation*}
		[p]_{\approx_\Omega}\cap K=\left\{z^{-J}\right\},\qquad\mbox{for some}\qquad J\in\mathbb{S},
	\end{equation*}
	and
	\begin{equation*}
		\mu=[(z^I,I)]_{\sim_\Omega}=[(z^J,J)]_{\sim_\Omega}\notin K_{\breve{\scriptscriptstyle{\Omega}}}^{[J]}.
	\end{equation*}
	Since $|[p]_{\approx_\Omega}|>1$, $z\notin \mathbb{R}$ and $\mu\neq\overline{\mu}$. According to \eqref{eq-kbso} and Lemma \ref{pr-lopbar}, there is $h_2\in\mathscr{O}(\breve{\Omega}^{\scriptscriptstyle[J]})$ such that
	\begin{equation}\label{eq-hhqm}
		h_2(\mu)=1,\qquad h_2|_{A_{\breve{\scriptscriptstyle\Omega}}^{\scriptscriptstyle[J]}\cup\{\overline{\mu}\}\backslash\{\mu\}}\equiv 0,\qquad\mbox{and}\qquad\sup_{K_{\breve{\scriptscriptstyle{\Omega}}}^{r,J}} |h_2|<\frac{\varepsilon}{|g(z^I)|+1}.
	\end{equation}
	
	Define
	\begin{equation}\label{eq-flh2r}
		f:=\left(h_2\right)_{\scriptscriptstyle{\Omega}}^J \cdot (1,J) \begin{pmatrix}
			1 & I\\ 1 & -J
		\end{pmatrix}^{-1}\begin{pmatrix}
			g(z^I)\\ 0
		\end{pmatrix}.
	\end{equation}
	By \eqref{eq-gioxm} and \eqref{eq-hhqm},
	\begin{equation*}
		\begin{cases*}
			\left(h_2\right)_{\scriptscriptstyle{\Omega}}^J(z^J)=h_2(\mu)=1,\\
			\left(h_2\right)_{\scriptscriptstyle{\Omega}}^J(z^{-J})=\left(h_2\right)_{\scriptscriptstyle{\Omega}}^J(\overline{z}^J)=h_2([(\overline{z}^J,J)]_{\sim_\Omega})=h_2(\overline{\mu})=0.
		\end{cases*}
	\end{equation*}
	It implies by  \eqref{eq-flh2r} that
	\begin{equation*}
		f(z^L)=(1,L)\begin{pmatrix}
			1 & I\\ 1 & -J
		\end{pmatrix}^{-1}\begin{pmatrix}
			g(z^I)\\ 0
		\end{pmatrix},\qquad L=\pm J,
	\end{equation*}
	and then
	\begin{equation*}
		\begin{split}
			f(z^I)
			=&(1,I)\begin{pmatrix}
				1&J\\1&-J
			\end{pmatrix}^{-1}\begin{pmatrix}
				f(z^J)\\f(z^{-J})
			\end{pmatrix}
			\\=&(1,I)\begin{pmatrix}
				1&J\\1&-J
			\end{pmatrix}^{-1}\begin{pmatrix}
				1&J\\1&-J
			\end{pmatrix}\begin{pmatrix}
				1 & I\\ 1 & -J
			\end{pmatrix}^{-1}\begin{pmatrix}
				g(z^I)\\ 0
			\end{pmatrix}=g(z^I).
		\end{split}
	\end{equation*}

	If $|[p]_{\approx_\Omega}\cap A|=1$, then by \eqref{eq-ziap}, $[p]_{\approx_\Omega}\cap A=\{z^I\}$. It implies that
	\begin{equation}\label{eq-fpsoca}
		f|_{[p]_{\approx_\Omega}\cap A}=g|_{[p]_{\approx_\Omega}\cap A}.
	\end{equation}
	
	Otherwise, $|[p]_{\approx_\Omega}\cap A|>1$. Then $p\in A_{\scriptscriptstyle\Omega,2}$ and $z^{-J}\in A_{\scriptscriptstyle\Omega,\widetilde{2}}\cap K\subset\mathcal{Z}(g_{\scriptscriptstyle\Omega,\widetilde{2}})$. It implies that
	\begin{equation*}
		\begin{cases*}
			f(z^I)=g(z^I)=g_{\scriptscriptstyle\Omega,\widetilde{2}}(z^I),\\
			f(z^J)=0=g(z^J).
		\end{cases*}
	\end{equation*}
	Since $f$ and $g_{\scriptscriptstyle\Omega,\widetilde{2}}$ are $\Omega$-path-slice, $f=g_{\scriptscriptstyle\Omega,\widetilde{2}}$ on $[p]_{\approx_\Omega}$. Hence \eqref{eq-fpsoca} holds. By similar method in (i), \eqref{eq-faga} holds and $\sup_K \left|f\right|< \varepsilon$. Hence \eqref{eq-faqm} holds.
	
	(iii) If $|[p]_{\approx_\Omega}|> 1$, $[p]\cap K=\varnothing$ and $|[p]_{\approx_\Omega}\cap A|=1$, then by \eqref{eq-ziap}, $[p]_{\approx_\Omega}\cap A=\{z^I\}$. By similar method in (i), there is $h_1\in\mathscr{O}(\breve{\Omega}^{\scriptscriptstyle[I]})$ satisfying \eqref{eq-hhh}, and $f:=(h_1)_{\scriptscriptstyle\Omega}^I\cdot g(z^I)$ satisfies \eqref{eq-faqm}.
	
	(iv) If $|[p]_{\approx_\Omega}|> 1$, $[p]\cap K=\varnothing$ and $|[p]_{\approx_\Omega}\cap A|>1$, then by Proposition \ref{pr-loitsm} \begin{equation*}
		\mu,\overline{\mu}\notin  K_{\breve{\scriptscriptstyle{\Omega}}}^{r,I}\cup K_{\breve{\scriptscriptstyle{\Omega}}}^{r,-I}\subset\psi_{\scriptscriptstyle \Omega}^\mathbb{S}(K).
	\end{equation*}
	By similar method in (i), there is $h_3\in\mathscr{O}(\breve{\Omega}^{\scriptscriptstyle[J]})$ and $h_4\in\mathscr{O}(\breve{\Omega}^{\scriptscriptstyle[-J]})$ such that
	\begin{equation*}
		h_3(\mu)=h_4(\overline{\mu})=1,\qquad h_3|_{A_{\scriptscriptstyle\Omega}^{\scriptscriptstyle[I]}\backslash\{\mu\}}\equiv 0,\qquad
		h_4|_{A_{\scriptscriptstyle\Omega}^{\scriptscriptstyle[-I]}\backslash\{\overline{\mu}\}}\equiv 0
	\end{equation*}
	\begin{equation*}
		\sup_{K_{\breve{\scriptscriptstyle{\Omega}}}^{r,I}} |h_3|<\frac{\varepsilon}{2|g(z^I)|+1},\qquad\mbox{and}\qquad
		\sup_{K_{\breve{\scriptscriptstyle{\Omega}}}^{r,-I}} |h_4|<\frac{\varepsilon}{2|g_{\scriptscriptstyle\Omega,\widetilde{2}}(z^{-I})|+1}.
	\end{equation*}
	Define $f_3:=2\cdot(h_3)_{\scriptscriptstyle\Omega}^I\cdot g(z^I)$, $f_4:=2\cdot (h_4)_{\scriptscriptstyle\Omega}^{-I}\cdot g_{\scriptscriptstyle\Omega,\widetilde{2}}(z^{-I})$ and $f:=\frac{f_3+f_4}{2}$. By similar method in (i),
	\begin{equation}\label{eq-fagaq}
		f'|_{A\backslash [p]_{\approx_\Omega}}
		=g|_{A\backslash [p]_{\approx_\Omega}}\equiv 0,
		\qquad\mbox{and}\qquad \sup_{K} |f'|<\epsilon,\qquad f'=f_3,f_4,f.
	\end{equation}
	By similar method in (ii), $f_3(z^I)=g(z^I)=2\cdot g_{\scriptscriptstyle\Omega,\widetilde{2}}(z^I)$, $f_4(z^{-I})=2\cdot g_{\scriptscriptstyle\Omega,\widetilde{2}}(z^{-I})$, and $f_3(z^{-I})=f_4(z^I)=0$. Thence
	\begin{equation*}
		\begin{cases*}
			f(z^I)=\frac{f_3(z^I)+f_4(z^I)}{2}=g_{\scriptscriptstyle\Omega,\widetilde{2}}(z^I),\\
			f(z^{-I})=\frac{f_3(z^I)+f_4(z^I)}{2}=g_{\scriptscriptstyle\Omega,\widetilde{2}}(z^{-I}).
		\end{cases*}
	\end{equation*}
	By similar method in (ii), $f|_{[p]_{\approx_\Omega}}=g_{\scriptscriptstyle\Omega,\widetilde{2}}|_{[p]_{\approx_\Omega}}$ and $f|_{A\cap [p]_{\approx_\Omega}}=g|_{A\cap [p]_{\approx_\Omega}}$. Hence \eqref{eq-faqm} hold by \eqref{eq-fagaq}.
\end{proof}

\begin{prop}\label{pr-l2psia}
	Let $\Omega\in\tau_s(\mathbb{H})$ be $2$-path-symmetric, $K\subset\Omega$ be invariant, $\varepsilon>0$, $A\subset\Omega$ be stem-finite, and $g:A\rightarrow\mathbb{H}$ be an $\Omega$-path-slice function with $g_{\scriptscriptstyle\Omega,\widetilde{2}}|_{A_{\scriptscriptstyle\Omega,\widetilde{2}}\cap K}=0$. Then there is $f\in\mathcal{SR}(\Omega)$ such that \eqref{eq-faqm} holds.
\end{prop}

\begin{proof}
	Since $A$ is $\Omega$-stem-finite, $\psi_{\scriptscriptstyle\Omega}^\mathbb{S}(A)$ is finite. Then there is $z_1,...,z_m\in\mathbb{C}$ and $I_1,...,I_m\in\mathbb{S}$ and such that
	\begin{equation*}
		\psi_{\scriptscriptstyle\Omega}^\mathbb{S}(A)
		=\bigcup_{\ell=1}^m \left\{\left[\left(z_\ell^{I_\ell},I_\ell\right)\right]_{\sim_\Omega}\right\}.
	\end{equation*}
	Let $u^L\in A$. Then $[(u^L,L)]_{\sim_\Omega}\in\psi_{\scriptscriptstyle\Omega}^\mathbb{S}(A)$. It implies that there is $\ell\in\{1,...,m\}$ such that $[(u^L,L)]_{\sim_\Omega}=\left[\left(z_\ell^{I_\ell},I_\ell\right)\right]_{\sim_\Omega}$, i.e. $(u^L,L)\sim_{\scriptscriptstyle \Omega}\left(z_\ell^{I_\ell},I_\ell\right)$. According to \eqref{eq-zaia}, $u^L\approx_{\scriptscriptstyle\Omega} z_\ell^{I_\ell}$, i.e. $u^L\in\left[z_\ell^{I_\ell}\right]_{\approx_\Omega}$. It implies that
	\begin{equation*}
		A\subset \bigcup_{\imath=1}^m\left\{\left[z_{\imath}^{I_{\imath}}\right]_{\approx_\Omega}\right\}.
	\end{equation*}
	Let $v_1,...,v_{m'}\in\{1,...,m\}$ such that
	\begin{equation*}
		\bigcup_{\imath=1}^{m'}\left\{\left[z_{v_\imath}^{I_{v_\imath}}\right]_{\approx_\Omega}\right\}=\bigcup_{\imath=1}^m\left\{\left[z_{\imath}^{I_{\imath}}\right]_{\approx_\Omega}\right\}\supset A,
	\end{equation*}
	and
	\begin{equation*}
		\left[z_{v_\imath}^{I_{v_\imath}}\right]_{\approx_\Omega}\cap\left[z_{v_\jmath}^{I_{v_\jmath}}\right]_{\approx_\Omega}
		=\varnothing,\qquad \imath\neq\jmath.
	\end{equation*}
	
	Let $\imath\in\{1,...,m'\}$. By Lemma \ref{pr-loitkso}, there is $f_\ell\in\mathcal{SR}(\Omega)$ such that
	\begin{equation*}
		f_\imath|_A=g_\imath,\qquad\mbox{and}\qquad\sup_K|f_\imath|<\frac{\varepsilon}{m'},
	\end{equation*}
	where $g_\imath:A\rightarrow\mathbb{H}$ defined by
	\begin{equation*}
		g_\imath(q):=\begin{cases}
			g(q),\qquad &q\in\left[z_{v_\imath}^{I_{v_\imath}}\right]_{\approx_\Omega},\\
			0,\qquad&\mbox{otherwise}.
		\end{cases}
	\end{equation*}
	It is easy to check that $f:=\sum_{\jmath=1}^{m'} f_\jmath$ satisfying \eqref{eq-faqm}.
\end{proof}

\begin{prop}\label{pr-increasing subsets}
	Let $\Omega\in\tau_s(\mathbb{H})$ be $2$-path-symmetric, then there are $\Omega$-invariant subsets $K_\ell$, $\ell\in\mathbb{N}$ of $\Omega$ such that
	\begin{equation}\label{eq-keis}
		(K_\ell)_I\subset (K_{\ell+1})_I^\circ,\qquad\bigcup_{\ell=0}^{+\infty} (K_\ell)^\circ_I=\Omega_I,\qquad\forall\ I\in\mathbb{S},
	\end{equation}
	where $(K_\ell)_I^\circ$ denote the set of interior points of $(K_\ell)_I$.
\end{prop}

\begin{proof}
	According to \cite[Remark 6.8]{Demailly2012001}, there are compact sets $K'_\ell \subset\breve{\Omega}$, $\ell\in\mathbb{N}$ such that
	\begin{equation*}
		K_\ell'\subset \left(K_{\ell+1}'\right)^\circ,\qquad (K_\ell')^\land_{\scriptscriptstyle\breve{\Omega}}=K_{\ell}',\qquad\mbox{and}\qquad\bigcup_{\ell=0}^{+\infty} \left(K_\ell'\right)^\circ=\breve{\Omega}.
	\end{equation*}
	By similar method of proof of \cite[Remark 6.8]{Demailly2012001}, there are $\breve{\Omega}$-invariant subsets $K''_\ell$, $\ell\in\mathbb{N}$ of $\breve{\Omega}$ such that
	\begin{equation}\label{eq-K ell property}
		K''_\ell\subset (K''_{\ell+1})^\circ,
		\qquad\mbox{and}\qquad
		\bigcup_{\ell=0}^{+\infty} \left(K''_\ell\right)^\circ=\breve{\Omega}:
	\end{equation}
	We may define $K_\ell''$ inductively by
	\begin{equation*}
		K_0'':=\left[K_0'\cap \Conj_{\breve{\Omega}}(K_0')\right]_{\breve{\Omega}}^\land
	\end{equation*}
	$\imath_\ell > \ell$ being the smallest nature number such that $K_\ell''\subset \left(K'_{\imath_\ell}\right)^\circ$, and
	\begin{equation}\label{eq-def K ell}
		K''_{\ell+1}:=\left[K_{\imath_\ell}'\cap \Conj_{\breve{\Omega}}(K_{\imath_\ell}')\right]_{\breve{\Omega}}^\land.
	\end{equation}
	It is easy to check that $K''_\ell$ is compact and
	\begin{equation*}
		K''_\ell=(K''_\ell)_{\scriptscriptstyle\breve{\Omega}}^\land={\Conj}_{\scriptscriptstyle\breve{\Omega}}(K_\ell'').
	\end{equation*}
	Therefore, $K_\ell$ is $\breve{\Omega}$-invariant. According to \eqref{eq-def K ell}, $K''_\ell\subset \left(K'_{\imath_\ell}\right)^\circ \subset \left(K''_{\ell+1}\right)^\circ$ and
	\begin{equation}\label{eq-union K'' to breve Omega}
		\bigcup_{\ell=0}^{+\infty} \left(K''_\ell\right)^\circ
		\supset
		\bigcup_{\ell=0}^{+\infty} \left(K'_{\imath_\ell}\right)^\circ
		\supset
		\bigcup_{\ell=0}^{+\infty} \left(K'_{\ell}\right)^\circ
		=\breve{\Omega}.
	\end{equation}
	Therefore, \eqref{eq-K ell property} holds.
	
	(i) Let
	\begin{equation*}
		K_\ell:=\varphi_{\breve{\scriptscriptstyle\Omega}}^{r_\ell}(K''_\ell),\qquad\mbox{where}\qquad r_\ell:=\frac{r^+(K''_\ell)}{\ell+1}.
	\end{equation*}
	Since $K''_\ell$ is $\breve{\Omega}$-invariant, $K''_\ell=\Conj_{\breve{\scriptscriptstyle\Omega}}(K''_\ell)$. It follows from \eqref{eq-cvsrc} and
	\begin{equation*}
		K''_\ell\subset\Omega_+^{r^+(K''_\ell)}\subset\Omega_+^{r_\ell}.
	\end{equation*}
	that
	\begin{equation*}
		\begin{split}
			\psi_{\scriptscriptstyle\Omega}^\mathbb{S}(K_\ell)
			=&\psi_{\scriptscriptstyle\Omega}^\mathbb{S}\circ\varphi_{\breve{\scriptscriptstyle\Omega}}^{r_\ell}(K''_\ell)
			=\phi_{\breve{\scriptscriptstyle\Omega}}^{r_\ell}(K''_\ell)
			\\=&\left[K''_\ell\cup\Conj_{\breve{\scriptscriptstyle\Omega}}(K''_\ell)\right]\cap\Omega_+^{r_\ell}
			=K''_\ell\cap\Omega_+^{r_\ell}
			=K''_\ell.
		\end{split}
	\end{equation*}
	According to \eqref{eq-cvsrc},
	\begin{equation*}
		\phi_{\scriptscriptstyle\Omega}^{r_\ell}(K_\ell)=\varphi_{\scriptscriptstyle\breve{\Omega}}^{r_\ell}\circ \psi_{\scriptscriptstyle\Omega}^{\mathbb{S}}(K_\ell)=\varphi_{\scriptscriptstyle\breve{\Omega}}^{r_\ell}(K''_\ell)=K_\ell.
	\end{equation*}
	It implies that $K_\ell$ is $\Omega$-invariant.
	
	(ii) Let $I\in\mathbb{S}$ and $q\in (K_\ell)_I$. Then $q=z^I$, for some $z\in\mathbb{C}$. Then
	\begin{equation*}
		\mu:=[(z^I,I)]_{\sim_\Omega}\in [(K_\ell)_I\times\{I\}]_{\sim_\Omega}\subset\psi_{\scriptscriptstyle\Omega}^\mathbb{S}(K_\ell)=K''_\ell\subset \left(K''_{\ell+1}\right)^\circ.
	\end{equation*}
	Since $z^I\in K_\ell=\varphi_{\breve{\scriptscriptstyle\Omega}}^{r_\ell}(K''_\ell)\subset\Omega^{r_\ell}$,
	\begin{equation*}
		\begin{split}
			r_\mu^I
			=&\delta_I(z_I,\mathbb{C}_I\backslash\Omega_I)
			=\sigma(z_I,\mathbb{H}\backslash\Omega)\ge r_\ell
			\\=&\frac{r^+(K''_\ell)}{\ell+1}>\frac{r^+(K''_\ell)}{\ell+2}
			\ge \frac{r^+(K''_{\ell+1})}{(\ell+1)+1}=r_{\ell+1}.
		\end{split}
	\end{equation*}
	Therefore, there is $r\in\left(0,r_\mu^I-r_{\ell+1}\right)$ such that $B_{\breve{\scriptscriptstyle\Omega}}(\mu,r)\subset \left(K''_{\ell+1}\right)^\circ$. Then for each $w^I\in B_I(z^I,r)$,
	\begin{equation*}
		[(w^I,I)]_{\sim_\Omega}\subset[B_I(z^I,r)\times\{I\}]_{\sim_\Omega}=B_{\breve{\scriptscriptstyle\Omega}}(\mu,r)\subset \left(K''_{\ell+1}\right)^\circ\subset K''_{\ell+1}.
	\end{equation*}
	It implies that $[(w^I,I)]_{\sim_\Omega}\in\psi_{\scriptscriptstyle\Omega}^\mathbb{S}(\{w^I\})\cap \left(K''_{\ell+1}\right)\neq\varnothing$ and
	\begin{equation*}
		\sigma(w^I,\mathbb{H}\backslash\Omega)\ge \sigma (z^I,\mathbb{H}\backslash\Omega)-\sigma(z^I,w^I)>r_\mu^I-r >r_{\ell+1}.
	\end{equation*}
	Hence $w^I\in\Omega^{r_{\ell+1}}$. By \eqref{eq-definitions}, $w^I\in\varphi_{\scriptscriptstyle\Omega}^{r_{\ell+1}}(K''_{\ell+1})=K_{\ell+1}$. It follows that $B_I(z^I,r)\subset (K_{\ell+1})_I$ and $z^I\in(K_{\ell+1})_I^\circ$. Therefore, $(K_\ell)_I\subset(K_{\ell+1})_I^\circ$.
	
	(iii) Let $z^I\in\Omega_I$. Then there is $\ell\in\mathbb{N}$ such that $\sigma(z^I,\mathbb{H}\backslash\Omega)>\frac{r^+(K''_0)}{\ell+1}$.
	Then
	\begin{equation*}
		\sigma(z^I,\mathbb{H}\backslash\Omega)>\frac{r^+(K''_0)}{\ell+1}\ge\frac{r^+(K''_{m})}{m+1}=r_m,\qquad\forall\ m>\ell.
	\end{equation*}
	According to \eqref{eq-union K'' to breve Omega}, there is $\imath>\ell$ such that $[(z^I,I)]_{\sim_\Omega}\in K''_\imath$. By similar method in (ii), $z^I\in K_\imath\subset(K_{\imath+1})^\circ$. Then $z^I\in(K_{\imath+1})^\circ$ and then $z^I\in(K_{\imath+1})_I^\circ$, for each $z^I\in\Omega_I$. It implies that $\bigcup_{\ell=0}^{+\infty} (K_\ell)^\circ_I=\Omega_I$. Therefore, Proposition \ref{pr-increasing subsets} holds by (i)-(iii).
\end{proof}

The Weierstrass factorization theorem for entire slice regular function is studied in \cite{MR2836832}, see also \cite{bookentire}. Some interpolation theorems for slice regular functions defined on axially symmetric domains are treated in \cite{MR3320522, MR4663899, MR3254690}.  Now we give an interpolation theorem for slice regular function defined on $2$-path-symmetric slice-open sets.

\begin{thm}(Interpolation Theorem)\label{th-it}
	Let $\Omega\in\tau_s(\mathbb{H})$ be $2$-path-symmetric, $A\subset\Omega$ be stem discrete, and $g:A\rightarrow\mathbb{H}$ be an $\Omega$-path-slice function. Then there is $f\in\mathcal{SR}(\Omega)$ such that $f|_A=g$.
\end{thm}

\begin{proof}
	According to Proposition \ref{pr-increasing subsets}, there are $\Omega$-invariant subsets $K_\ell$, $\ell\in\mathbb{N}$ of $\Omega$ such that \eqref{eq-keis} holds. Define
	\begin{equation*}
		A_\ell:=(A\cap K_{\ell+1})\backslash K_\ell,
	\end{equation*}
	and $g_\ell:\cup_{\imath=1}^\ell A_\imath\rightarrow\mathbb{H}$ by
	\begin{equation*}
		g_\ell(q)=\begin{cases}
			g(q)-\sum_{\imath=1}^{\ell-1}f_\imath(q),\qquad& q\ \in A_\ell,\\
			0, &\mbox{otherwise}.
		\end{cases}
	\end{equation*}
	Here, by Proposition \ref{pr-l2psia}, we can construct a function $f_\ell\in\mathcal{SR}(\Omega)$ such that
	\begin{equation*}
		f_\ell|_{\cup_{\imath=1}^\ell A_\imath}=g_\ell,\qquad\mbox{and}\qquad \sup_{K_{\ell-1}}|f_\ell|<\frac{\varepsilon}{2^\ell}.
	\end{equation*}
	Let $f:=\sum_{\imath=1}^{+\infty} f_\imath$. Since $\sum_{\imath=1}^{+\infty} f_\imath$ converges uniformly on $(K_\jmath)_I$, $\jmath\in\mathbb{N}\backslash\{0\}$, we have $f_I$ is $I$-holomorphic. It implies that $f$ is slice regular.
	
	Let $a\in A$. Then $a\in A_\imath$ for some $\imath\in\mathbb{N}$. Then
	\begin{equation*}
		\begin{split}
			f(a)
			=&\sum_{\jmath=1}^{+\infty} f_\jmath(a)
			=\sum_{\jmath=1}^{\imath-1}f_\jmath(a)+f_\imath(a)+\sum_{\jmath=\imath+1}^{+\infty} f_\jmath(a)
			\\=&\sum_{\jmath=1}^{\imath-1}f_\jmath(a)+\left[g(a)-\sum_{\jmath=1}^{\imath-1}f_\jmath(a)\right]+0
			=g(a).
		\end{split}
	\end{equation*}
	It implies that $f|_A=g$. Hence Theorem \ref{th-it} holds.
\end{proof}

\section{Proof of Cartan-Thullen Theorem}\label{sec-Cartan-Thullen}

In this section, we will prove the counterpart of the Cartan-Thullen theorem in the theory of slice regular functions. As in the complex case, we need a notion of domain of existence, see Definition \ref{dommex}. The preparation results provided in the previous sections together with the description of slice regularly convex domains in the next proposition allow to prove the main Theorem \ref{thm-ct}.

\begin{prop}\label{eq-oits}
A set	$\Omega\in\tau_s(\mathbb{H})$ is slice regularly convex, if and only if $\Omega$ is $2$-path-symmetric.
\end{prop}

\begin{proof}
Let $\Omega$ be a slice regularly convex slice-open set and let us suppose, by absurd, that $\Omega$ is not $2$-path-symmetric. Then there exists $\gamma\in\mathscr{P}_0(\mathbb{C})$ and $I_1,I_2,J\in\mathbb{S}$, $I_1\neq I_2$, such that $\gamma^{I_1},\gamma^{I_2}\subset\Omega$ but $\gamma^J\nsubseteq\Omega$. Since $\gamma^J(0)=\gamma^{I_1}(0)\in\Omega$, there is $t\in(0,1]$ such that  $\gamma^J\big([0,t)\big)\subset\Omega$ and $\gamma^J\big([0,t]\big)\nsubseteq\Omega$. Let
	\begin{equation*}
		K:=\gamma^{I_1}([0,1])\cup\gamma^{I_2}([0,1])
	\end{equation*}
	be a finite-non-symmetric compact subset. Then $\gamma^{I_1},\gamma^{I_2}\subset K$. By the representation formula,
	\begin{equation*}
		\sup_{\gamma^{J}([0,t))}|f|\le |c_1|\sup_{\gamma^{I_1}([0,t))}|f|+|c_2|\sup_{\gamma^{I_2}([0,t))}|f|,\qquad\forall\ f\in\mathcal{SR}(\Omega),
	\end{equation*}
	where
	\begin{equation*}
		(c_1,c_2)=(1,J)\begin{pmatrix}
			1&I_1\\1&I_2
		\end{pmatrix}^{-1}.
	\end{equation*}
	It implies that $\gamma^{J}([0,t))\subset \widehat{K}_\Omega\cap\Omega_J=(\widehat{K}_\Omega)_J$ and $(\widehat{K}_\Omega)_J$ is not relatively compact in $\Omega_J$. Thus $\Omega$ is not slice regularly convex, contradicting the hypothesis.
	
To prove the converse, suppose that $\Omega$ is $2$-path-symmetric and let $K\subset\Omega$ be a finite-non-symmetric compact. According to \eqref{eq-o2ps1}, $K_{2ps}\subset\Omega_{2ps}$. By Remark \ref{rmk-2ps} \ref{it-Omega 2ps}, $K^{2ps}\subset\Omega^{2ps}$. Let $I,J\in\mathbb{S}\backslash\mathbb{S}_K^{2ps}$ and $\mu\in [K_I\times\{I\}]_{\sim_\Omega}$. According to \eqref{eq-s omega 2ps}, $z_\mu^{I}\in K_I=(K_{2ps})^I\subset(\Omega_{2ps})^I$. Therefore $z_\mu\in K_{2ps}\subset \Omega_{2ps}$, and $z_\mu^J\in (K_{2ps})^J=K_J$. By \eqref{eq-two}, $(z_\mu^I,I)\sim_\Omega(z_\mu^J,J)$ and
	\begin{equation*}
		[(z_\mu^I,I)]_{\sim_\Omega}=[(z_\mu^J,J)]_{\sim_\Omega}\subset [K_J\times\{J\}]_{\sim_\Omega}.
	\end{equation*}
	It implies that $[K_I\times\{I\}]_{\sim_\Omega}\subset[K_J\times\{J\}]_{\sim_\Omega}$. Similarly, $[K_J\times\{J\}]_{\sim_\Omega}\subset[K_I\times\{I\}]_{\sim_\Omega}$. Hence
	\begin{equation}\label{eq-eq}
		[K_I\times\{I\}]_{\sim_\Omega}=[K_J\times\{J\}]_{\sim_\Omega},\qquad\forall\ I,J\in \mathbb{S}\backslash\mathbb{S}_{K}^{2ps}.
	\end{equation}
	Let $L\in\mathbb{S}$. Since $K_L$ is compact in $\tau(\mathbb{C}_L)$, $K_L\times {L}$ is compact in $\tau(\mathbb{C}_L)\times\{L\}$. According to Proposition \ref{pr-homeomorphism}, $\imath_{\scriptscriptstyle \Omega}^I$ is a homeomorphism. Therefore $[K_I\times\{I\}]_{\sim_\Omega}=\imath_{\scriptscriptstyle \Omega}^I(K_L\times {L})$ is compact in $[\Omega_I\times\{I\}]_{\sim_\Omega}$ and also compact in $\breve{\Omega}\supset[\Omega_I\times\{I\}]_{\sim_\Omega}$. Let $J_0\in\mathbb{S}\backslash\mathbb{S}_K^{2ps}$. Since $\mathbb{S}_K^{2ps}$ is finite and \eqref{eq-eq},
	\begin{equation*}
		\begin{split}
			\psi_{\scriptscriptstyle\Omega}^\mathbb{S}(K)=&\bigcup_{I\in\mathbb{S}}[K_I\times\{I\}]_{\sim_\Omega}
			\\=&\left(\bigcup_{I\in\mathbb{S}_K^{2ps}}[K_I\times\{I\}]_{\sim_\Omega}\right)\bigcup\left(\bigcup_{I\in\mathbb{S}\backslash\mathbb{S}_K^{2ps}}[K_I\times\{I\}]_{\sim_\Omega}\right)
			\\=&\left(\bigcup_{I\in\mathbb{S}_K^{2ps}}[K_I\times\{I\}]_{\sim_\Omega}\right)\bigcup[K_{J_0}\times\{J_0\}]_{\sim_\Omega}
		\end{split}
	\end{equation*}
	is compact in $\breve{\Omega}$. It follows from $\Conj_{\breve{\scriptscriptstyle \Omega}}\left([K_I\times\{I\}]_{\sim_\Omega}\right)=[K_{-I}\times\{-I\}]_{\sim_\Omega}$ that  $\Conj_{\breve{\scriptscriptstyle \Omega}}\left[\psi_{\scriptscriptstyle\Omega}^\mathbb{S}(K)\right]=\psi_{\scriptscriptstyle\Omega}^\mathbb{S}(K)$. By similar method in the proof of Proposition \ref{pr-increasing subsets}, $K':=\left[\psi_{\scriptscriptstyle\Omega}^\mathbb{S}(K)\right]^\land_{\breve{\scriptscriptstyle\Omega}}$ is $\breve{\Omega}$-invariant and $K'':=\varphi_{\breve{\scriptscriptstyle\Omega}}^{r^+(K')}(K')$ is $\Omega$-invariant with $K'=\psi_{\scriptscriptstyle\Omega}^\mathbb{S}(K'')$. Next we show that $(K'')_I$ is compact. Let $I\in\mathbb{S}$ and $\{w_\ell^I\}_{\ell\in\mathbb{N}}$ be a Cauchy sequence in $K''_I$ with
	\begin{equation*}
		w^I:=\lim_{\ell\rightarrow +\infty} w_\ell^I.
	\end{equation*}
	Since $K''=\varphi_{\breve{\scriptscriptstyle\Omega}}^{r^+(K')}(K')\subset\Omega^{r^+(K')}$, $w^I\in\Omega^{r^+(K')}\subset\Omega_I$.
	Then $\{(w_\ell^I,I)\}_{\ell\in\mathbb{N}}$ is a Cauchy sequence in $K''_I\times\{I\}$. By Proposition \ref{pr-homeomorphism}, $\{[(w_\ell^I,I)]_{\sim_\Omega}=\imath_\Omega^I((w_\ell^I,I))\}$ is a Cauchy sequence in $\breve{\Omega}$ and
	\begin{equation*}
		\lim_{\ell\rightarrow +\infty} [(w_\ell^I,I)]_{\sim_\Omega}=\imath_\Omega^I\left(\lim_{\ell\rightarrow +\infty}(w_\ell^I,I)\right)=\imath_\Omega^I((w^I,I))=[(w^I,I)]_{\sim_\Omega}.
	\end{equation*}
	Since $[(w_\ell^I,I)]_{\sim_\Omega}\in\psi_{\scriptscriptstyle\Omega}^{\mathbb{S}}(\{w_\ell^I\})\subset \psi_{\scriptscriptstyle\Omega}^{\mathbb{S}}(K'')=K'$ and $K'$ is compact, $[(w^I,I)]_{\sim_\Omega}\in K'$. It implies that
	\begin{equation*}
		[(w^I,I)]_{\sim_\Omega}\in\psi_{\scriptscriptstyle\Omega}^\mathbb{S}(\{w^I\})\cap K'\neq\varnothing.
	\end{equation*}
	By \eqref{eq-definitions}, $w^I\in\varphi_{\breve{\scriptscriptstyle\Omega}}^{r^+(K')}(K')=K''$. It implies that $w^I\in (K'')_I$ and $(K'')_I$ is compact.
	
	Let $p\in\Omega\backslash K''$. If $p\in\widehat{K}_\Omega$, then there is $C_p>0$ such that
	\begin{equation}\label{eq-flcp}
		|f(p)|\le C_p\sup_{q\in K}|f(q)|,\qquad\forall\ f\in\mathcal{SR}(\Omega).
	\end{equation}
	However, by Proposition \ref{pr-l2psia}, there is $h\in\mathcal{SR}(\Omega)$ such that
	\begin{equation*}
		h(p)=1,\qquad\mbox{and}\qquad \sup_{q\in K} |h(q)|\le \sup_{q\in K''} |h(q)|<\frac{1}{C_p+1},
	\end{equation*}
	a contradiction to \eqref{eq-flcp}. Hence $\widehat{K}_\Omega\subset K''$ and then $(\widehat{K}_\Omega)_I\subset (K'')_I\Subset \Omega_I$, for each $I\in\mathbb{S}$. It implies that $K$ is slice regularly convex.
\end{proof}

\begin{defn}\label{dommex}
	Let $I\in\mathbb{S}$, $U\in\tau(\mathbb{C}_I)$ and $\mathcal{F}\subset\mathscr{O}_I(U)$. $U$ is called an \textit{\textbf{$\mathcal{F}$-domain of existence}}\index{$\mathcal{F}$-domain of slice existence} if there are no $U_1,U_2\in\tau(\mathbb{C}_I)$ with the following properties:
	\begin{enumerate}[label=(\roman*)]
		\item $\varnothing\neq U_1\subset U_2\cap U$.
		\item $U_2$ is connected in $\tau(\mathbb{C}_I)$ and $U_2\nsubseteq U$.
		\item For each $f\in\mathcal{F}$, there is $\widetilde{f}\in\mathcal{SR}(U_2)$ such that $f=\widetilde{f}$ in $U_1$.
	\end{enumerate}
	
	If there are sets $U,U_1,U_2\in\tau(\mathbb{C}_I)$ satisfying (i)-(iii), then we call $(U,U_1,U_2)$ a \textit{\textbf{$\mathcal{F}$-triple}}\index{Slice-triple}. Moreover, we say that  $(U,U_1,U_2)$ is \textit{\textbf{special}}\index{Special slice-triple}, if $U_1$ is a connected component of $U\cap U_2$ in $\mathbb{C}_I$.
\end{defn}

We now have all the needed terminology and arguments to prove the main result of this paper, namely Theorem \ref{thm-ct}, that is repeated here for the reader's convenience:
\begin{thm} (Cartan-Thullen)
	Let $\Omega\in\tau_s(\mathbb{H})$. Then the following statements are equivalent:
	\begin{enumerate}[label=(\roman*)]
		\item $\Omega$ is the domain of existence of some slice regular function.
		\item $\Omega$ is a domain of slice regularity.
		\item $\Omega$ is slice regularly convex.
		\item $\Omega$ is $2$-path-symmetric.
	\end{enumerate}
\end{thm}

\begin{proof}
	(iii)$\Leftrightarrow$(iv)  follows directly by Proposition \ref{eq-oits}.
	
	(i)$\Rightarrow$(ii) follows by definition.
	
	(ii)$\Rightarrow$(iv) Directly by \cite[Proposition 9.7]{Dou2020001}.
	
	(iv)$\Rightarrow$(i) Let $\Omega$ be $2$-path-symmetric. We only need to prove that every slice-connected component of $\Omega$ is a domain of existence of some slice regular function;  without loss of generality, we can assume that $\Omega$ is slice-connected.
	
	Let $J_0\in\mathbb{S}\backslash\mathbb{S}^{2ps}_{\Omega}$ and consider
	\begin{equation*}
		A_{[I]}:=\{q\in\Omega_I: q=x+yI,\mbox{ for some }x,y\in\mathbb{Q}\},\qquad\forall\ I\in\mathbb{S}
	\end{equation*}
	and
	\begin{equation*}
		A:=\bigcup_{I\in\mathbb{S}_{\Omega}^{2ps}\cup\{J_0\}} A_{[I]}\times\{I\}.
	\end{equation*}
	By Proposition \ref{pr-osmba}, $\mathbb{S}_{\Omega}^{2ps}$ is countable, and so is $A$. Let $\vartheta_A:A\rightarrow\mathbb{N}$ be a bijection and $(a_\imath, I_\imath):=\vartheta_A^{-1}(\imath)$.
	
	Let
	\begin{equation}\label{eq-bi def}
		b_\imath\in B_{I_\imath}(a_\imath,\sigma(a,\mathbb{H}\backslash\Omega))
	\end{equation}
	with $\sigma(b_\imath,\mathbb{H}\backslash\Omega)<\frac{1}{\imath+1}$ and $[b_\imath]_{\approx_\Omega}\neq [b_\jmath]_{\approx_\Omega}$ for $\jmath<\imath$. Set
	\begin{equation}\label{ef-bimath def}
		B_\imath:=\begin{cases}
			[b_\imath]_{\approx_\Omega},\qquad & \mathbb{S}_{[(b_\imath,I_\imath)]_{\sim_\Omega}}^+=\mathbb{S},
			\\\{b_\imath\},\qquad & \mbox{otherwise},
		\end{cases}
	\end{equation}
and
$$
B:=\bigcup_{\imath\in\mathbb{N}} B_\imath.
$$

According to Proposition \ref{pr-omegar}, $B_\ell\cap\Omega^{\frac{1}{\imath+1}}=\varnothing$, for all $\ell\ge \imath$. It implies that
	\begin{equation*}
		B\cap\Omega^{\frac{1}{\imath+1}}\subset\bigcup_{\ell=0,1,...,\imath-1}[b_\ell]_{\approx_\Omega}.
	\end{equation*}
	It follows from \eqref{eq-sqmbo}, \eqref{eq-suh} and Proposition \ref{pr-losmt} \ref{it-lp} that
	\begin{equation*}
		\psi_{\scriptscriptstyle\Omega}^{\mathbb{S}}\left(B\cap\Omega^{\frac{1}{\imath+1}}\right)
		\subset\bigcup_{\ell=0,1,...,\imath-1}\bigg\{[(b_\ell,I_\ell)]_{\sim_\Omega},[(b_\ell,-I_\ell)]_{\sim_\Omega}\bigg\}.
	\end{equation*}
	Then $\left|\psi_{\scriptscriptstyle\Omega}^{\mathbb{S}}\left(B\cap\Omega^{\frac{1}{\imath+1}}\right)\right|\le 2\imath$, i.e. the cardinality of $\psi_{\scriptscriptstyle\Omega}^{\mathbb{S}}\left(B\cap\Omega^{\frac{1}{\imath+1}}\right)$ is finite and has no accumulation point in $\breve{\Omega}$. It implies that $B$ is stem-discrete. Define $g:B\rightarrow\mathbb{H}$ by
	\begin{equation*}
		g(q)=\imath,\qquad q\in[b_\imath]_{\approx_\Omega}.
	\end{equation*}
	By Interpolation Theorem \ref{th-it}, there is $f\in\mathcal{SR}(\Omega)$ such that $f|_B=g$.
	
	By the fact in complex analysis, for each $I\in\mathbb{S}_{\Omega}^{2ps}$, $f_I$ can not extend to a holomorphic function near any point of $\partial_I(\Omega_I)$, i.e.
	\begin{equation}\label{eq-domainofexist}
		\Omega_I\mbox{ is the domain of existence of }f_I,\qquad\mbox{ for each }I\in\mathbb{S}_{\Omega}^{2ps}.
	\end{equation}
	
	Suppose that there is $J\in\mathbb{S}\backslash\mathbb{S}_{\Omega}^{2ps}$ such that $\Omega_J$ is not the domain of existence of $f_J$. Then there is a special $\{f_J\}$-triple $(\Omega_J,U_1,U_2)$ and a $J$-holomorphic functions $h:U_2\rightarrow\mathbb{H}$ such that $h|_{U_1}=f_J|_{U_1}$.
	
	Let $z^J\in\partial_J(\Omega_J)\cap\partial_J(U_1)\cap U_2$. If $z^I\in\Omega$ for some $I\in\mathbb{S}$, then it follows from $z^J\notin\Omega$ that $z\notin\Omega_{2ps}$. Therefore $I\in\mathbb{S}_{\Omega}^{2ps}$ and $z^{-I}\in\partial_I(\Omega_I)$. By representation formula with similar method in the proof of \cite[Proposition 9.7]{Dou2020001}, $f_I$ can extend to a $I$-holomorphic function near $z^{-I}\in\partial_I(\Omega_I)$, a contradiction to \eqref{eq-domainofexist}.
	
	Otherwise, $z^I\notin\Omega$ for all $I\in\mathbb{S}$. Then there is $r>0$ such that $B_J(z^J,2r)\subset U_2$. Denote
	\begin{equation*}
		W:=\Psi_J^i(U_1\cap B_J(z^J,r)).
	\end{equation*}
	Let $w\in W$ and $L\in\mathbb{S}$. Since $|w^L-z^L|<r$ and $z^L\in\partial_L(\Omega_L)$, we have $\sigma(w^L,\mathbb{H}\backslash\Omega)<r$. It follows from
	\begin{equation*}
		B_L(w^L,\sigma(w^L,\mathbb{H}\backslash \Omega))\subset B_L(z^J,2r)\cap\Omega_L\subset \Psi_J^L(U_2)\cap\Omega_L=\Psi_J^L(U_2\cap\Omega),
	\end{equation*}
	the connectedness of $B_L(w^L,\sigma(w^L,\mathbb{H}\backslash \Omega))$ and $(\Omega_J,U_1,U_2)$ being special that
	\begin{equation}\label{eq-in U1}
		B_L(w^L,\sigma(w^L,\mathbb{H}\backslash \Omega))\subset \Psi_J^L(U_1).
	\end{equation}
	It implies from $U_1\subset\Omega_J\subset\Omega^{2ps}\subset\Omega$ that
	\begin{equation*}
		\Psi_J^I(U_1)\subset\Omega,\qquad\mbox{and}\qquad \sigma(w^L,\mathbb{H}\backslash \Omega)=\sigma(w^L,\Psi_J^L(U_1))=\sigma(W^J,U_1).
	\end{equation*}
	Therefore,
	\begin{equation}\label{eq-splus=s}
		\mathbb{S}_{[(w^L,L)]_{\sim_\Omega}}^+=\mathbb{S},\qquad\forall\ w\in W\mbox{ and } L\in\mathbb{S}.
	\end{equation}
	
	Denote
	\begin{equation*}
		W_*:=\Psi_J^i\left(U_1\cap B_J\left(z^J,\frac{r}{2}\right)\right)
	\end{equation*}
	Let $w_1\in W_*$ and $L_1\in\mathbb{S}$. By $\sigma(w_1^{L_1},\mathbb{H}\backslash\Omega)\le |w_1^{L_1}-z^{L_1}|<\frac{r}{2}$, we have
	\begin{equation*}
		B_{L_1}\left(w_1^{L_1},\sigma\left(w_1^{L_1},\mathbb{H}\backslash\Omega\right)\right)\subset B_{L_1}(z^{L_1},r).
	\end{equation*}
	According to \eqref{eq-in U1},
	\begin{equation}\label{eq-bw_1in ws}
		B_{L_1}\left(w_1^{L_1},\sigma\left(w_1^{L_1},\mathbb{H}\backslash\Omega\right)\right)\subset W^\mathbb{S}.
	\end{equation}
	
	Let $J_1\in\mathbb{S}\backslash\mathbb{S}_{\Omega}^{2ps}$. Then there is $\{c_\ell\}_{\ell\in\mathbb{N}}$ with $c_\ell\in A_{[J_1]}\cap W_*^{J_1}$ such that
	\begin{equation*}
		\lim_{\ell\rightarrow +\infty} c_\ell=z^{J_1}.
	\end{equation*}
	Since $c_\ell\in\Omega_{J_1}\subset\Omega^{2ps}$,
	\begin{equation*}
		\Psi_{J_1}^{J_0}(c_\ell)\in A_{[J_0]}\subset A.
	\end{equation*}
	Then there is $\{v_\ell\}_{\ell\in\mathbb{N}}$ with $v_\ell\in\mathbb{N}$ such that $a_{v_\ell}=\Psi_{J_1}^{J_0}(c_\ell)$. It follows from
	\begin{equation*}
		a_{v_\ell}=\Psi_{J_1}^{J_0}(c_\ell)\in W_*^\mathbb{S},
	\end{equation*}
	\eqref{eq-bi def} and \eqref{eq-bw_1in ws} that
	\begin{equation*}
		b_{v_\ell}\in B\left(a_{v_\ell},\sigma\left(a_{v_\ell},\mathbb{H}\backslash\Omega\right)\right)\subset W^\mathbb{S}.
	\end{equation*}
	According to \eqref{eq-splus=s} and \eqref{ef-bimath def}, $B_{v_\ell}=[b_{v_\ell}]_{\approx_\Omega}$. Since $b_{v_\ell}\in\Omega_{J_0}\subset\Omega^{2ps}$,
	\begin{equation*}
		d_{v_\ell}:=\Psi_{J_0}^{J_1}(b_{v_\ell})\in B_{v_\ell}.
	\end{equation*}
	Then
	\begin{equation*}
		\lim_{\imath\rightarrow +\infty}d_{v_\imath}
		=\lim_{\imath\rightarrow +\infty}\left(\Psi_{J_1}^{J_0}(b_{v_\imath})\right)
		=\Psi_{J_1}^{J_0}\lim_{\imath\rightarrow +\infty}b_{v_\imath}
		=\Psi_{J_1}^{J_0}\lim_{\imath\rightarrow +\infty}a_{v_\imath}
		=\lim_{\imath\rightarrow +\infty}c_{v_\imath}=z^{J_1}.
	\end{equation*}
	By definition,
	\begin{equation*}
		\begin{split}
			h(z^J)
			=&\lim_{\imath\rightarrow +\infty}h(\Psi_I^J(d_{v_\imath}))
			=\lim_{\imath\rightarrow +\infty} f(\Psi_I^J(d_{v_\imath}))
			\\=&\lim_{\imath\rightarrow +\infty} g(\Psi_I^J(d_{v_\imath}))
			=\lim_{\imath\rightarrow +\infty} v_\imath
			=+\infty,
		\end{split}
	\end{equation*}
	a contradiction. Therefore $\Omega_J$ is the domain of existence of $f_J$.
	
	In summary, $\Omega_I$ is the domain of existence of $f_I$ for each $I\in\mathbb{S}$. Following the method in the proof of \cite[Proposition 9.4]{Dou2020001}, $\Omega$ is the domain of existence of $f$.
\end{proof}

{\bf Conflict of Interest}. The authors declare that they have no competing interests regarding the publication of this paper.

{\bf Availability of data}. There are no data associated with the research in this paper.

\bibliographystyle{amsplain}
\bibliography{mybibfile}

\begin{bibdiv}
\begin{biblist}

\bib{MR3320522}{article}{
      author={Alpay, Daniel},
      author={Bolotnikov, Vladimir},
      author={Colombo, Fabrizio},
      author={Sabadini, Irene},
       title={Self-mappings of the quaternionic unit ball: multiplier
  properties, the {S}chwarz-{P}ick inequality, and the {N}evanlinna-{P}ick
  interpolation problem},
        date={2015},
        ISSN={0022-2518,1943-5258},
     journal={Indiana Univ. Math. J.},
      volume={64},
      number={1},
       pages={151\ndash 180},
         url={https://doi.org/10.1512/iumj.2015.64.5456},
      review={\MR{3320522}},
}

\bib{Alpay2015001}{article}{
      author={Alpay, Daniel},
      author={Colombo, Fabrizio},
      author={Gantner, Jonathan},
      author={Sabadini, Irene},
       title={A new resolvent equation for the {$S$}-functional calculus},
        date={2015},
        ISSN={1050-6926,1559-002X},
     journal={J. Geom. Anal.},
      volume={25},
      number={3},
       pages={1939\ndash 1968},
         url={https://doi.org/10.1007/s12220-014-9499-9},
      review={\MR{3358079}},
}

\bib{Alpay2013001}{article}{
      author={Alpay, Daniel},
      author={Colombo, Fabrizio},
      author={Sabadini, Irene},
       title={Pontryagin-de {B}ranges-{R}ovnyak spaces of slice
  hyperholomorphic functions},
        date={2013},
        ISSN={0021-7670},
     journal={J. Anal. Math.},
      volume={121},
       pages={87\ndash 125},
         url={https://doi.org/10.1007/s11854-013-0028-8},
      review={\MR{3127378}},
}

\bib{Alpay2016001B}{book}{
      author={Alpay, Daniel},
      author={Colombo, Fabrizio},
      author={Sabadini, Irene},
       title={Slice hyperholomorphic {S}chur analysis},
      series={Operator Theory: Advances and Applications},
   publisher={Birkh\"auser/Springer, Cham},
        date={2016},
      volume={256},
        ISBN={978-3-319-42513-9; 978-3-319-42514-6},
         url={https://doi.org/10.1007/978-3-319-42514-6},
      review={\MR{3585855}},
}

\bib{MR4663899}{article}{
      author={Bolotnikov, Vladimir},
       title={On the intersection of left and right ideals in the ring of
  quaternion polynomials},
        date={2023},
        ISSN={0219-4988,1793-6829},
     journal={J. Algebra Appl.},
      volume={22},
      number={12},
       pages={Paper No. 2350249, 23},
         url={https://doi.org/10.1142/S0219498823502493},
      review={\MR{4663899}},
}

\bib{Cartan1932001}{article}{
      author={Cartan, Henri},
      author={Thullen, Peter},
       title={Zur {T}heorie der {S}ingularit\"{a}ten der {F}unktionen mehrerer
  komplexen {V}er\"{a}nderlichen},
        date={1932},
        ISSN={0025-5831,1432-1807},
     journal={Math. Ann.},
      volume={106},
      number={1},
       pages={617\ndash 647},
         url={https://doi.org/10.1007/BF01455905},
      review={\MR{1512777}},
}

\bib{MR3967697}{book}{
      author={Colombo, Fabrizio},
      author={Gantner, Jonathan},
       title={Quaternionic closed operators, fractional powers and fractional
  diffusion processes},
      series={Operator Theory: Advances and Applications},
   publisher={Birkh\"{a}user/Springer, Cham},
        date={2019},
      volume={274},
        ISBN={978-3-030-16408-9; 978-3-030-16409-6},
         url={https://doi.org/10.1007/978-3-030-16409-6},
      review={\MR{3967697}},
}

\bib{MR3887616}{book}{
      author={Colombo, Fabrizio},
      author={Gantner, Jonathan},
      author={Kimsey, David~P.},
       title={Spectral theory on the {S}-spectrum for quaternionic operators},
      series={Operator Theory: Advances and Applications},
   publisher={Birkh\"{a}user/Springer, Cham},
        date={2018},
      volume={270},
        ISBN={978-3-030-03073-5; 978-3-030-03074-2},
      review={\MR{3887616}},
}

\bib{Colombo2022001}{article}{
      author={Colombo, Fabrizio},
      author={Gantner, Jonathan},
      author={Kimsey, David~P.},
      author={Sabadini, Irene},
       title={Universality property of the {$S$}-functional calculus,
  noncommuting matrix variables and {C}lifford operators},
        date={2022},
        ISSN={0001-8708,1090-2082},
     journal={Adv. Math.},
      volume={410},
       pages={Paper No. 108719, 39},
         url={https://doi.org/10.1016/j.aim.2022.108719},
      review={\MR{4496722}},
}

\bib{Colombo2009001}{article}{
      author={Colombo, Fabrizio},
      author={Gentili, Graziano},
      author={Sabadini, Irene},
      author={Struppa, Daniele},
       title={Extension results for slice regular functions of a quaternionic
  variable},
        date={2009},
        ISSN={0001-8708},
     journal={Adv. Math.},
      volume={222},
      number={5},
       pages={1793\ndash 1808},
         url={https://doi.org/10.1016/j.aim.2009.06.015},
      review={\MR{2555912}},
}

\bib{Colombo2015001}{article}{
      author={Colombo, Fabrizio},
      author={L\'avi\~vcka, Roman},
      author={Sabadini, Irene},
      author={Sou\~vcek, Vladimír},
       title={The {R}adon transform between monogenic and generalized slice
  monogenic functions},
        date={2015},
        ISSN={0025-5831},
     journal={Math. Ann.},
      volume={363},
      number={3-4},
       pages={733\ndash 752},
         url={https://doi.org/10.1007/s00208-015-1182-3},
      review={\MR{3412341}},
}

\bib{Colombo2011001}{article}{
      author={Colombo, Fabrizio},
      author={Sabadini, Irene},
       title={The quaternionic evolution operator},
        date={2011},
        ISSN={0001-8708,1090-2082},
     journal={Adv. Math.},
      volume={227},
      number={5},
       pages={1772\ndash 1805},
         url={https://doi.org/10.1016/j.aim.2011.04.001},
      review={\MR{2803786}},
}

\bib{Colombo2009002}{article}{
      author={Colombo, Fabrizio},
      author={Sabadini, Irene},
      author={Struppa, Daniele~C.},
       title={Slice monogenic functions},
        date={2009},
        ISSN={0021-2172},
     journal={Israel J. Math.},
      volume={171},
       pages={385\ndash 403},
         url={https://doi.org/10.1007/s11856-009-0055-4},
      review={\MR{2520116}},
}

\bib{Colombo2010003}{article}{
      author={Colombo, Fabrizio},
      author={Sabadini, Irene},
      author={Struppa, Daniele~C.},
       title={Duality theorems for slice hyperholomorphic functions},
        date={2010},
        ISSN={0075-4102},
     journal={J. Reine Angew. Math.},
      volume={645},
       pages={85\ndash 105},
      review={\MR{2673423}},
}

\bib{Colombo2011001B}{book}{
      author={Colombo, Fabrizio},
      author={Sabadini, Irene},
      author={Struppa, Daniele~C.},
       title={Noncommutative functional calculus},
      series={Progress in Mathematics},
   publisher={Birkh\"auser/Springer Basel AG, Basel},
        date={2011},
      volume={289},
        ISBN={978-3-0348-0109-6},
         url={https://doi.org/10.1007/978-3-0348-0110-2},
        note={Theory and applications of slice hyperholomorphic functions},
      review={\MR{2752913}},
}

\bib{bookentire}{book}{
      author={Colombo, Fabrizio},
      author={Sabadini, Irene},
      author={Struppa, Daniele~C.},
       title={Entire slice regular functions},
      series={SpringerBriefs in Mathematics},
   publisher={Springer, Cham},
        date={2016},
        ISBN={978-3-319-49264-3; 978-3-319-49265-0},
         url={https://doi.org/10.1007/978-3-319-49265-0},
      review={\MR{3585395}},
}

\bib{Cullen1965001}{article}{
      author={Cullen, Charles.~G.},
       title={An integral theorem for analytic intrinsic functions on
  quaternions},
        date={1965},
        ISSN={0012-7094},
     journal={Duke Math. J.},
      volume={32},
       pages={139\ndash 148},
         url={http://projecteuclid.org/euclid.dmj/1077375642},
      review={\MR{0173012}},
}

\bib{Demailly2012001}{book}{
      author={Demailly, Jean-Pierre},
       title={Complex analytic and differential geometry},
        date={2012},
  url={https://www-fourier.ujf-grenoble.fr/~demailly/manuscripts/agbook.pdf},
}

\bib{Dou2021001}{article}{
      author={Dou, Xinyuan},
      author={Jin, Ming},
      author={Ren, Guangbin},
      author={Sabadini, Irene},
       title={A new approach to slice analysis via slice topology},
        date={2021},
        ISSN={0188-7009},
     journal={Adv. Appl. Clifford Algebr.},
      volume={31},
      number={5},
       pages={Paper No. 67, 36},
         url={https://doi.org/10.1007/s00006-021-01170-3},
      review={\MR{4309797}},
}

\bib{Dou2020001}{article}{
      author={Dou, Xinyuan},
      author={Ren, Guangbin},
      author={Sabadini, Irene},
       title={Extension theorem and representation formula in
  non-axially-symmetric domains for slice regular functions},
        date={2023},
        ISSN={1435-9855,1435-9863},
     journal={J. Eur. Math. Soc. (JEMS)},
      volume={25},
      number={9},
       pages={3665\ndash 3694},
         url={https://doi.org/10.4171/jems/1260},
      review={\MR{4634680}},
}

\bib{Dou2023002}{article}{
      author={Dou, Xinyuan},
      author={Ren, Guangbin},
      author={Sabadini, Irene},
       title={A representation formula for slice regular functions over
  slice-cones in several variables},
        date={2023},
        ISSN={0373-3114,1618-1891},
     journal={Ann. Mat. Pura Appl. (4)},
      volume={202},
      number={5},
       pages={2421\ndash 2446},
         url={https://doi.org/10.1007/s10231-023-01325-y},
      review={\MR{4634271}},
}

\bib{Fritzsche2002001B}{book}{
      author={Fritzsche, Klaus},
      author={Grauert, Hans},
       title={From holomorphic functions to complex manifolds},
      series={Graduate Texts in Mathematics},
   publisher={Springer-Verlag, New York},
        date={2002},
      volume={213},
        ISBN={0-387-95395-7},
         url={https://doi.org/10.1007/978-1-4684-9273-6},
      review={\MR{1893803}},
}

\bib{MR3254690}{article}{
      author={Gal, Sorin~G.},
      author={Sabadini, Irene},
       title={Walsh equiconvergence theorems in the quaternionic setting},
        date={2014},
        ISSN={1747-6933,1747-6941},
     journal={Complex Var. Elliptic Equ.},
      volume={59},
      number={12},
       pages={1589\ndash 1607},
         url={https://doi.org/10.1080/17476933.2013.868448},
      review={\MR{3254690}},
}

\bib{Gantner2020001}{article}{
      author={Gantner, Jonathan},
       title={Operator theory on one-sided quaternion linear spaces: intrinsic
  {$S$}-functional calculus and spectral operators},
        date={2020},
        ISSN={0065-9266},
     journal={Mem. Amer. Math. Soc.},
      volume={267},
      number={1297},
       pages={iii+101},
         url={https://doi.org/10.1090/memo/1297},
      review={\MR{4195727}},
}

\bib{Gentili2014001}{article}{
      author={Gentili, Graziano},
      author={Salamon, Simon},
      author={Stoppato, Caterina},
       title={Twistor transforms of quaternionic functions and orthogonal
  complex structures},
        date={2014},
        ISSN={1435-9855},
     journal={J. Eur. Math. Soc. (JEMS)},
      volume={16},
      number={11},
       pages={2323\ndash 2353},
         url={https://doi.org/10.4171/JEMS/488},
      review={\MR{3283400}},
}

\bib{Gentili2012001}{article}{
      author={Gentili, Graziano},
      author={Stoppato, Caterina},
       title={Power series and analyticity over the quaternions},
        date={2012},
        ISSN={0025-5831},
     journal={Math. Ann.},
      volume={352},
      number={1},
       pages={113\ndash 131},
         url={https://doi.org/10.1007/s00208-010-0631-2},
      review={\MR{2885578}},
}

\bib{Gesto2021B}{article}{
      author={Gentili, Graziano},
      author={Stoppato, Caterina},
       title={Geometric function theory over quaternionic slice domains},
        date={2021},
        ISSN={0022-247X,1096-0813},
     journal={J. Math. Anal. Appl.},
      volume={495},
      number={2},
       pages={Paper No. 124780, 38},
         url={https://doi.org/10.1016/j.jmaa.2020.124780},
      review={\MR{4182982}},
}

\bib{GeSto2021}{article}{
      author={Gentili, Graziano},
      author={Stoppato, Caterina},
       title={A local representation formula for quaternionic slice regular
  functions},
        date={2021},
        ISSN={0002-9939,1088-6826},
     journal={Proc. Amer. Math. Soc.},
      volume={149},
      number={5},
       pages={2025\ndash 2034},
         url={https://doi.org/10.1090/proc/15339},
      review={\MR{4232195}},
}

\bib{Gentili2013001B}{book}{
      author={Gentili, Graziano},
      author={Stoppato, Caterina},
      author={Struppa, Daniele~C.},
       title={Regular functions of a quaternionic variable},
      series={Springer Monographs in Mathematics},
   publisher={Springer, Heidelberg},
        date={2013},
        ISBN={978-3-642-33870-0; 978-3-642-33871-7},
         url={https://doi.org/10.1007/978-3-642-33871-7},
      review={\MR{3013643}},
}

\bib{Gentili2007001}{article}{
      author={Gentili, Graziano},
      author={Struppa, Daniele~C.},
       title={A new theory of regular functions of a quaternionic variable},
        date={2007},
        ISSN={0001-8708},
     journal={Adv. Math.},
      volume={216},
      number={1},
       pages={279\ndash 301},
         url={https://doi.org/10.1016/j.aim.2007.05.010},
      review={\MR{2353257}},
}

\bib{Gentili2010001}{article}{
      author={Gentili, Graziano},
      author={Struppa, Daniele~C.},
       title={Regular functions on the space of {C}ayley numbers},
        date={2010},
        ISSN={0035-7596},
     journal={Rocky Mountain J. Math.},
      volume={40},
      number={1},
       pages={225\ndash 241},
         url={https://doi.org/10.1216/RMJ-2010-40-1-225},
      review={\MR{2607115}},
}

\bib{MR2836832}{article}{
      author={Gentili, Graziano},
      author={Vignozzi, Irene},
       title={The {W}eierstrass factorization theorem for slice regular
  functions over the quaternions},
        date={2011},
        ISSN={0232-704X,1572-9060},
     journal={Ann. Global Anal. Geom.},
      volume={40},
      number={4},
       pages={435\ndash 466},
         url={https://doi.org/10.1007/s10455-011-9266-0},
      review={\MR{2836832}},
}

\bib{Ghiloni2011001}{article}{
      author={Ghiloni, Riccardo},
      author={Perotti, Alessandro},
       title={Slice regular functions on real alternative algebras},
        date={2011},
        ISSN={0001-8708},
     journal={Adv. Math.},
      volume={226},
      number={2},
       pages={1662\ndash 1691},
         url={https://doi.org/10.1016/j.aim.2010.08.015},
      review={\MR{2737796}},
}

\bib{Hartogg1906001}{book}{
      author={Hartogs, Fritz},
       title={Einige folgerungen aus der cauchyschen integralformel bei
  funktionen mehrerer veränderlichen},
      series={Sitzungsberichte},
     address={München},
        date={1907},
      volume={1906,6},
         url={https://publikationen.badw.de/en/003388866},
}

\bib{krantz}{book}{
      author={Krantz, Steven~G.},
       title={Function theory of several complex variables},
   publisher={AMS Chelsea Publishing, Providence, RI},
        date={2001},
        ISBN={0-8218-2724-3},
         url={https://doi.org/10.1090/chel/340},
        note={Reprint of the 1992 edition},
      review={\MR{1846625}},
}

\bib{Noguchi2016001}{book}{
      author={Noguchi, Junjiro},
       title={Analytic function theory of several variables},
   publisher={Springer, Singapore},
        date={2016},
        ISBN={978-981-10-0289-2; 978-981-10-0291-5},
         url={https://doi.org/10.1007/978-981-10-0291-5},
        note={Elements of Oka's coherence},
      review={\MR{3526579}},
}

\bib{Stoppato2012001}{article}{
      author={Stoppato, Caterina},
       title={A new series expansion for slice regular functions},
        date={2012},
        ISSN={0001-8708},
     journal={Adv. Math.},
      volume={231},
      number={3-4},
       pages={1401\ndash 1416},
         url={https://doi.org/10.1016/j.aim.2012.05.023},
      review={\MR{2964609}},
}

\end{biblist}
\end{bibdiv}
\printindex

\end{document}